\documentclass[11pt,a4paper]{article}

\usepackage{a4wide}
\usepackage{graphicx}
\usepackage{float}
\usepackage{amssymb}
\usepackage{amsmath}
\usepackage{amsthm}
\usepackage{xcolor,colortbl}
\usepackage{mathrsfs}
\usepackage{array}
\usepackage{tikz}
\usetikzlibrary{patterns}
\usepackage{multido}
\usepackage[T1]{fontenc}
\usepackage{hhline}
\usepackage{inputenc}
\usepackage[english]{babel}
\usetikzlibrary{shapes.misc}
\tikzset{cross/.style={cross out, draw=blue, minimum size=1*(#1-\pgflinewidth), inner sep=0pt, outer sep=0pt},
cross/.default={6pt}}
\usepackage{lmodern}
\usepackage{geometry}
\usepackage{changepage}
\changepage{0pt}{}{}{}{}{0pt}{}{0pt}{10pt}
\usepackage[numbers]{natbib}
\newcommand{\Gammabar}{\ensuremath{{\overline{\Gamma}}}}
\newcommand{\dsphere}{\ensuremath{\mathbb S^{d-1}}}
\def\d{\partial}
\newcommand{\Omegabar}{\ensuremath{{\overline{\Omega}}}}
\newcommand{\vf}{\ensuremath{\varphi}}
\newcommand{\sphere}{\ensuremath{\mathbb S^{d-1}}}
\newcommand{\rel}{r}
\usepackage{hyperref}
\hypersetup{
pdfpagemode=none,
pdftoolbar=true,        
pdfmenubar=true,        
pdffitwindow=false,     
pdfstartview={Fit},    
pdftitle={Qualitative indicator functions for imaging crack networks using acoustic waves},    
pdfauthor={L. Audibert, L. Chesnel, H. Haddar, K. Napal},     
pdfsubject={},  
pdfcreator={L. Audibert, L. Chesnel, H. Haddar, K. Napal},   
pdfproducer={L. Audibert, L. Chesnel, H. Haddar, K. Napal}, 
pdfkeywords={}, 
pdfnewwindow=true,      
colorlinks=true,       
linkcolor=magenta,          
citecolor=red,        
filecolor=cyan,      
urlcolor=blue           
}

\newcommand{\dsp}{\displaystyle}
\newcommand{\eps}{\varepsilon}

\newcommand{\Om}{\Omega}
\newcommand{\mrm}[1]{\mathrm{#1}}
\newcommand{\Cplx}{\mathbb{C}}
\newcommand{\N}{\mathbb{N}}
\newcommand{\R}{\mathbb{R}}
\newcommand{\Z}{\mathbb{Z}}

\newcommand{\mL}{\mrm{L}}
\newcommand{\mH}{\mrm{H}}

\newcommand{\mY}{\mrm{Y}}
\newcommand{\loc}{\mbox{\scriptsize loc}}

\renewcommand{\d}{\partial}
\newtheorem{theorem}{Theorem}[section]
\newtheorem{lemma}[theorem]{Lemma}
\newtheorem{remark}[theorem]{Remark}
\newtheorem{definition}[theorem]{Definition}

\newtheorem{proposition}[theorem]{Proposition}

\begin{document}

~\vspace{-0.3cm}
\begin{center}
{\sc \bf\LARGE Qualitative indicator functions for imaging\\[6pt] crack networks using acoustic waves}
\end{center}

\begin{center}
\textsc{Lorenzo Audibert}$^1$, \textsc{Lucas Chesnel}$^{2}$, \textsc{Houssem Haddar}$^2$, \textsc{Kevish Napal}$^{3}$\\[16pt]
\begin{minipage}{0.92\textwidth}
{\small
$^1$ Department PRISME, EDF R\&D, 6 quai Watier, 78401, Chatou CEDEX, France; \\
$^2$ INRIA/Centre de math\'ematiques appliqu\'ees, \'Ecole Polytechnique, Institut Polytechnique de Paris, Route de Saclay, 91128 Palaiseau, France;\\
$^3$ Department of Civil, Environmental \& Architectural Engineering, University of Colorado Boulder, CO, USA (corresponding author)\\[10pt]
E-mails: \textit{Lorenzo.Audibert@edf.fr},\, \textit{Lucas.Chesnel@inria.fr}, \textit{Houssem.Haddar@inria.fr},\\ \textit{Kevish.Napal@colorado.edu}\\[-14pt]
\begin{center}
(\today)
\end{center}
}
\end{minipage}
\end{center}
\vspace{0.4cm}

\noindent\textbf{Abstract.} We consider the problem of imaging a crack network  embedded in some homogeneous background from measured multi-static far field data generated by acoustic plane waves. We propose two novel approaches that can be seen as extensions of linear sampling-type methods and that provide indicator functions which are sensitive to local cracks densities. The first approach uses multiple frequencies data to compute  spectral signatures associated with artificially embedded localized obstacles. The second approach also exploits the  idea of incorporating an artificial background but uses data for a single frequency. The indicator function is built using a similar concept as  for differential sampling methods: compare the solution of the interior transmission problem for healthy inclusion with the one with embedded cracks. The performance of the methods is tested and discussed on synthetic examples and the numerical results are compared with the ones obtained using the classical factorization method. 

\section{Introduction}

We consider the problem of identifying a set of cracks embedded in some homogeneous background from measured far field data generated by acoustic plane waves. This type of problem has been extensively studied in the literature and we refer to \cite{arc_sca,cakoni2003linear,kress2005hybrid,alessandrini2005determining,ben2005use,johansson2007reconstruction,ivanyshyn2008inverse,zeev2009identification,AGKPS2011,liu2010reconstruction,bonnet2011fast,waveguids,bellis515,crack_imp,GWY015} for an overview of some non-iterative methods that have been employed to solve this problem. We are interested here in cases where the cracks form a relatively dense network that may cause failure of most of these methods (due to important multiple scattering effects). In general, these methods would furnish an indicator function that only reveals a volumetric domain encompassing the cracks network. 
We propose and study two methods that would leverage this limitation and provide indicator functions that are sensitive to local densities of the cracks inside the network. These two methods belong to the family of so-called sampling methods (e.g. \cite{colton1988inverse,kirsch1999factorization}) and make use in particular of the  Generalized Linear Sampling Method (GLSM) \cite{GLSM,CCH}. In the case of inclusions with non empty interior, sampling methods fail if the interrogating frequency coincides with so called Transmission Eigenvalues (TEs) (that coincide with resonant frequencies in the case of non penetrable obstacles). Exploiting the failure of the these methods  at TEs, it has been shown in \cite{CCH_Determination_Dirichlet&TEs,cakoni2007use,cakoni2009computation} that TEs can be computed from far field data, and moreover that the knowledge of these quantities can be used to infer qualitative information on the material \cite{cakoni2008transmission,cakoni2007use,giorgi2012computing,haddar2004interior}. However, recovering sharp information from TEs is not an easy task for penetrable inclusions because they cannot be viewed as the spectrum of a selfadjoint operator. To overcome this difficulty, recent works \cite{Art, cakoni2016stekloff, audibert2017new, cogar2017modified} suggested to rather consider a modified spectrum that we refer to as Relative Transmission Eigenvalues (RTEs) and that can still be computed from far field data. The original proposed idea consists in the introduction of an artificial background that can be chosen by the observer. Then RTEs correspond to wavenumbers such that there exists an incident field for which the far field resulting from the effective background and the one resulting from the artificial background are arbitrarily close. The advantage of using RTEs is that they depend on the artificial background which can be fit as desired, in accordance with the problem under consideration. Hence choosing the appropriate setting for the artificial obstacle, such as the position, the geometry, whether the obstacle is penetrable or not and the corresponding refractive index or boundary conditions, can greatly simplify the link between the RTEs and the parameters of interest. 

In this work we shall adapt  the idea of RTE for crack monitoring perspectives.
First, we point out that cracks have an empty interior and therefore TEs do not exist (according to their definition for volumetric inclusions). However, one can still define RTEs as long as the artificially included background contains inclusions with non empty interior. This is what we would like to exploit here to build indicator functions that are sensitive to the number of cracks inside the artificial inclusion. 
Considering  the latter to be a non-penetrable obstacle $\Omega$ (with some prescribed boundary conditions), the RTEs then coincide with the eigenvalues of the Laplace operator with the boundary conditions on $\partial \Omega$ and the boundary conditions on the cracks intersecting $\Omega$. This indicates that these RTEs are sensitive to the number of cracks contained in $\Omega$. For instance, a monotonicity result can be easily shown for RTEs associated with new born cracks.  As for TEs, we prove that  RTEs can be determined from the multi-frequency data using an adaptation of the GLSM \cite{GLSM}. This leads to a first algorithm where the domain $\Omega$ is swept over the interrogated region and for each position, the computed RTEs  are compared with RTEs associated with a crack free inclusion $\Omega$. The desired resolution of the method is determined by the size of $\Omega$  which is in return limited  by the frequency band of the data (since it should contain at least the first RTE associated with $\Omega$).

A weak point of this first method is the relatively high numerical cost of the computations of RTEs associated to one artificial background. Since increasing the resolution (or equivalently, reducing the size of $\Omega$) requires to increase the number of considered artificial backgrounds, obtaining high resolution images may be prohibitive. To bypass this drawback, we suggest a second method which uses far field data at one fixed wavenumber. This alternative approach mixes the notion of artificial background with ideas from the Differential Linear Sampling Method \cite{DLSM}. Indeed we consider the same $\Omega$ as previously but instead of comparing RTEs, we compare the solutions of the corresponding transmission problems. We prove that a measure of the difference between these solutions can be obtained using the GLSM. This measure is indeed sensitive to the presence of the cracks but we are unable to obtain monotonicity results as for the first indicator function. The numerical  results however indicate that the obtained indicator function is sensitive to the crack numbers inside $\Omega$. Varying  $\Omega$ over the probed domain provide an image that reflects the density of the cracks.  

The numerical implementation of  both algorithms is presented in a simplified 2D setting. The performance of the two methods is tested and discussed on synthetic examples. Moreover, the numerical results are compared with the ones obtained using the Factorization Method (FM). We numerically observe the superiority of the second method for relatively sparse networks. When the crack network becomes  dense, only the first method provides an indicator function that shows variations with respect to local densities of the  cracks.  

The outline of the paper is as follows.  We first introduce in Section 2 the setting of the scalar inverse scattering problem for cracks. We also  outline the Factorization Method and show expected numerical results on a representative targeted configuration. Section 3 is dedicated to the proposed  imaging algorithms based on the notion of relative transmission eigenvalues. After defining these quantities,  we show how they can be determined from measured multi-static and multi-frequency data. We then describe the associated algorithm and give some preliminary numerical validations. We present in Section 4 the method based on ideas from differential linear sampling method that uses muti-static data at a fixed frequency. After analyzing the method, we  provide an extensive list of validating examples and make a comparison between the three inversion algorithms. Some technical results used in the proofs are given in an appendix.

\section{Setting of the problem}

\begin{figure}[!ht]
\centering
\begin{tikzpicture}[scale=1]
\draw[line width=0.1mm][dotted] (0.7,0) circle (0.5);
\draw[line width=1pt] (0.7,0.5) arc (90:199:0.5);
\draw (0.7,0) node{\textcolor{gray!90}{$D$}};
\draw (2.5,-0.3) node{\textcolor{gray!90}{$D$}};
\draw (0,0) node{$\Gamma$};
\begin{scope} [rotate around={-45:(0.7,0)}]
\draw [line width=1pt,->,gray!50] (0.2,0) -- (-0.3,0);
\end{scope}
\draw [line width=1pt,->,gray!50] (2.5,0.2) -- (2.5,0.7);
\draw (0.32,0.6) node{\textcolor{gray!50}{$\nu$}};
\draw (2.7,0.45) node{\textcolor{gray!50}{$\nu$}};
\draw[line width=0.1mm][dotted] (2.5,-0.3) ellipse  (1 and 0.5);
\draw[line width=1pt] (3.2,0.05) arc (45:135:1 and 0.5);

\begin{scope} [rotate around={50:(2.5,0)},xshift=0.5cm,yshift=1.3cm]
\draw[line width=0.1mm][dotted] (2.5,-0.3) ellipse  (1 and 0.5);
\draw[line width=1pt] (3.2,0.05) arc (45:135:1 and 0.5);
\draw (2.5,-0.3) node{\textcolor{gray!90}{$D$}};
\draw (2.7,0.45) node{\textcolor{gray!50}{$\nu$}};
\draw [line width=1pt,->,gray!50] (2.5,0.2) -- (2.5,0.7);
\end{scope}

\end{tikzpicture}
\caption{Example of setting in $\R^2$. \label{FigSetting}} 
\end{figure}
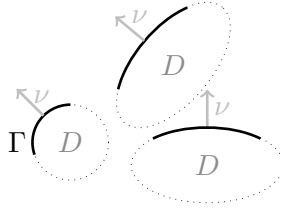

\noindent We are interested in the scattering of waves by a collection of sound hard cracks in $\R^d$, with $d=2$ or $3$. Following \cite{meca}, the cracks coincide with a portion $\Gamma$ of a smooth nonintersecting curve ($d = 2$) or surface ($d = 3$) that encloses a domain $D$ such that its boundary $\partial D$ is
smooth. We assume that $\Gamma$ is an open set with respect to the induced topology
on $\d D$. We denote by $\nu$ the unit normal vector to $\partial D$ pointing to the exterior of $D$ (see Figure \ref{FigSetting}). To define traces and normal derivatives of functions on $\Gamma$, we use the following notation for all $x \in\Gamma$:
\[
 f^\pm(x) = \lim\limits_{h \to 0^+}f(x\pm h\nu(x))\quad\text{ and }\quad \partial_\nu^\pm f(x) = \lim\limits_{h \to 0^+}\nu(x)\cdot\nabla f(x\pm h\nu(x)).
\]
Note that if $f$ is a function which is $\mH^2$ in an open set containing $\Gamma$, we have $\partial_\nu^+ f=\partial_\nu^- f$ on $\Gamma$ and we simply write $\partial_\nu f$. We shall also work with the jump functions
\begin{equation}\label{DefJumps}
[f] := f^+ - f^-  \quad\text{ and }\quad \left[\d_\nu f\right] := \d_\nu^+f - \d_\nu^-f.  
\end{equation}
Let $u_i$ be a smooth function satisfying $\Delta u_i+k^2u_i=0$ in $\R^d$ where $\Delta$ stands for the Laplace operator of $\R^d$ and where $k$ is the wavenumber proportional to the frequency. The scattering of $u_i$ by the cracks is described by the problem 
\begin{equation}\label{PbChampTotalFreeSpaceIntrod}
\begin{array}{|rcll}
\multicolumn{4}{|l}{\mbox{Find }u=u_i+u_s\mbox{ such that}}\\[2pt]
\Delta u+k^2 u & = & 0 & \mbox{ in }\R^d\setminus\overline{\Gamma}\\[5pt]
\partial^{\pm}_{\nu} u &=& 0 &\text{ on } \Gamma\\[5pt]
\multicolumn{4}{|c}{\dsp{\dsp\lim_{R\to +\infty}\int_{|x|=R}\left|\frac{\partial u_s}{\partial r}-ik u_s \right|^2 ds(x)  = 0.}}
\end{array}
\end{equation}
The last line of (\ref{PbChampTotalFreeSpaceIntrod}), in which $r=|x|$, is the so-called Sommerfeld radiation condition. For all $k^2>0$, Problem (\ref{PbChampTotalFreeSpaceIntrod}) has a unique solution $u$ belonging to $\mH^1_{\loc}(\R^d\setminus\overline{\Gamma})$. In this document, $\mH^1_{\loc}(\R^d\setminus\overline{\Gamma})$ denotes the set of functions which are in $\mH^1(\mathcal{O}\setminus\overline{\Gamma})$ for all bounded domain $\mathcal{O}\subset\R^d$ \cite{Colton-Kress}. Far from the cracks, the scattered field $u_s$ behaves at first order like a cylindrical wave (resp. spherical wave) when $d=2$ (resp. $d=3$) and we have the expansion (see e.g. \cite{Colton-Kress})
\begin{equation}\label{scatteredFieldFreeSpaceIntro}
u_{s}(r\theta_{s})=\dsp \eta_d\,e^{i k r}r^{-\frac{d-1}{2}}\,\Big(\,u_{s}^{\infty}(\theta_s)+O(1/r)\,\Big),
\end{equation}
as $r=|x|\to+\infty$, uniformly in $\theta_{s}\in \mathbb{S}^{d-1}:=\{\theta\in\R^d;\;|\theta|=1\}$. In \eqref{scatteredFieldFreeSpaceIntro} the constant $\eta_d$ is given by $\eta_d=e^{i \frac{\pi}{4}}/\sqrt{8\pi k}$ for $d=2$ and by $\eta_d=1/(4\pi)$ for $d=3$. The function $u_s^{\infty}: \mathbb{S}^{d-1}\to\Cplx$ is called the farfield pattern of $u_s$ associated with $u_{i}$. When $u_i$ coincides with the incident plane wave $u_{i}(\cdot,\theta_i):=e^{i k \theta_{i}\cdot x}$, of direction of propagation $\theta_{i}\in\mathbb{S}^{d-1}$, we denote respectively $u_s(\cdot,\theta_i)$, $u^{\infty}_s(\cdot,\theta_i)$ the corresponding scattered field and far field pattern.  
From the farfield pattern associated with one incident plane wave, by linearity we can define the farfield operator $F:\mL^2(\mathbb{S}^{d-1})\to\mL^2(\mathbb{S}^{d-1})$ such that 
\begin{equation}\label{defInitialFFoperator}
(Fg)(\theta_{s})=\int_{\mathbb{S}^{d-1}}g(\theta_{i})\,u_s^{\infty}(\theta_{s},\theta_{i})\,ds(\theta_{i}).
\end{equation}
The function $Fg$ corresponds to the farfield pattern of $u_s$ defined in (\ref{PbChampTotalFreeSpaceIntrod}) with 
\begin{equation}\label{HerglotzWave}
u_i=v_g:=\int_{\mathbb{S}^{d-1}}g(\theta_{i})e^{ik\theta_{i}\cdot x}\,ds(\theta_{i}).
\end{equation}
We call Herglotz wave functions the incident fields of the above form. Our goal is to develop techniques to identify the cracks $\Gamma$ from the knowledge of $F:\mL^2(\mathbb{S}^{d-1})\to \mL^2(\mathbb{S}^{d-1})$.\\
\newline
As mentioned in the introduction, sampling methods for crack monitoring purposes, such as the LSM or the FM, can already be found in the literature. However, these techniques turn out to be inefficient when it comes to image highly damaged backgrounds. We illustrate this point on numerical simulations where we implement the FM to recover sound-hard cracks coinciding with a set $\Gamma$. 
{The method relies on the solving of the problem
\begin{equation}\label{eq:FFeqCrack}
(F_\sharp)^{1/2}(g_L)(\hat{x})=\Phi_L^\infty(\hat{x})\quad\forall \hat{x}\in\dsphere.
\end{equation}
In equation  \eqref{eq:FFeqCrack}, the opertor $F_\sharp$ is defined by 
\begin{equation}\label{DefFsharp}
F_\sharp := |\Re e\,F|+\Im m\,F,
\end{equation}
with $\Re e\,F := (F+F^*)/2$ and $\Im m \,F := (F-F^*)/(2i)$. The right hand side  is defined by 
\[
\Phi_L^\infty(\hat{x}) = \int_L(-ik\hat{x}.\nu(y)\alpha(y) + \beta(y))e^{-ik\hat{x}.y}ds(y),
\]
where $\alpha$ and $\beta$ are given density functions. A chosen non intersecting curve $L$ is then part of the crack if and only if $\Phi_L^\infty\in R(F_\sharp^{1/2})$.} For more details on this method, we refer the reader to \cite{arc_sca,boukari2013factorization}. Figure \ref{fig:FM} illustrates the performance of the FM in a particular setting (we work with simulated data). We observe that the isolated cracks are relatively well recovered. However regions which concentrate too many cracks are poorly reconstructed. The use of higher wavenumber data should improve the results, as we try to point out in Figure \ref{fig:FMHRes}. But in practice one might not have a priori information on the damage level of the probed background nor access to high wavenumber data. In this work, we address this problematic by introducing new types of indicators which instead of giving the exact shape of cracks, focus on quantifying crack density.

\begin{figure}[!ht]
\centering
\includegraphics[height=5cm]{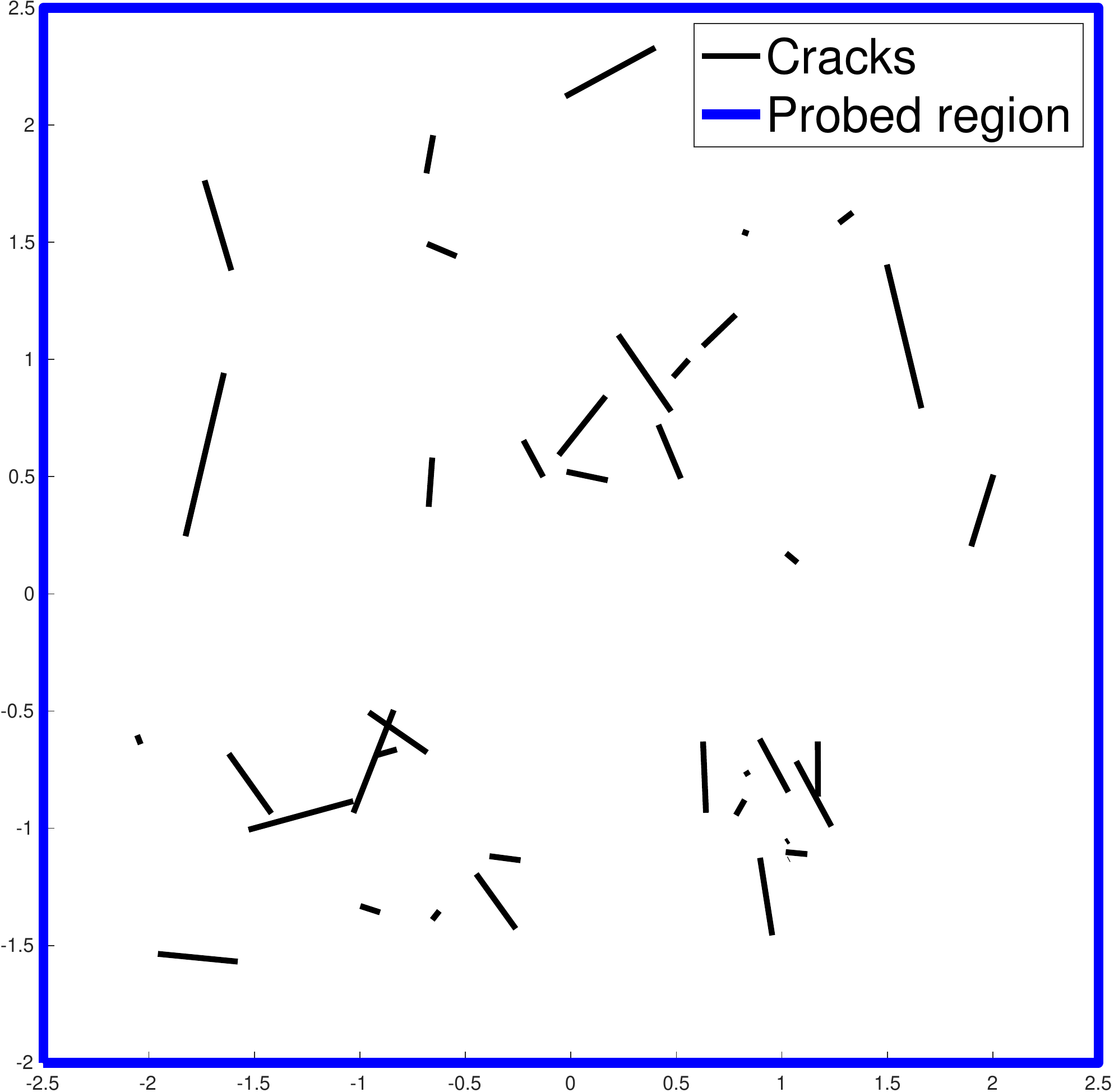}\qquad\qquad
\includegraphics[height=5cm]{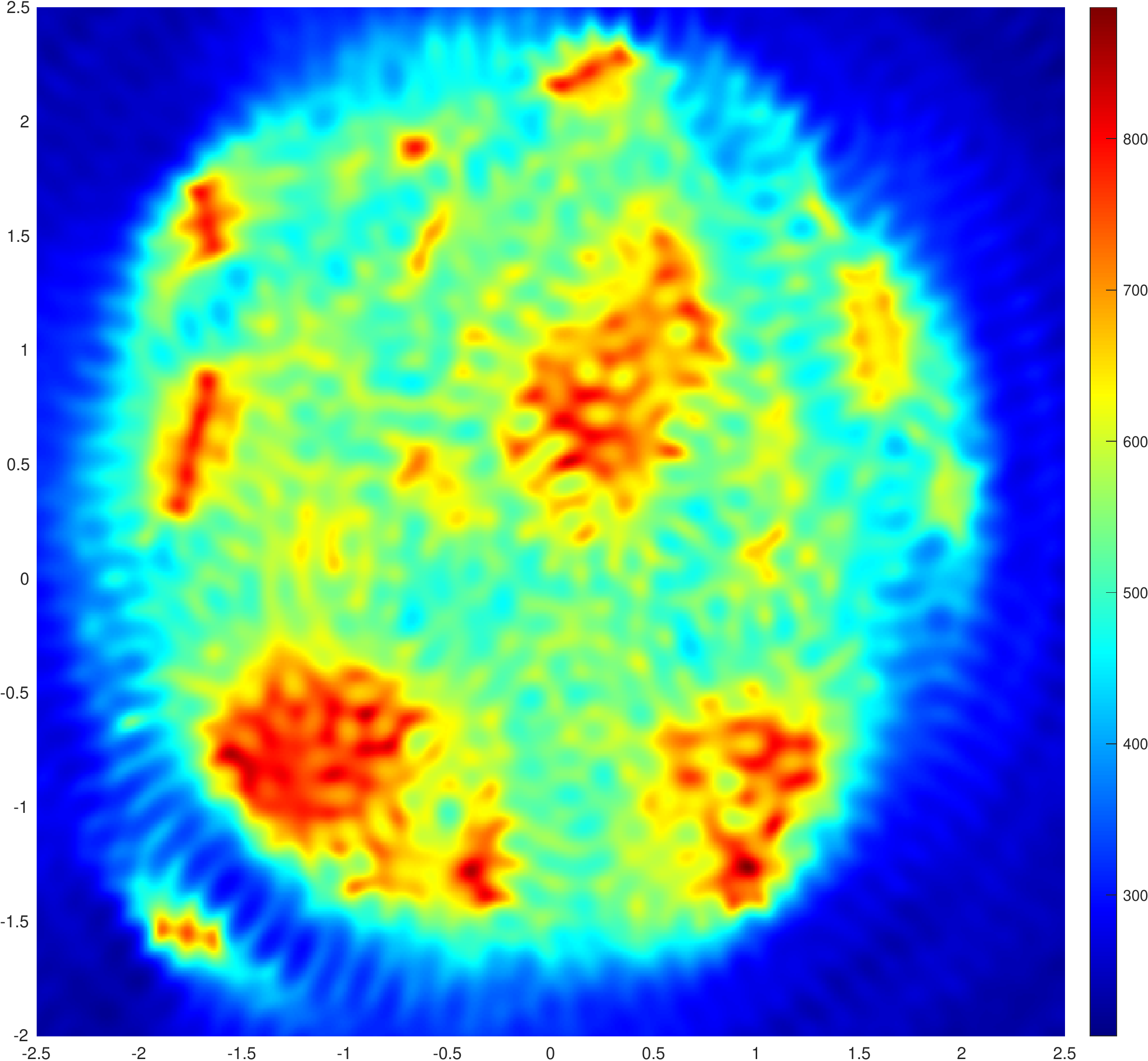}
\caption{Imaging of a simulated damaged background with the FM indicator. The reconstruction uses far field data generated at the wavelength $\lambda:=2\pi/k=0.3$. The resolution of the image is set by the sampling points step $dz = 0.01$.}
\label{fig:FM}
\end{figure}

\begin{figure}[!ht]
\centering
\includegraphics[height=5cm]{Images/CrackDistribution.pdf}\qquad\qquad
\includegraphics[height=5cm]{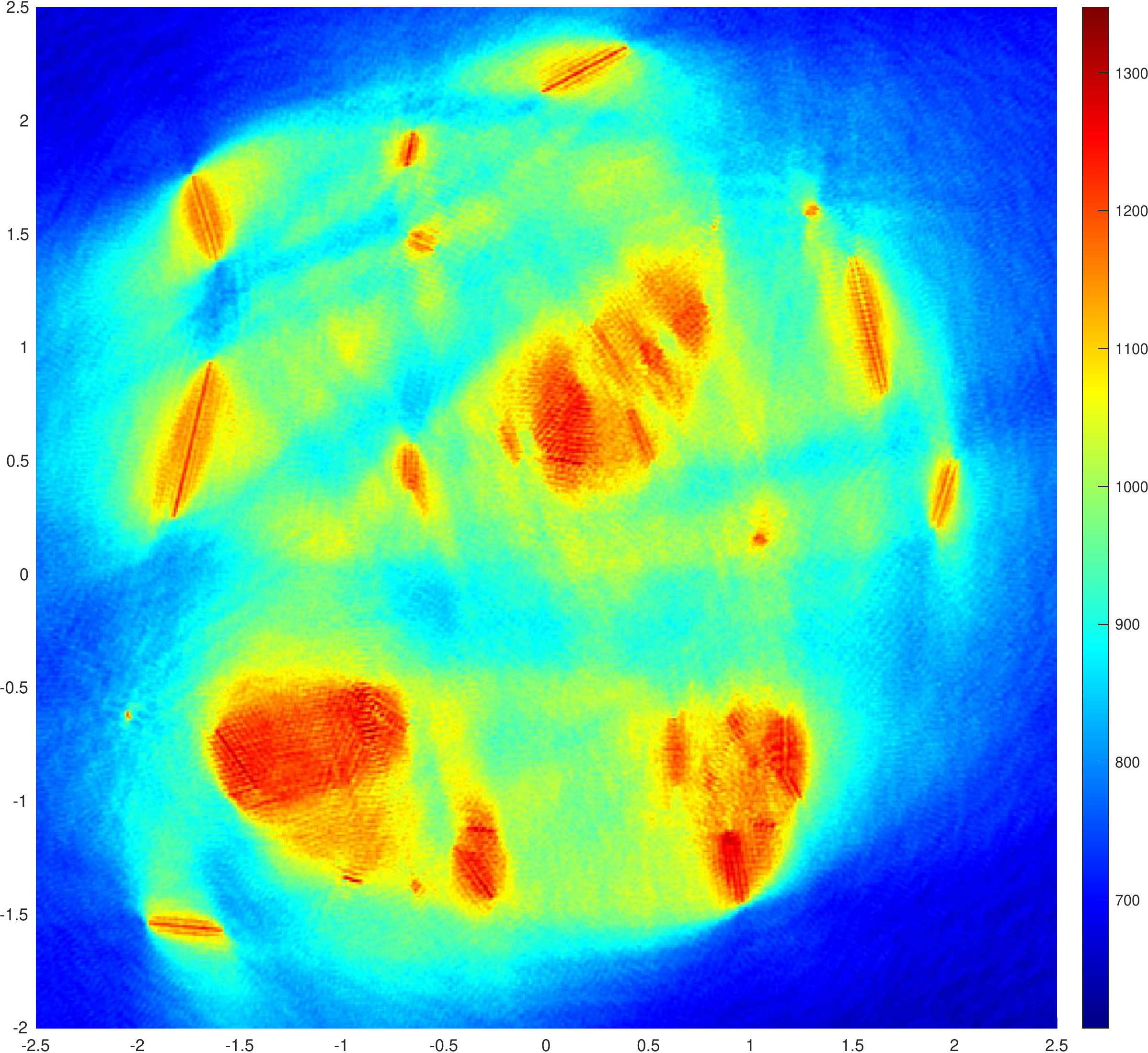}
\caption{Imaging of a simulated damaged background with the FM indicator. The reconstruction uses far field data generated at the wavelength $\lambda=0.05$. The resolution of the image is set by the sampling points step $dz = 0.01$.}
\label{fig:FMHRes}
\end{figure}

\noindent To proceed, we propose two methods relying on the notion of artificial backgrounds. The first one uses the concept of Relative Transmission Eigenvalues which has been introduced in \cite{Art,AuCH19}. It will be presented in the next section and will allow us to quantify small crack aggregates. This technique requires data for a range of frequencies and its implementation is quite expensive. To overcome this drawback, we will present in Section \ref{SectionSecondApproach} an alternative approach inspired by the Differential Linear Sampling Method of \cite{DLSM} which requires only data at one fixed wavenumber and fewer computations.\\
\newline
In our analysis, we will use several times a result on the factorization of the operator $F_\sharp$ defined in (\ref{DefFsharp}) that we recall here (which can be easily deduced from the analysis in \cite{arc_sca} and the abstract theory for the factorization method as in \cite{kirsch2008factorization,CCH}). First, define the Herglotz operator $\mathcal{H}_\Gamma : \mL^2(\dsphere) \to \mH^{-1/2}(\Gamma)$ such that 
\begin{equation}\label{DefHGamma}
\mathcal{H}_\Gamma g=  \partial_\nu  v_g|_{\Gamma},
\end{equation}
where $v_g$ appears in (\ref{HerglotzWave}).
\begin{proposition}\label{PropositionFactoTGamma}
Assume that $k^2$ is such that $\mathcal{H}_\Gamma : \mL^2(\dsphere) \to \mH^{-1/2}(\Gamma)$ is injective. Then we have the factorization 
\begin{equation}\label{FactoF}
F_\sharp=\mathcal{H}_\Gamma^{\ast}\,\mathcal{T}_\Gamma\,\mathcal{H}_\Gamma,
\end{equation}
where the bounded operator $\mathcal{T}_\Gamma:\mH^{-1/2}(\Gamma)\to \mH^{1/2}(\Gamma)$ is such that there is $c>0$ such that
\[
\forall\varphi\in\mH^{-1/2}(\Gamma),\qquad \langle \varphi,\mathcal{T}_\Gamma \varphi\rangle_{\Gamma} \ge c\,\|\varphi\|^2_{\mH^{-1/2}(\Gamma)}.
\]
Here $\langle\cdot,\cdot\rangle_{\Gamma}$ denotes the duality pairing in $\mH^{-1/2}(\Gamma)\times\mH^{1/2}(\Gamma)$. 
\end{proposition}

\begin{remark}
By assumption, $\Gamma$ coincides with a portion of the boundary of a smooth domain $D$. When $\partial D$ is analytic, from the results concerning zeros of analytic functions, we infer in particular  that $\mathcal{H}_\Gamma : \mL^2(\dsphere) \to \mH^{-1/2}(\Gamma)$ is injective when $k^2$ is a not a Neumann eigenvalue of the Laplace operator in $D$.  
\end{remark}
 
\section{Quantification of crack density using Relative Transmission Eigenvalues}

In the article \cite{Art}, it was proposed to work with relative far field data obtained by considering the difference between the measured far fields and some computed far fields corresponding to an artificial background. This was done in the context of the determination of the features of a penetrable obstacle via the use of the so-called transmission eigenvalues as spectral signatures. The motivation for introducing these relative far field data was to simplify the associated interior transmission problem in order to facilitate the reconstruction. In our setting, the use of artificial backgrounds is especially interesting because it will allow us to define a transmission problem which usually does not exist for scatterers of empty interior. After defining the relative far field operator, we will derive an associated boundary eigenvalue problem whose spectrum is shown to carry straightforward information on the cracks. Furthermore it will be shown that the spectrum can be computed from collected far field data at multiple frequencies. The possibility to quantify small crack aggregates with this spectrum will be illustrated by some numerical simulations.

\subsection{Relative Transmission Eigenvalues}

Let $\Omega\subset\R^d$ be a given bounded domain with a smooth boundary whose complement $\Om^c$ is connected. This domain, which is chosen by the observer, should be seen as an artificial impenetrable obstacle. To simplify the presentation of the theory, we shall assume the cracks do not meet the boundary of the artificial obstacle. More precisely, we assume that $\overline{\Gamma}\cap\partial\Om=\emptyset$. Note however that the cracks can be inside and/or outside $\Om$. Consider the scattering problem
\begin{equation}\label{PbChampTotalFreeSpaceComp}
\begin{array}{|rcll}
\multicolumn{4}{|l}{\mbox{Find }u_b=u_i+u_{b,s}\mbox{ such that}}\\[2pt]
\Delta u_b+k^2u_b & = & 0 & \mbox{ in }\R^d\setminus\overline{\Om},\\[3pt]
u_b & = & 0 & \mbox{ on }\partial\Om\\[3pt]
\multicolumn{4}{|l}{+\ \mbox{Radiation condition}.}
\end{array}
\end{equation}
Here the radiation condition is the same as in (\ref{PbChampTotalFreeSpaceIntrod}). For all $k^2>0$, Problem (\ref{PbChampTotalFreeSpaceComp}) has a unique solution $u_b$ in $\mH^1_{\loc}(\R^d\setminus\overline{\Om})$ \cite{Colton-Kress}. When $u_i$ coincides with the incident plane wave $u_{i}(\cdot,\theta_i)=e^{i k \theta_{i}\cdot x}$, of direction of propagation $\theta_{i}\in\mathbb{S}^{d-1}$, we denote respectively $u_{b,s}(\cdot,\theta_{i})$, $u^{\infty}_{b,s}(\cdot,\theta_{i})$ the corresponding scattered field and far field pattern (see the expansion in (\ref{scatteredFieldFreeSpaceIntro})). From the farfield pattern, we define by linearity the farfield operator $F^b:\mL^2(\mathbb{S}^{d-1})\to\mL^2(\mathbb{S}^{d-1})$ such that 
\[
(F^bg)(\theta_{s})=\int_{\mathbb{S}^{d-1}}g(\theta_{i})\,u^{\infty}_{b,s}(\theta_{s},\theta_{i})\,ds(\theta_{i}).
\]
Finally we define the relative farfield operator $F^{\rel}:\mL^2(\mathbb{S}^{d-1})\to\mL^2(\mathbb{S}^{d-1})$ as
\begin{equation}\label{def:RelFF}
F^{\rel} :=F-F^{b}.
\end{equation}
In practice, $F$ is obtained from the measurements while $F^{b}$ has to be computed by solving numerically \eqref{PbChampTotalFreeSpaceComp} for a given $\Om$.\\   
\newline
In what follows, we will be interested in the values of $k^2>0$ for which there is a non zero $u_i\in\mL^2_{\loc}(\R^d)$ satisfying $\Delta u_i+k^2u_i=0$ in $\R^d$ and such that the corresponding $u_{s}$, $u_{b,s}$ defined via (\ref{PbChampTotalFreeSpaceIntrod}), (\ref{PbChampTotalFreeSpaceComp}) are such that $u^{\infty}_{s}=u^{\infty}_{b,s}$. In other words, we are interesting in general incident fields such that the associated far field of the scattered fields for the cracks and for the artificial obstacle coincide. In such a case, from the Rellich lemma, this implies $u_{s}=u_{b,s}$ in $\R^d\setminus(\overline{\Om}\cup\Gamma)$ and so $u=u_i+u_s$ must coincide with $u_b=u_i+u_{b,s}$ in $\R^d\setminus(\overline{\Om}\cup\Gamma)$. Since $u_b=0$ on $\partial\Om$, we then have 
\begin{equation}\label{eq:RTEPb+AB}
\begin{array}{|rccl}
\Delta u + k^2 u &=&  0 & \text{in } \Omega \setminus\Gammabar\\[2pt]
\partial^{\pm}_{\nu} u &=& 0 &\text{on } \Gamma\cap\Om\\[2pt]
 u & = & 0 & \text{on } \d\Omega.
\end{array}
\end{equation}
\begin{definition}\label{DefinitionRTE}
We say that $k^2>0$ is a Relative Transmission Eigenvalue (RTE) if Problem (\ref{eq:RTEPb+AB}) admits a non zero solution $u\in\mH^1(\Omega \setminus\Gammabar)$.
\end{definition}
\noindent Let us make a first simple observation: if the chosen artificial obstacle $\Omega$ is such that $\Gammabar\cap\Omega=\emptyset$, then the RTEs coincide with the spectrum of the problem
\begin{equation}\label{eq:RTEPbNoCrack+AB}
\begin{array}{|rccl}
\Delta u + k^2 u &=&  0 & \mbox{ in } \Omega\\
u & = & 0 & \mbox{ on } \d\Omega 
\end{array}
\end{equation}
which can be computed independently from the data (it suffices to compute the eigenvalues of the Dirichlet Laplacian in $\Om$). It will be shown in the next paragraph that the spectrum of \eqref{eq:RTEPb+AB} can be computed from far field data. Consequently it is possible to determine whether $\Gammabar\cap\Omega$ is empty or not by comparing the spectra of \eqref{eq:RTEPb+AB} and \eqref{eq:RTEPbNoCrack+AB}. Using this result and sweeping the probed region with a collection of artificial obstacles $\Om$, we can then identify the cracks $\Gamma$. Note that with this approach, the resolution of the image is directly related to the size of $\Om$.\\
\newline
To improve the quality of the reconstruction, one can also use a result of monotonicity for the spectrum of (\ref{eq:RTEPbNoCrack+AB}) with respect to the size of the cracks that we describe now. Classical results concerning linear self-adjoint compact operators guarantee that the spectrum of (\ref{eq:RTEPb+AB}) is made of real positive isolated eigenvalues of finite multiplicity $0<k^2_1 \le k^2_2 \le \dots \le k^2_p \le\dots$ (the numbering is chosen so that each eigenvalue is repeated according to its multiplicity). Moreover, there holds $\lim_{p\to+\infty}k^2_p=+\infty$ and we have the \textit{min-max} formulas
\begin{equation}\label{FormuleMinMax}
k_p^2=\min_{E_p\in\mathscr{E}_p}\max_{u\in E_p\setminus\{0\}}\cfrac{\|\nabla u\|^2_{\mL^2(\Om)}}{\|u\|^2_{\mL^2(\Om)}}.
\end{equation}
Here $\mathscr{E}_p$ denotes the set of subspaces $E_p$ of $\{\varphi\in\mH^1(\Om\setminus\overline{\Gamma})\,|\,\varphi=0\mbox{ on }\partial\Om\}$ of dimension $p$. As a consequence, if $(\Gamma\cap\Om)\subset(\tilde{\Gamma}\cap\Om)$ and if we denote by $\tilde{k}^2_p$ the eigenvalues of (\ref{eq:RTEPb+AB}) with $\Gamma$ replaced by $\tilde{\Gamma}$, we obtain $k^2_p \ge \tilde{k}^2_p$ for all $p\in\N^{\ast}:=\{1,2,\dots\}$. Such a result is interesting to get an idea of how damaged is a material in presence of a dense network of cracks. In order to state our results below, we will have to exclude certain values of $k^2$, coinciding with a subset of RTEs. 

\begin{definition}\label{DefinitionRNSE}
Particular RTEs  $k^2>0$ for which one of the corresponding eigenfunctions is of the form $u=v_g$ in $\Omega \setminus\Gammabar$, where $v_g$ is a Herglotz wave defined by \eqref{HerglotzWave}, are called Relative Non Scattering Eigenvalues (RNSEs). 
\end{definition}
\begin{remark}
The condition $u=v_g$ in $\Omega \setminus\Gammabar$ means that the eigenfunction of (\ref{eq:RTEPb+AB}) can be extended to a Herglotz wave in the whole space. This is a rather restrictive condition and it has been shown that such $k^2$ do not exist in certain contexts close to the one considered here (see \cite{BlPS14,PaSV17,ElHu15,ElHu18,CaXi19,BlLX19}).
\end{remark}
\noindent Before proceeding, we state a factorization result of the operator $F^b$ that will be useful in the analysis. Its classical proof can be found in \cite{kirsch2008factorization}. Define the  operator $H_{\partial\Om} : \mL^2(\dsphere) \to \mH^{-1/2}(\partial\Om)$ such that 
\begin{equation}\label{DefHpartialOm}
H_{\partial\Om} g=  \partial_{\nu}u_b|_{\partial\Om},
\end{equation}
where $u_b$ is the total field defined in (\ref{HerglotzWave}) with $u_i=v_g$. Similarly to $F_\sharp$ in (\ref{DefFsharp}), we introduce
$$
F^b_\sharp := |\Re e\,F^b| + \Im m\,F^b.
$$ 
\begin{proposition}\label{PropositionFactoTpartialOm}
Assume that $k^2$ is not a Dirichlet eigenvalue of problem \eqref{eq:RTEPbNoCrack+AB}. Then we have the factorization
\[
F^b_\sharp=H_{\partial\Om}^{\ast}\,T_{\partial\Om}\,H_{\partial\Om},
\]
where the bounded operator $T_{\partial\Om}:\mH^{-1/2}(\partial\Om)\to \mH^{1/2}(\partial\Om)$ is such that there is $c>0$ such that
\[
\begin{array}{ll}
\forall\varphi\in\mH^{-1/2}(\partial\Om),
\qquad &\langle \varphi,T_{\partial\Om} \varphi\rangle_{\partial\Om} \ge c\,\|\varphi\|^2_{\mH^{-1/2}(\partial\Om)}.
\end{array}
\]
Here $\langle\cdot,\cdot\rangle_{\partial\Om}$ denotes the duality pairing in $\mH^{-1/2}(\partial\Om)\times\mH^{1/2}(\partial\Om)$. 
\end{proposition}

\subsection{Computation of the Relative Transmission Eigenvalues from the data}

In this paragraph, we explain how to compute RTEs from the data. The results justifying the method are the Theorems \ref{th:GLSM+AB} and \ref{th:RTEsComputation+AB} below. The idea is similar to the one used to characterize classical transmission eigenvalues from the knowledge of the far field operator $F$ (see \cite{CCH_Determination_Dirichlet&TEs}). We start by writing a factorization of the relative operator $F^{\rel}$. For $g\in \mL^2(\dsphere)$, consider the incident field $u_i=v_g$ as in (\ref{HerglotzWave}). Denote $u$, $u_{b}$ the total fields defined respectively via (\ref{PbChampTotalFreeSpaceIntrod}), (\ref{PbChampTotalFreeSpaceComp}) and set
\[
w := 
\left\{
\begin{array}{ ll}
u-u_b & \text{in } \R^d\setminus\Omegabar\\
u  & \text{in } \Omega.
\end{array}\right.
\]
With such a definition, we have $w^\infty = F^\rel g$ where $w^\infty$ stands for the far field pattern of $w$. The function $w$ solves the problem
\begin{equation}\label{eq:Def_G+AB}
\begin{array}{|rcllrcll}
\multicolumn{8}{|l}{\mbox{Find }w\in\mH^1_{\loc}(\R^d\setminus\overline{\Gamma})\mbox{ such that }}\\[3pt]
\Delta w+k^2\,w & = & 0 & \mbox{ in }\R^d\setminus(\partial\Omega\cup\overline{\Gamma}) & \\[3pt]
\partial^{\pm}_{\nu} w &=& 0 &\mbox{ on } \Gamma\cap\Omega & \qquad\qquad \left[ \partial_{\nu}w\right] &=& - \psi_1  & \mbox{ on } \d\Omega\\[3pt]
\partial^{\pm}_{\nu} w &=&   - \psi_2  & \mbox{ on } \Gamma\cap\Omega^c & \multicolumn{4}{l}{\qquad\qquad +\ \mbox{Radiation condition}},
\end{array}
\end{equation}
where $\psi_1 =  \partial_{\nu}^+u_b|_{\partial \Omega}$ and $\psi_2 =  \partial_{\nu} u_b|_{\Gamma\cap\Omega^c}$. Note that on $\partial\Om$, the unit vector $\nu$ is directed to the exterior of $\Om$. On $\partial\Om$, the quantities $[\cdot]$, $[\partial_{\nu}\cdot]$ are then defined as in (\ref{DefJumps}). Consider the space $\mY:=\mH^{-1/2}(\partial\Om)\times\mH^{-1/2}(\Gamma\cap\Omega^c)$, endowed with the natural norm that we denote $\|\cdot\|_{\mY}$, and introduce the operators 
\begin{equation}\label{DefHrel}
\begin{array}{rccl}
H^\rel:&\mL^2(\dsphere)& \rightarrow & \mY \\
&g &\mapsto & (\partial_{\nu}^+u_b|_{\partial\Om},\partial_{\nu} u_b|_{\Gamma\cap\Omega^c}),
 \end{array}
\end{equation}
where $u_b$ is the solution of (\ref{PbChampTotalFreeSpaceComp}) with $u_i=v_g$. Define also
\begin{equation}\label{DefGrel}
\begin{array}{rccl}
G^\rel: & \mY &  \rightarrow & \mL^2(\mathbb{S}^{d-1})\\
&(\psi_1,\psi_2) &\mapsto & w^\infty ,
\end{array}
\end{equation}
where $w^\infty$ is the far field pattern of the solution of \eqref{eq:Def_G+AB}. The following factorization of $F^\rel$ is then straightforward
\[
F^\rel = G^\rel H^\rel.
\]
For $z\in\R^d$, introduce $\Phi_z$ the outgoing Green function which solves  
$\Delta \Phi_z + k^2\Phi_z= -\delta_z$ in the sense of distributions of $\R^d$ and such that 
\begin{equation}\label{FormulaPhi}
 \Phi_z(x) = \frac{i}{4}(J_0(k|x-z|)+iY_0(k|x-z|)) \text{ if } d=2;
\qquad \qquad
\Phi_z(x) = \frac{1}{4\pi}\frac{e^{ik|x-z|}}{|x-z|} \text{ if } d=3.
\end{equation}
Here $J_0$ and $Y_0$ are respectively the Bessel functions of order zero of first kind and of second kind. It is known that the far field pattern of $\Phi_z$ is the function $\Phi_z^\infty$ such that $\Phi_z^\infty(\theta_s)=e^{-ikz\cdot\theta_s}$ for $\theta_s\in\dsphere$. We begin with the simple but fundamental following result.
\begin{proposition} \label{Prop:LSM+AB}
Assume that $k^2$ is not a RTE. Then $\Phi_z^\infty\in \mrm{Range}\,G^\rel$ if and only if $z\in\Omega$.
\end{proposition}

\begin{proof}
First consider some $z\in\Omega$. When $k^2$ is not a RTE, the problem 
\[
\begin{array}{|rccl}
\Delta w + k^2 w &=&  0 & \text{in } \Omega\setminus\Gammabar\\
\partial_{\nu}^{\pm}w &=& 0 & \text{on } \Gamma\cap\Omega\\
 w & = & \Phi_z & \text{on } \d\Omega. 
\end{array}
\]
admits a unique solution $w\in \mH^1(\Omega\setminus\Gammabar)$. Extend $w$ by $\Phi_z$ in $\R^d\setminus(\overline{\Om}\cup\overline{\Gamma})$. Then $w$ is a solution to \eqref{eq:Def_G+AB} with $\psi_1 = \partial_\nu(w-\Phi_z)|_{\partial\Om}$ and $\psi_2 = -\partial_{\nu}\Phi_z|_{\Gamma\cap\Omega^c}$, which shows that $\Phi_z^\infty\in \mrm{Range}\,G^\rel$.\\
\newline
Now let $z\in\R^d\setminus\Omega$. Assume that there exists $(\psi_1,\psi_2)\in\mY$ such that $G^\rel(\psi_1,\psi_2) = \Phi_z^\infty$. By the Rellich lemma,  we then have $w = \Phi_z$ in $\R^d\setminus(\overline{\Om}\cup\overline{\Gamma})$ where $w$ is the solution to \eqref{eq:Def_G+AB}. This leads to a contradiction due to the singular behaviour of $\Phi_z$ because $w$ is in $\mH^1_{\loc}(\R^d\setminus\overline{\Gamma})$.
\end{proof}
\noindent We now introduce a quantity that will serve as a penalization term below in the minimization of the functional $J^n_z$ (see \eqref{def:cost_function+AB}). We emphasize that it is a crucial ingredient in the two methods we propose. For $g\in \mL^2(\dsphere)$, set 
\begin{equation}\label{def:penalization+AB}
P(g) := ( F^\sharp g,g)_{\mL^2(\dsphere)}  +  (F^b_\sharp g,g)_{\mL^2(\dsphere)}.
\end{equation}
First, we show that $P(g)$ is equivalent to $\|H^\rel g\|_{\mY}$. In the following, $c,\,C>0$ denote generic constants whose value may change from one line to another but remains independent from $g\in \mL^2(\dsphere)$.
\begin{lemma}\label{Lem:Penalization+AB}
Assume that $k^2$ is not a Dirichlet eigenvalue of problem \eqref{eq:RTEPbNoCrack+AB} and that $\mathcal{H}_\Gamma : \mL^2(\dsphere) \to \mH^{-1/2}(\Gamma)$ is injective. Then there exist constants $c$, $C>0$ such that for all $g\in \mL^2(\dsphere)$, we have
\begin{equation} \label{equivH}
c\,\|H^\rel g\|^2_{\mY}\le 
\|H_{\partial\Omega} g\|^2_{\mH^{-1/2}(\partial\Om)}+\|\mathcal{H}_{\Gamma} g\|^2_{\mH^{-1/2}(\Gamma)} \leq C\,\|H^\rel g\|^2_{\mY}.
\end{equation}
Here $H^\rel$, $H_{\partial\Omega}$, $\mathcal{H}_{\Gamma}$ are respectively defined in (\ref{DefHrel}), (\ref{DefHpartialOm}), (\ref{DefHGamma}). This implies in particular that 
\begin{equation} \label{penalization-equiv}
c'\,\|H^\rel g\|^2_{\mY}\le P(g)\leq C'\,\|H^\rel g\|^2_{\mY}
\end{equation}
for some constants $c'$, $C'>0$ independent from $g\in \mL^2(\dsphere)$.
\end{lemma}
\begin{proof} 
For $g\in \mL^2(\dsphere)$, let $u_b$ be the solution of (\ref{PbChampTotalFreeSpaceComp}) with $u_i=v_g$. Let us start by establishing the first inequality of (\ref{equivH}). By definition, we have 
\begin{equation}\label{TermsToControl}
\begin{array}{lcl}
\|H^\rel g\|^2_{\mY}&=&\|\partial_{\nu}^+u_b\|^2_{\mH^{-1/2}(\partial\Om)}+\|\partial_{\nu} u_b\|^2_{\mH^{-1/2}(\Gamma\cap\Omega^c)} \\[6pt]
&=&\|H_{\partial\Omega} g\|^2_{\mH^{-1/2}(\partial\Om)}+\|\partial_{\nu} u_b\|^2_{\mH^{-1/2}(\Gamma\cap\Omega^c)}.
\end{array}
\end{equation}
Since $u_b=v_g+u_{b,s}$, we can write
\begin{equation}\label{MonEq13}
\begin{array}{lcl}
\|\partial_{\nu} u_b\|_{\mH^{-1/2}(\Gamma\cap\Omega^c)} &\le& \|\partial_{\nu} v_g\|_{\mH^{-1/2}(\Gamma\cap\Omega^c)}+\|\partial_{\nu} u_{b,s}\|_{\mH^{-1/2}(\Gamma\cap\Omega^c)} \\[6pt]
&\le& \|\mathcal{H}_{\Gamma} g\|_{\mH^{-1/2}(\Gamma)}+\|\partial_{\nu} u_{b,s}\|_{\mH^{-1/2}(\Gamma\cap\Omega^c)}.
\end{array}
\end{equation}
The scattered field $u_{b,s}$ associated with (\ref{PbChampTotalFreeSpaceComp}) admits the integral representation, for $x\in \R^d \setminus \overline \Omega$, 
\begin{equation}\label{SingleLayerPot}
\begin{array}{lcl}
u_{b,s}(x) &= & \dsp\int_{\partial \Omega} \partial_{\nu} \Phi_x (y)u_{b,s}(y)-\Phi_x(y)\partial_{\nu}^+u_{b,s}(y)\, ds(y)\\[10pt]
&= &  \dsp\int_{\partial \Omega} \partial_{\nu} \Phi_x (y)(u_{b,s}+v_g)(y)-\Phi_x(y)\partial_{\nu}^+(u_{b,s}+v_g)(y)\, ds(y)\\[10pt]
 &= & - \dsp\int_{\partial \Omega} \Phi_x (y)\partial_{\nu}^+u_{b} \, ds(y).
\end{array}
\end{equation}
Note that the second line above has been obtained from the first one thanks to a simple integration by parts. The continuity properties of single layer potentials (see \cite[Theorem 6.11]{McLean}) then imply that 
 \begin{equation} \label{estimusb}
\|\partial_{\nu} u_{b,s}\|_{\mH^{-1/2}(\Gamma\cap\Omega^c)} \le C \|\partial_{\nu}^+u_b\|_{\mH^{-1/2}(\partial\Om)}=C \|H_{\partial\Omega} g\|_{\mH^{-1/2}(\partial\Om)}.
\end{equation}
 Consequently, since   $v_g = u_b - u_{b,s}$, we can write
\begin{equation}\label{MonEq6}
\|\partial_{\nu} v_g\|_{\mH^{-1/2}(\Gamma\cap\Omega^c)} \le C \left(\|\partial_{\nu} u_b\|_{\mH^{-1/2}(\Gamma\cap\Omega^c)}+ \|\partial_{\nu}^+u_b\|_{\mH^{-1/2}(\partial\Om)}  \right).
\end{equation}
Inserting \eqref{MonEq6} in \eqref{MonEq13} and \eqref{MonEq13} in \eqref{TermsToControl}, we get the first inequality of (\ref{equivH}).\\ 
\newline
Let us prove the second estimate of (\ref{equivH}). Since there holds $\Delta v_g+k^2v_g=0$ in $\Om$, when $k^2$ is not a Dirichlet eigenvalue for the Laplace operator in $\Omega$, we can write
\begin{equation}\label{MonEq4}
\|\partial_{\nu} v_g\|_{\mH^{-1/2}(\Gamma\cap \Omega)}\le 
c\,(\|\Delta v_g\|_{\mL^2(\Om)}+\|v_g\|_{\mH^1(\Om)}) \le C \|v_g\|_{\mH^{1/2}(\partial\Om)}.
\end{equation}
But on $\partial\Om$, there holds $v_g=-u_{b,s}$ and again from the continuity properties of single layer potentials, from (\ref{SingleLayerPot}), we get 
\begin{equation}\label{MonEq4ter}
\|v_g\|_{\mH^{1/2}(\partial\Om)}=\|u_{b,s}\|_{\mH^{1/2}(\partial\Om)} \le  C \|\partial_{\nu}^+u_b\|_{\mH^{-1/2}(\partial\Om)}=C \|H_{\partial\Omega} g\|_{\mH^{-1/2}(\partial\Om)}.
\end{equation}
Gathering (\ref{MonEq6}), (\ref{MonEq4}) and (\ref{MonEq4ter}), we obtain
\[
\|\mathcal{H}_{\Gamma} g\|_{\mH^{-1/2}(\Gamma)} = \|\partial_{\nu} v_g\|_{\mH^{-1/2}(\Gamma)}\le C\,\|H_{\partial\Omega} g\|_{\mH^{-1/2}(\partial\Om)} \le C\,\|H^\rel g\|_{\mY}.
\]
This is enough to conclude to the second estimate of (\ref{equivH}).\\
\newline
Finally, from the factorization results of Propositions \ref{PropositionFactoTGamma} and \ref{PropositionFactoTpartialOm}, we can write
\[
\begin{array}{rcl}
c\,\|H_{\partial\Omega} g\|^2_{\mH^{-1/2}(\partial\Om)} &\le&  \langle H_{\partial\Omega} g,T_{\partial\Om} H_{\partial\Omega} g\rangle_{\partial\Om}= ((F^b)^{\sharp}g,g)_{\mL^2(\dsphere)}\\[6pt]
c\,\|\mathcal{H}_{\Gamma} g\|^2_{\mH^{-1/2}(\Gamma)} &\le& \langle \mathcal{H}_{\Gamma} g,\mathcal{T}_\Gamma \mathcal{H}_{\Gamma} g\rangle_{\Gamma}= (F^{\sharp}g,g)_{\mL^2(\dsphere)}.
\end{array}
\]
Summing these two lines and using (\ref{equivH}) provides the first inequality of \eqref{penalization-equiv}. From the continuity properties of the operators $T_{\partial\Om}$, $\mathcal{T}_{\Gamma}$ we obtain 
\[
\begin{array}{rcl}
((F^b)^{\sharp}g,g)_{\mL^2(\dsphere)} &\le& C\,\|H_{\partial\Omega} g\|_{\mH^{-1/2}(\partial\Om)}^2 \\
(F^{\sharp}g,g)_{\mL^2(\dsphere)} &\le& C\,\|\mathcal{H}_{\Gamma} g\|_{\mH^{-1/2}(\Gamma)}^2.
\end{array}
\]
Combining these estimates and (\ref{equivH}) gives the second inequality of \eqref{penalization-equiv}.
\end{proof}

\noindent
Now we have all the ingredients to show that RTEs can be determined from $F^\rel$. Let $(\alpha_n)$ be a sequence of positive numbers such that $\lim_{n\to+\infty}\alpha_n=0$. For a given $z\in\Om$ and $n\in\N$, define the functional $J_z^n:\mL^2(\sphere)\to\R$ such that 
\begin{equation}\label{def:cost_function+AB}
J_z^n(g):= \alpha_n P(g)+
\|F^\rel g - \Phi_z^\infty\|_{\mL^2(\sphere)}^2
\end{equation}
where $P(g)$ is defined by (\ref{def:penalization+AB}). Set 
\begin{equation}\label{DefInf}
j_z^n := \inf_{g\in \mL^2(\dsphere)} J^n_z(g)\ge0
\end{equation}
and introduce a sequence  $(g^n_z)_{n\in\N}$ of elements of $\mL^2(\sphere)$ such that for all $n\in\N$,
\begin{equation}
\label{def:GLSMconvergenceCondition+AB}
 J_z^n(g_z^n)\leq j_z^n + \alpha_n^\eta.
\end{equation}
The power exponent $\eta$ is chosen such that $\eta>1$. Roughly speaking, for almost every $z\in\Om$, we will show that the quantity $\lim_{n\to+\infty}P(g_z^n)$ takes infinite values if and only if $k$ is a RTE. We split this result in Theorems \ref{th:GLSM+AB} and \ref{th:RTEsComputation+AB} below. 
\begin{theorem}\label{th:GLSM+AB}
Assume that $k^2$ is not a RTE nor a Dirichlet eigenvalue of problem \eqref{eq:RTEPbNoCrack+AB}. Assume also that $\mathcal{H}_\Gamma : \mL^2(\dsphere) \to \mH^{-1/2}(\Gamma) 
$ is injective. Then for all $z\in\Omega$, the sequence $(P(g_z^n))$ remains bounded as $n\to+\infty$.
\end{theorem}
\begin{proof}
Consider a point $z\in\Omega$. From the definitions of $J_z^n$, $g_z^n$ and $j_z^n$, there holds
\begin{equation}\label{ineq:GLSMfirstEstimation}
\alpha_n P(g_z^n)  \leq   J_z^n(g_z^n) \leq j_z^n + \alpha_n^\eta.
\end{equation}
According to Proposition \ref{Prop:LSM+AB}, there exists $\varphi_z\in\mY$ such that $G^\rel\varphi_z=\Phi_z^\infty$. Using the fact that the range of $H^\rel$ is dense in $\mY$ (see Proposition \ref{PropDenseRange} in Appendix),  there exists $\psi_n\in\mL^2(\dsphere)$ such that $\|H^\rel\psi_n-\vf_z\|_{\mY}^2 <\alpha_n$. The continuity of $G^\rel:\mY\to\mL^2(\sphere)$ then allows us to write
\begin{equation}\label{Relation1}
\begin{array}{lcl}
j_z^n \le J^n_z(\psi_n)&=& \alpha_n P(\psi_n)+
\|F^\rel \psi_n - \Phi_z^\infty\|_{\mL^2(\sphere)}^2\\[4pt]
 & \le & \alpha_n P(\psi_n)+\|G^\rel\|^2\|H^\rel\psi_n-\vf_z\|_{\mY}^2\\[4pt]
 & \le & \alpha_n\,(P(\psi_n)+\|G^\rel\|^2).
\end{array}
\end{equation}
Inserting (\ref{Relation1}) in (\ref{ineq:GLSMfirstEstimation}) and using Lemma  \ref{Lem:Penalization+AB}, we conclude  that 
\[
P(g_z^n) \le (C_2\|H^\rel\psi_n\|_{\mY}^2 +  \|G^\rel\|^2) + \alpha_n^{\eta-1}.
\]
This guarantees that $(P(g_z^n))$ remains bounded as $n\to+\infty$ (note that $\|H^\rel\psi_n\|_{\mY}^2 \le \|\varphi_z\|_{\mY}^2+\alpha_n$). 
\end{proof}

\noindent
The final ingredient allowing us to obtain a method to compute RTEs is the following result complementary to the previous one.

\begin{theorem}\label{th:RTEsComputation+AB}
Assume that $k^2$ is a RTE but not a Dirichlet eigenvalue of \eqref{eq:RTEPbNoCrack+AB} and not a RNSE. Assume also that $\mathcal{H}_\Gamma : \mL^2(\dsphere) \to \mH^{-1/2}(\Gamma) 
$ is injective. Then for all non empty open set $\mathcal{O}\subset\Omega$, the function  $z\mapsto\liminf_{n\to+\infty}P(g_z^n)$ is not bounded in $\mathcal{O}$.
\end{theorem}

\begin{proof}
We proceed by using a contradiction argument. Without lost of generality, we assume that $\overline{\mathcal{O}}\subset\Omega\setminus\Gammabar$ (otherwise work with a subset of $\mathcal{O}$). Assume that exists $M>0$ such that for all $z\in\mathcal{O}$, there holds
\[
\liminf_{n\to+\infty}P(g_z^n)\leq M .
\]
Pick $z\in\mathcal{O}$. The result of Lemma \ref{Lem:Penalization+AB}, which requires that $k^2$ is not a Dirichlet eigenvalue of \eqref{eq:RTEPbNoCrack+AB} implies that $(H^\rel g_z^{n})$ is bounded in $\mY$. As a consequence, there is a subsequence, also denoted $(g_z^{n})$, such that $(H^\rel g_z^{n})$ weakly converges in $\mY$ to a limit denoted $\psi_z$. Since the operator $G^\rel:\mY\to\mL^2(\dsphere)$ is compact (Proposition \ref{PropCompactness} in Appendix), we deduce that $(G^\rel H^\rel g_z^{n})=(F^\rel g_z^n)$ converges to $G^\rel\psi_z$.\\
\newline
Let us show that $(F^\rel g_z^n)$ converges to $\Phi^{\infty}_z$ as $n\to+\infty$. Using the fact that the range of $F^\rel$ is dense in $\mL^2(\dsphere)$ when $k^2$ is not a RNSE (Proposition \ref{densityFrel} in Appendix), for any $\eps>0$, we can find $g\in\mL^2(\dsphere)$ such that $\|F^\rel g - \Phi_z^\infty\|_{\mL^2(\sphere)}^2\le \eps$. By definition of $j^n_z$ (see (\ref{DefInf})), then we have
\[
j^n_z \le \alpha_nP(g)+\eps,
\]
which is enough to conclude that $\lim_{n\to+\infty}j^n_z=0$. From (\ref{ineq:GLSMfirstEstimation}), we infer that 
\[
\lim_{n\to+\infty}J^n_z(g^n_z)=\alpha_nP(g^n_z)+\|F^\rel g^n_z - \Phi_z^\infty\|_{\mL^2(\sphere)}^2=0.
\]
Since $(P(g^n_z))$ is bounded and $(\alpha_n)$ tends to zero, this implies that $(F^\rel g_z^n)$ converges to $\Phi^{\infty}_z$ in $\mL^2(\sphere)$.\\
\newline
Thus we have $G^\rel\psi_z=\Phi^{\infty}_z$. Now by definition of $G^\rel$, there exists a solution $w_z$ of \eqref{eq:Def_G+AB} such that $w_z^\infty = \Phi_z^\infty$. The Rellich's Lemma implies that $w_z$ satisfies the following equations,

\begin{equation}
\label{eq:RTEPbSourcePhi+AB}
\begin{array}{|rccl}
\Delta w_z + k^2 w_z &=&  0 & \text{in } \Omega\setminus\Gammabar\\
\partial_{\nu}^{\pm}w_z &=& 0 & \text{on } \Gamma\cap\Omega\\
 w_z & = & \Phi_z & \text{on } \d\Omega.  \\
\end{array}
\end{equation}
We have shown the existence of a solution $w_z\in \mH^1(\Omega\setminus\Gammabar)$ to equation \eqref{eq:RTEPbSourcePhi+AB} for all $z\in\mathcal{O}$. Now we prove that this is not compatible with the fact that $k^2$ is a RTE. Let $w_0\not\equiv0$ be an eigenfunction of \eqref{eq:RTEPb+AB}. Multiplying the first equation of \eqref{eq:RTEPbSourcePhi+AB} by $w_0$ and integrating by parts, using also that $\partial_{\nu}^{\pm}w_0 = \partial_{\nu}^{\pm}w_z = 0$ on $\Gamma\cap\Omega$, we get
\[
0=\int_{\d\Omega} \d_\nu w_0 w_z- w_0\d_\nu w_z\, ds=\int_{\d\Omega} \d_\nu w_0 w_z\, ds=\int_{\d\Omega} \d_\nu w_0 \Phi_z- w_0\d_\nu\Phi_z\, ds. 
\]
The Green's representation theorem together with this last equation imply that for all $z\in\mathcal{O}$, we have 
\begin{equation}\label{Representationw_0+AB}
w_0(z)=\int_{\Gamma\cap\Omega}[w_0]\d_\nu\Phi_z\, ds(x) - \int_{\Gamma\cap\Omega}[\d_\nu w_0]\Phi_z \, ds(x).
\end{equation}
The right hand side of \eqref{Representationw_0+AB} is an outgoing solution of the Helmholtz equation in $\R^d\setminus(\Gammabar\cap\Omega)$ which coincides with $w_0$ in $\mathcal{O}$. From the unique continuation principle, we infer that it coincides with $w_0$ in $\Om\setminus\Gammabar$. Since we have $w_0=0$ on $\partial\Om$, we deduce that the right hand side of (\ref{Representationw_0+AB}) is null. And so $w_0$ must be zero in $\mathcal{O}$. Again from the unique continuation principle, this implies $w_0=0$ in $\Om\setminus\Gammabar$ which yields a contradiction. 
\end{proof}

\begin{remark} 
\label{rk:issueComputationRTE}
Theorems \ref{th:GLSM+AB} and \ref{th:RTEsComputation+AB} suggest that RTEs correspond to values of $k^2$ for which peaks are observed in the curve 
\begin{equation}\label{def:PeaksFunc+AB}
f:k \mapsto \int_{\mathcal{O}} P(g^n_z)\,dz
\end{equation}
for large values of $n$ and any non empty open set $\mathcal{O}\subset\Omega$. Note however that rigorously the above theorems require that $k^2$ is not a Dirichlet eigenvalue of \eqref{eq:RTEPbNoCrack+AB} nor a RNSE.
\end{remark}

\subsection{Description of the identification algorithm}

Let $F(k)$, $k\in[k_m;k_M]$, be a collection of far field operators associated with the initial problem (\ref{PbChampTotalFreeSpaceIntrod}) (scattering by cracks). For a discrete set of sampling points $T\subset\R^d$ covering the probed area, we consider the collection of artificial backgrounds $(\Omega_t)_{t\in T}$ defined by $\Omega_t:=B(t,\rho)$ where $B(t,\rho)$ is the ball centered at $t$ of radius $\rho>0$. First, we compute the far field pattern $F^b(k)$ for $k\in[k_m;k_M]$. To proceed, we need to solve the scattering problem (\ref{PbChampTotalFreeSpaceComp}). But since $\Omega_t$ is a ball, this can be done analytically in a quite cheap way. Thus we have access to $F^\rel(k)=F(k)-F^b(k)$. Then for all $k\in[k_m;k_M]$, for all $z\in\Om_t$, we approximate the solution of the far field equation $F^\rel(k)g = \Phi_z^\infty$ by computing $g_z^n(k)$ (see \eqref{def:GLSMconvergenceCondition+AB}) for small values of $\alpha_n>0$. Plotting the curve 
\begin{equation}\label{EigCurve}
\mathcal{E}^n_t : k \mapsto \int_{\Omega_t} P(g^n_z(k))\,dz
\end{equation}
and identifying the peaks as $n\to+\infty$ allows us to get the spectrum of Problem (\ref{eq:RTEPb+AB}). Below, we only compare the list of the  eigenvalues of (\ref{eq:RTEPb+AB}) contained in $[k^2_m;k^2_M]$ with the list of eigenvalues of (\ref{eq:RTEPbNoCrack+AB}) contained in the same interval. We denote them respectively by $\sigma_{\Gamma}(\Om_t)$ and $\sigma_{\emptyset}(\Om_t)$. Note that the reference list of eigenvalues $\sigma_{\emptyset}(\Om_t)$ is independent of $t$ and can be computed once for all. However it depends on $\rho$ and it is better to choose $k_m$, $k_M$ such that $\sigma_{\emptyset}(\Om_t)\ne\emptyset$. Then we have the following result 
\[
\sigma_{\Gamma}(\Om_t)\ne\sigma_{\emptyset}(\Om_t) \qquad \Rightarrow \qquad \Gamma\cap\Omega_t\neq\emptyset.
\]
This criteria can be implemented in practice by defining
\begin{equation}
\label{MultFrqCrackCriteria}
d(\sigma_\Gamma(\Omega_t),\sigma_\emptyset(\Omega_t)):=\max_{k^2_\Gamma\in\sigma_\Gamma(\Omega_t)}\min_{k^2_\emptyset\in\sigma_\emptyset(\Omega_t)}|k^2_\Gamma-k^2_\emptyset|+\max_{k^2_\emptyset\in\sigma_\emptyset(\Omega_t)}\min_{k^2_\Gamma\in\sigma_\Gamma(\Omega_t)}|k^2_\Gamma-k^2_\emptyset|.
\end{equation}
Then we have $d(\sigma_\Gamma(\Omega_t),\sigma_\emptyset(\Omega_t))=0$ if and only if $\sigma_{\Gamma}(\Om_t) = \sigma_{\emptyset}(\Om_t)$. However, this result is not completely satisfactory because we may have $\Gamma\cap\Omega_t\neq\emptyset$ and $\sigma_{\Gamma}(\Om_t) = \sigma_{\emptyset}(\Om_t)$ (in which case the possibility to compute the elements of $\sigma_{\Gamma}(\Om_t)$ is not justified by the theory, see Remark \ref{rk:issueComputationRTE}). 
In order to get a more robust indicator, we can average it as follows. Instead of simply defining $d(\sigma_\Gamma(\Omega_t),\sigma_\emptyset(\Omega_t))$ as the crack density at point $t$, we also involve nearby artificial disks by defining
\begin{equation}\label{IndicatorFunction}
\mathcal{I}(t):=\cfrac{1}{\mrm{card}\{s\in T\,|\, |t-s| \le\eta\}}\sum_{\substack{s\in T \\ |t-s| \le\eta}}d(\sigma_\Gamma(\Omega_s),\sigma_\emptyset(\Omega_s)).
\end{equation}
Here the parameter $\eta$ fixes the amount of disks which are taken into account.
With this indicator, we have more chances to obtain $\mathcal{I}(t)\ne0$ if and only if cracks are present in the set $W_t$ defined by 
$$W_t := \dsp\bigcup_{\substack{s\in T \\ |t-s| \le\eta}}\Omega_s.$$
 Besides, since the eigenvalues $k_{\Gamma}$ of \eqref{eq:RTEPb+AB} decreases with respect to the size of $\Gamma\cap\Om$ for the inclusion order, we also expect that $\mathcal{I}(t)$ gives a qualitative information on the density of cracks inside $W_t$ (the more cracks in $W_t$, the higher $\mathcal{I}(t)$). However this remains an intuition and this result seems hard to establish theoretically. On the other hand, we emphasize that in practice, we take a large but fixed $n$ in (\ref{EigCurve}). Therefore, in the numerics, we work with an approximation of $\mathcal{I}(t)$ that we denote
\begin{equation}\label{IndicatorFunctionDiscrete}
\mathcal{I}^n(t).
\end{equation}

\subsection{Numerical validation of the algorithm}\label{ParagrapheNum}

We conclude this section by testing the indicator function (\ref{IndicatorFunctionDiscrete}) in a two dimensional setting. For a given wavenumber $k$, we generate a discretization of the far field operator $F$ by solving numerically the direct problem \eqref{PbChampTotalFreeSpaceIntrod} for multiple incident fields $u_i(\cdot,\theta_p)=e^{-ik\theta_p\cdot x}$. To proceed, we use a boundary element method working with the solver gypsilab \cite{alouges2018fem}. We then compute the matrix $F=(u_s^\infty(\theta_q,\theta_p))_{p,q}$  for $\theta_p$, $\theta_q$ in $\{\cos(2\pi l/100),\,\sin(2\pi l/100),\,l=1\dots 100\}$ (somehow we discretize $\mL^2(\mathbb{S}^1)$). We then add random noise to the simulated $F$ and obtain a noisy far field data $F^\delta$ such that $F^\delta_{pq}=F_{pq}(1+\gamma N)$. Here $N$ is a complex random variable whose real and imaginary parts are uniformly distributed in $ [-1;1]^2$. The parameter $\gamma>0$ is chosen so that $\|F^\delta-F\|\leq\delta$. We repeat this process for multiple wavenumbers $k\in[k_{m};k_{M}]$ to obtain a collection of noisy far field data $(F^\delta(k))_{k\in[k_{m};k_{M}]}$. In what follows, we do not write the dependence with respect to the wavenumber to have a lighter notation.\\
\newline
We work with artificial obstacles which are balls of radius $\rho$ centered at $t$. In this case, we have an analytic expression of the far field pattern of the scattered field associated with an incident plane wave for the problem (\ref{PbChampTotalFreeSpaceComp}). For $\Om_t=\Om_0$ (the ball of radius $\rho$ centered at the origin), the asymptotic behavior of the Hankel functions, we find
\begin{equation}\label{DefScaAnalytic}
u_s^\infty(\theta_q,\theta_p) = e^{-i\pi/4}\sqrt{\frac{2}{\pi k}}\sum_{m\in\Z} -\frac{J_m(k\rho)}{H_m(k\rho)} e^{im(\theta_q-\theta_p)}=:\tilde{u}_s^\infty(\theta_q,\theta_p).
\end{equation}
From the translation formula (see \cite[Identity (5.3)]{Colton-Kress}), the solution $\Omega_t$ for a generic $t\in\R^2$, has the following far field pattern
\[
u_s^\infty(\theta_q,\theta_p) =e^{ikt\cdot(\theta_p - \theta_q)} \tilde{u}_s^\infty(\theta_q,\theta_p).
\]
Consequently, if we define the two matrices $T$ and $F^b$ by 
\[
T_{pq}  = e^{ikt\cdot(\theta_p - \theta_q)} \qquad \text{ and } \qquad F^b_{pq} = \tilde{u}_s^\infty(\theta_q,\theta_p),
\]
then the relative far field operators $F^\rel_t$ are given by
\begin{equation}\label{eq:ComputeFrel}
F^\rel_t = F^\delta -F^b_t,
\end{equation}
where $F^b_t$ is the component wise multiplication of $T$ with $F^b$. \\
\newline
To handle the noise added on the data, we work with a regularized version of the cost function $J_z^n$ introduced in \eqref{def:cost_function+AB}. It consists in finding minimizers $g_z^{n,\delta}(t)$ of the functional $J_z^{n,\delta}(g,t)$ such that
\begin{equation}\label{def:RegCostFunc+AB}
J_z^{n,\delta}(g,t)=  \alpha_n\left(P^\delta(g,t) +\delta\|g\|_{\mL^2(\mathbb{S}^{1})}^2\right)+ \|F^\rel_t g - \Phi_z^\infty\|_{\mL^2(\mathbb{S}^{1})}^2,
 \end{equation}
where $P^\delta(g,t) = \langle F^\delta_\sharp g,g\rangle + \langle (F^b_t)^\sharp g,g\rangle$. Following \cite[Section 5.2]{GLSM}, we fit $\alpha_n$ to $\delta$ as follows, 
\begin{equation}
\alpha_n(\delta,t) = \frac{\alpha^n_{\mrm{LSM}}(t)}{\|F^{\delta}_\sharp\|+\|(F^b_t)^\sharp\|+\delta}
\end{equation}
where $\alpha^n_{\mrm{LSM}}(k,t)$ is the regularization parameter given by the Morozov discrepancy principle in the Tikhonov regularization of the equation $F^\rel_t g = \Phi_z^\infty$. \\
\newline
From the computed $g_z^{n,\delta}(k,t)$, we identify the RTEs $k^2_\Gamma$ as described in the previous paragraph. We first illustrate the indicator function $\mathcal{I}^n(t)$  introduced in \eqref{IndicatorFunctionDiscrete} on a neat example (Figure \ref{Fig:DemoMult}). We also present the plot of the curve $\mathcal{E}^n_t$ defined by \eqref{EigCurve} for two different artificial disks $\Omega_{t_1}$ (Figure \ref{Fig:MultSeepNoCrack}) and $\Omega_{t_2}$ (Figure \ref{Fig:MultSeepCrack}) so one can appreciate how the first eigenvalue of \eqref{eq:RTEPb+AB}, determined by the first peak of the curve $\mathcal{E}^n_t$, deviates from the first eigenvalue of  \eqref{eq:RTEPbNoCrack+AB} when $\Omega_t$ intersects the crack. In Figure \ref{Fig:DemoMultHR}, we decrease the artificial disks radius to $\rho = 0.1$ to recover the crack with a better resolution.

\begin{figure}[!ht]
\centering
\includegraphics[height=5cm,trim = 14cm 2.5cm 13cm 2cm, clip]{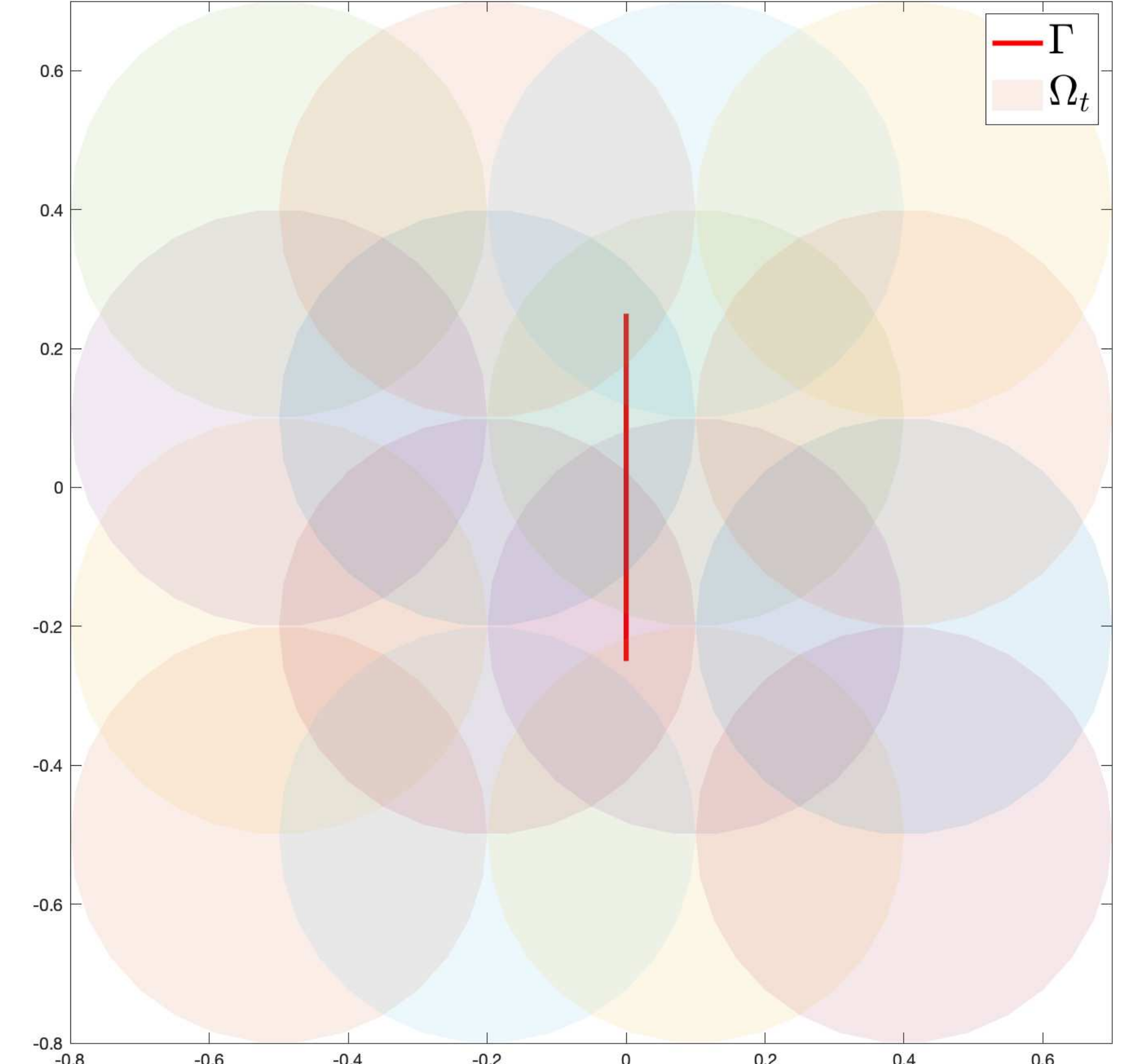}\qquad\qquad
\includegraphics[height=5cm]{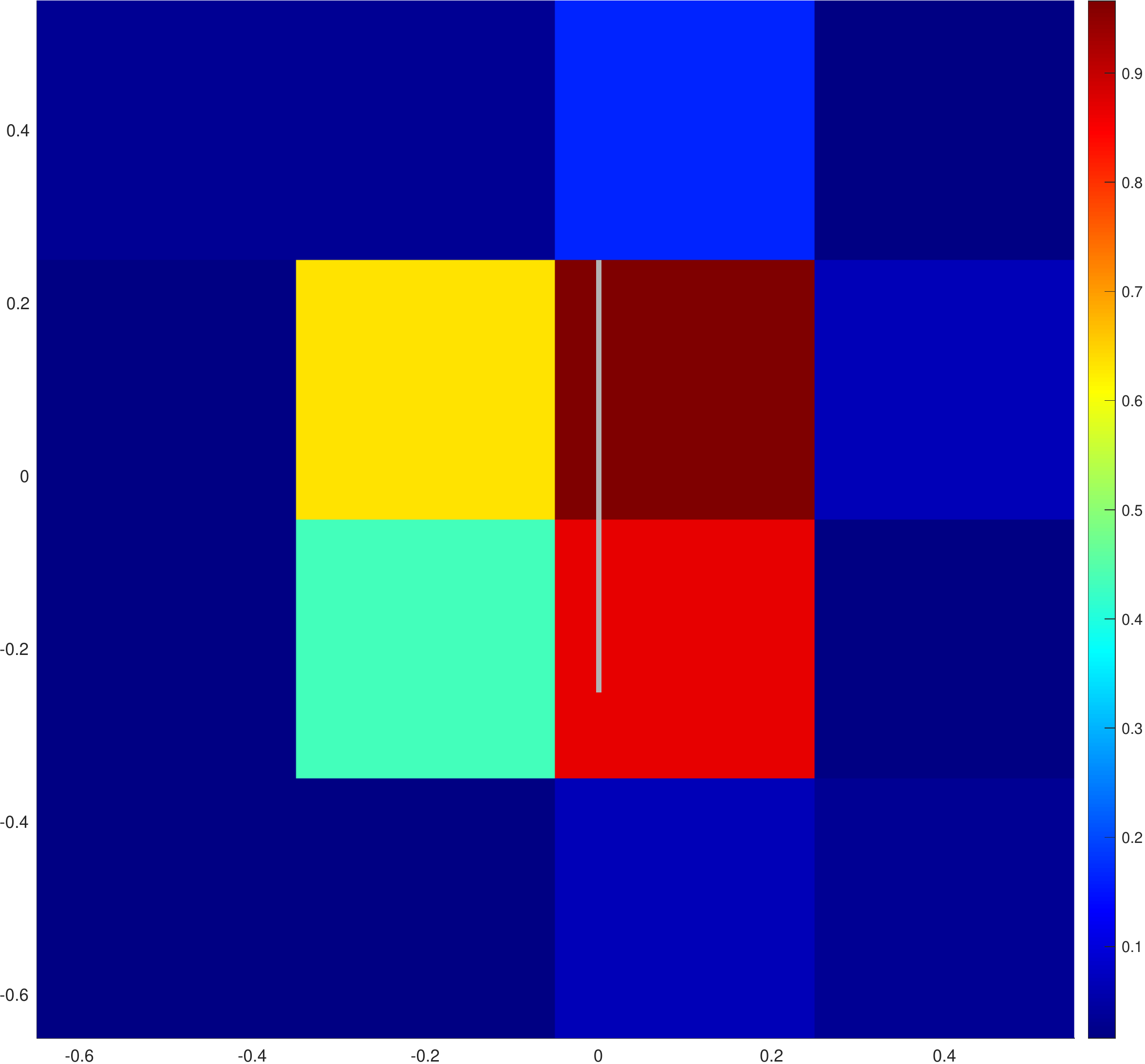}    
\caption{Left: collection of artificial disks of radius $\rho = 0.3$ used to identify a single crack of length 0.5. Right: results of the reconstruction provided by the indicator $\mathcal{I}^n$ (with the parameter $\eta = 0$).}                                                                             
\label{Fig:DemoMult}
\end{figure}

\begin{figure}[!ht]
\centering
\includegraphics[height=5cm]{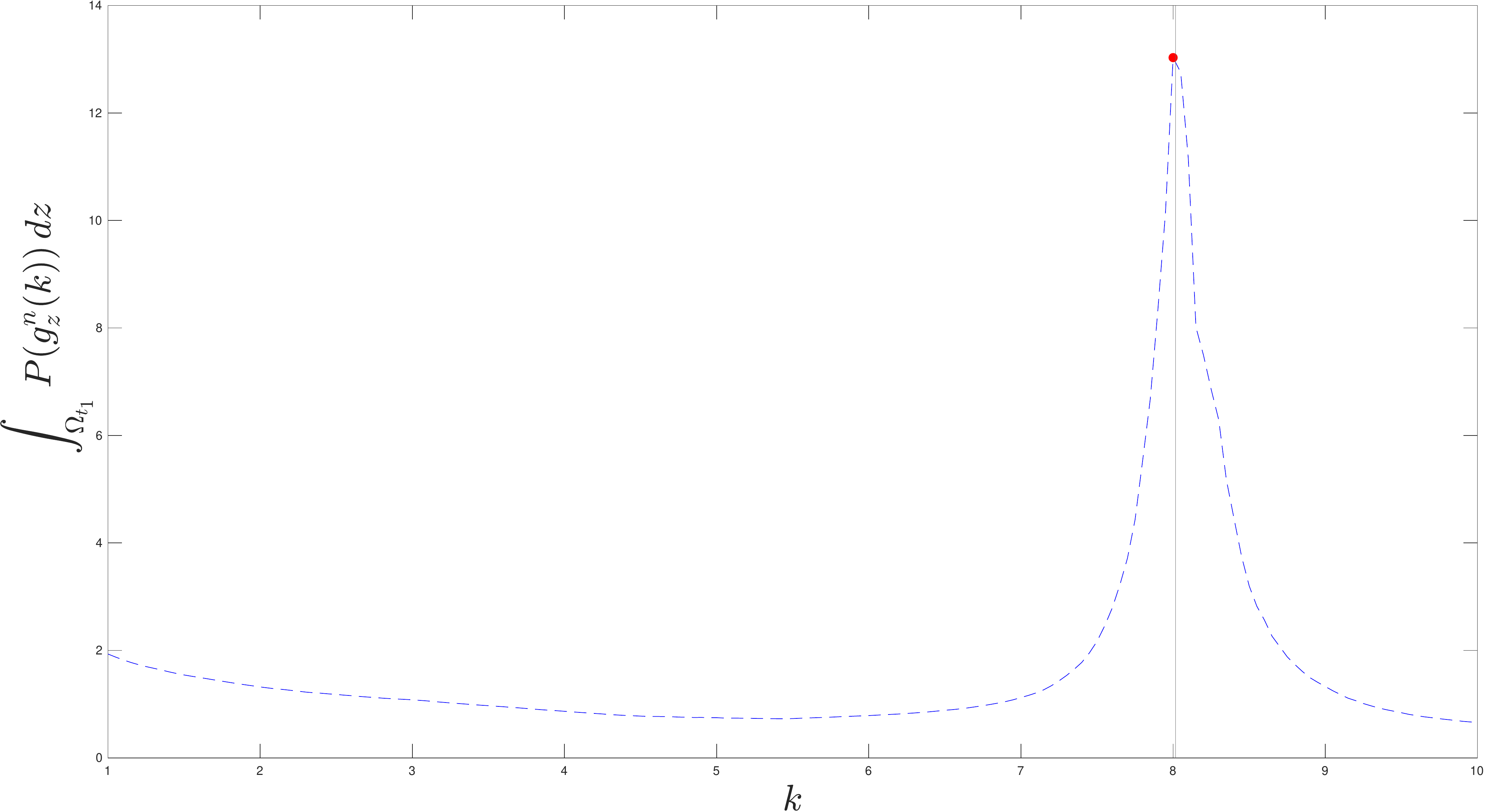}\qquad
\raisebox{0.25cm}{\includegraphics[height=4.7cm,trim = 14cm 2.5cm 13cm 2cm, clip]{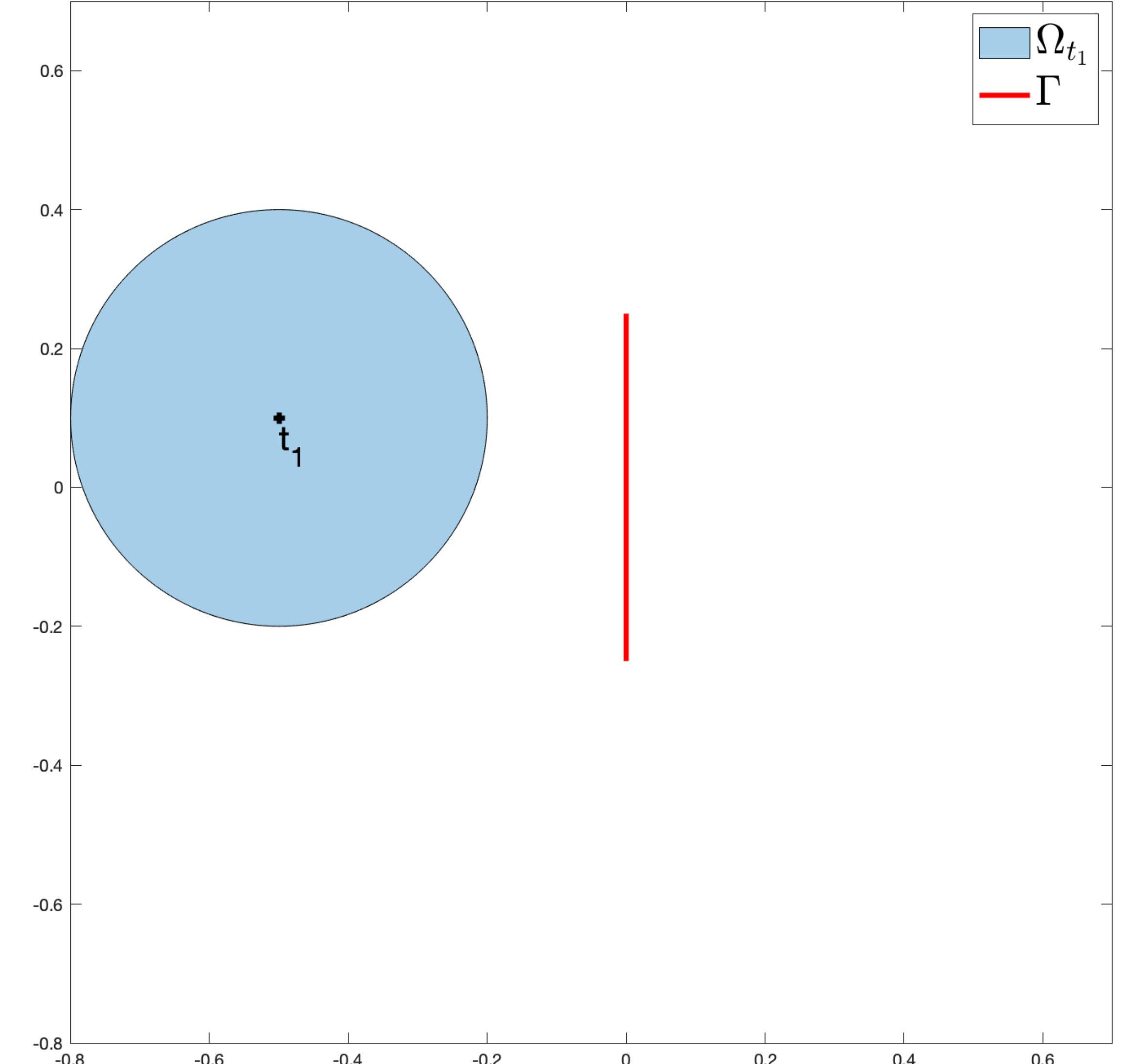}}            
\caption{Plot of the curve $\mathcal{E}^n_{t_1}$ (left) defined by \eqref{EigCurve} computed with the artificial background $\Omega_{t_1}$ represented on right. The peak of the curve is reached at $k_{\Gamma}\approx k_{\emptyset}$,  the latter quantity being indicated by the vertical line. Consequently, it is obtained that $\mathcal{I}^n(t_1) = |k^2_{\Gamma} - k_{\emptyset}^2|\approx0$.}                                                       
\label{Fig:MultSeepNoCrack}
\end{figure}

\begin{figure}[!ht]
\centering
\includegraphics[height=5cm]{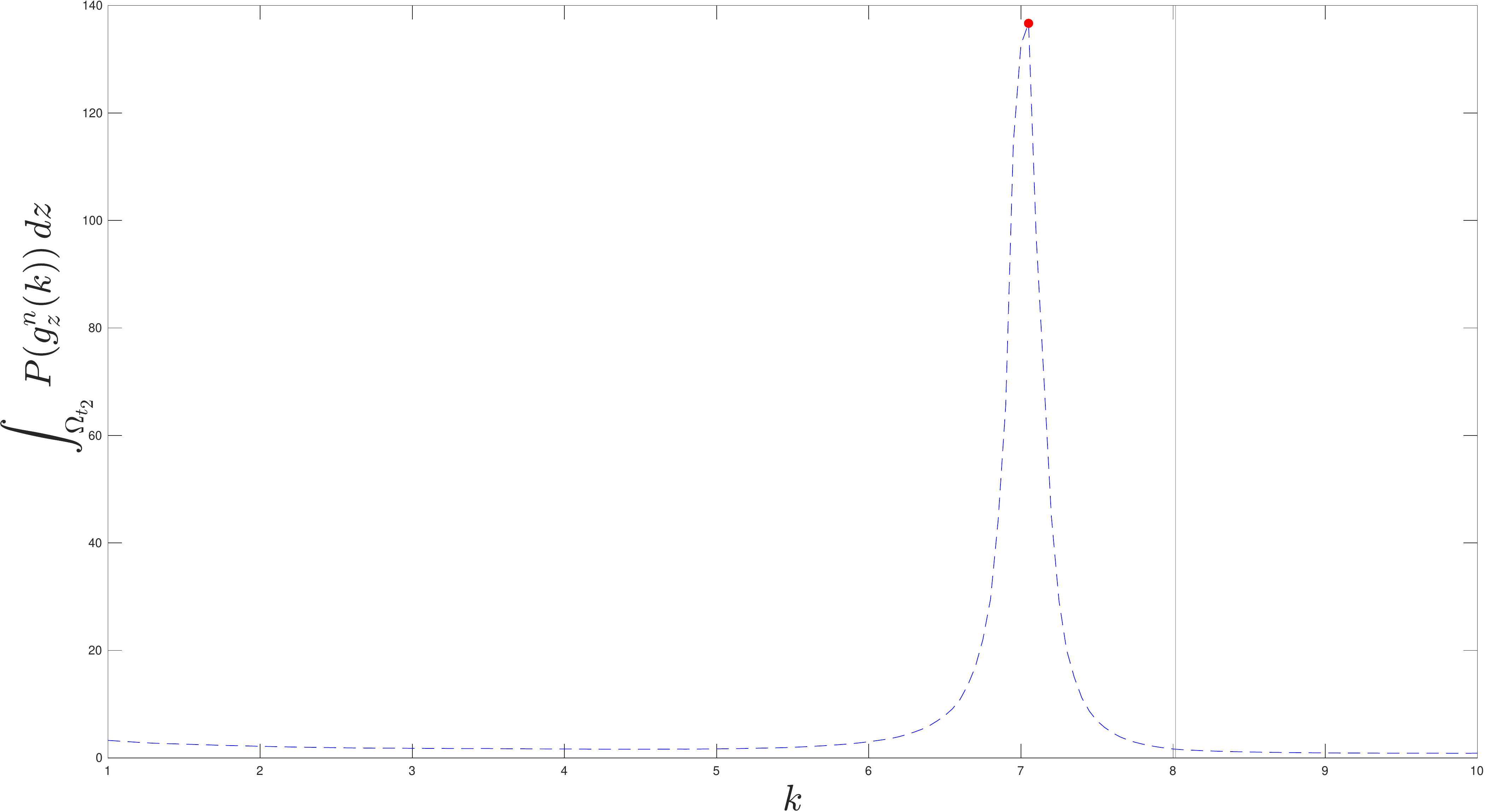}\qquad
\raisebox{0.25cm}{\includegraphics[height=4.7cm,trim = 14cm 2.5cm 13cm 2cm, clip]{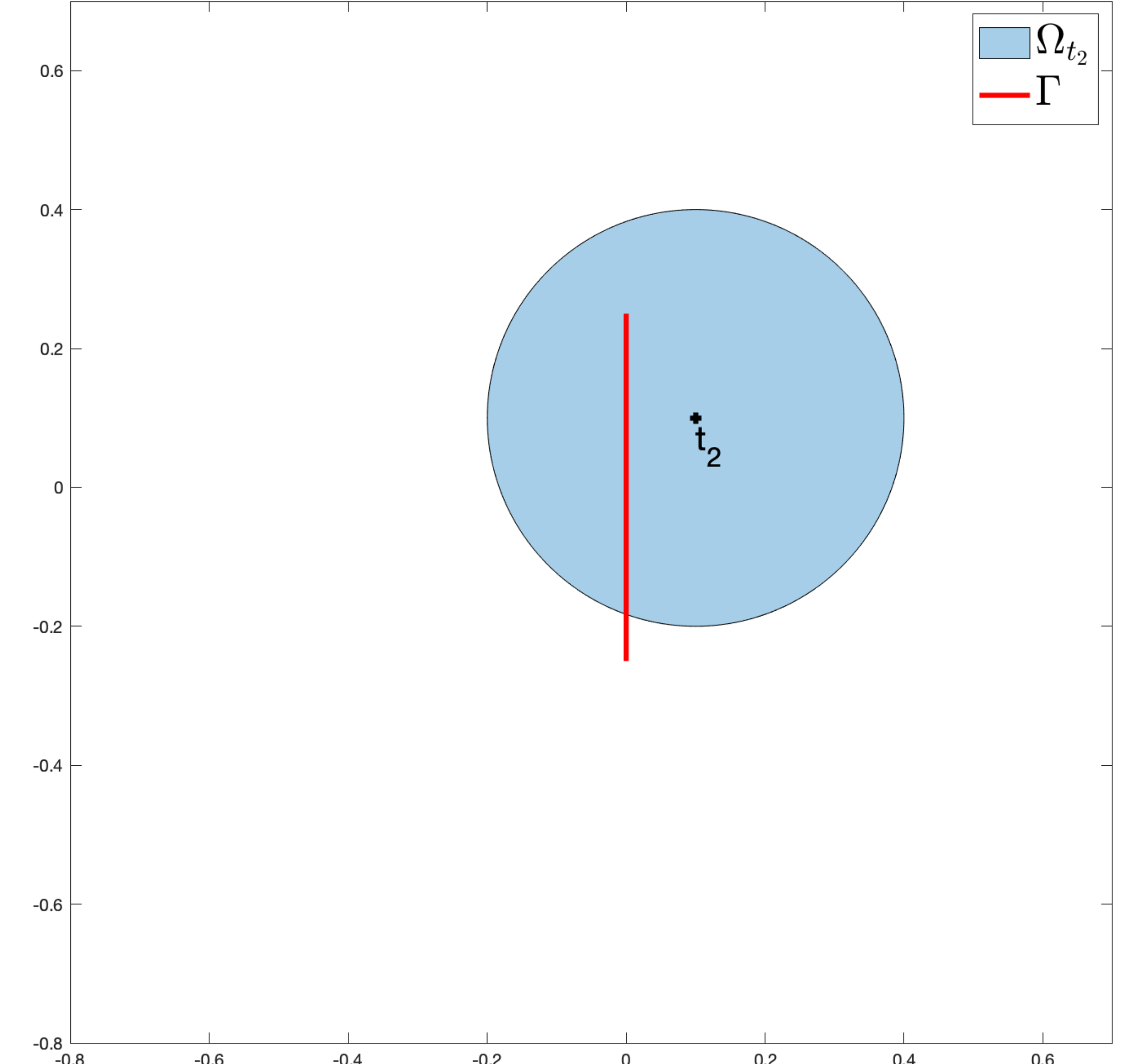}}
\caption{Plot of the curve $\mathcal{E}^n_{t_2}$ (left) defined by \eqref{EigCurve} computed with the artificial background $\Omega_{t_2}$ represented on right. The peak of the curve is reached at $k_{\Gamma}< k_{\emptyset}$,  the latter quantity being indicated by the vertical line. Consequently, we obtain that $\mathcal{I}^n(t_2) = |k^2_{\Gamma} - k_{\emptyset}^2|\ne0$.}                                                      
\label{Fig:MultSeepCrack}
\end{figure}

\begin{figure}[!ht]
\centering
\includegraphics[height=5cm]{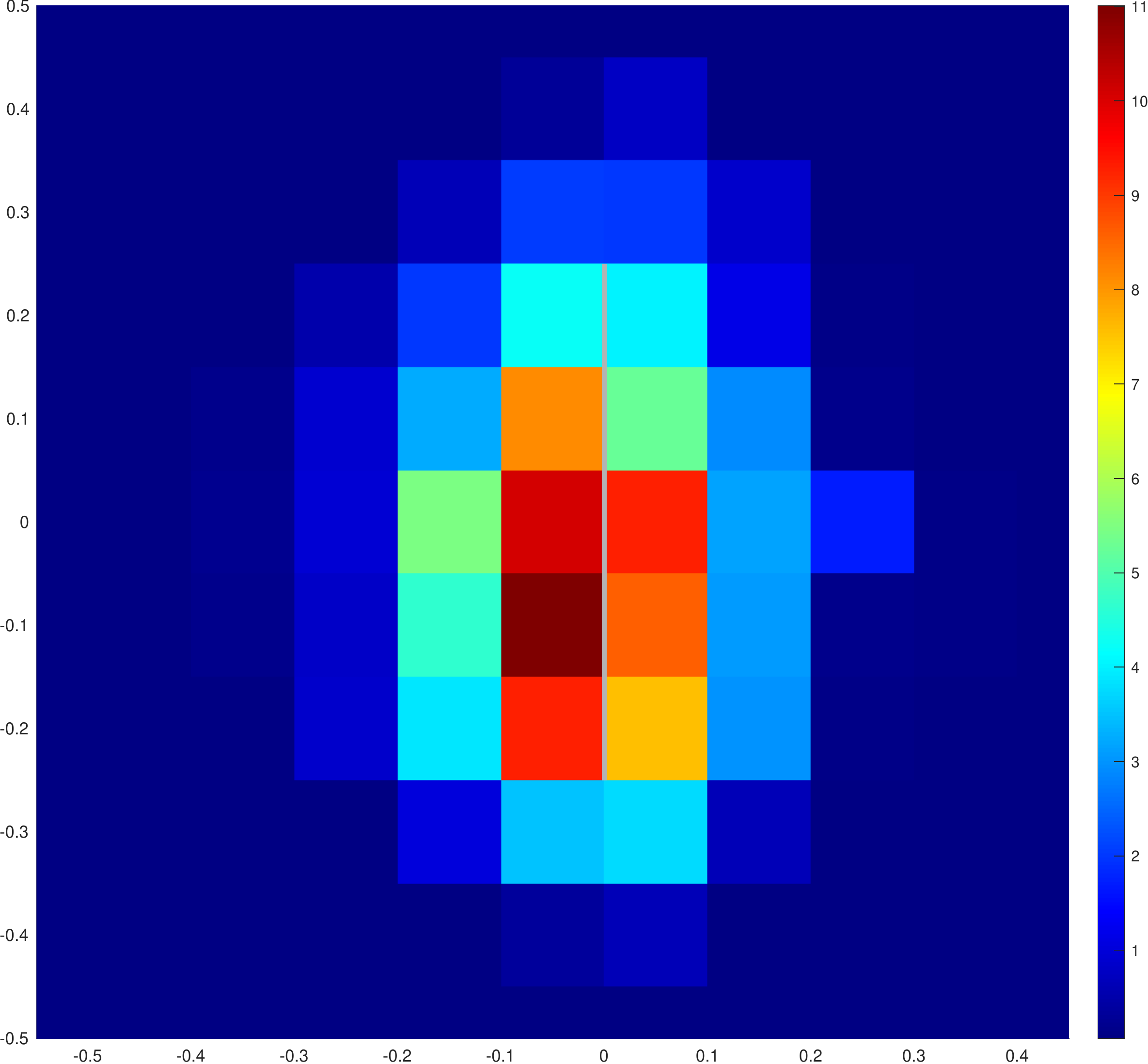}
\caption{Results of the reconstruction provided by the indicator $\mathcal{I}^n$ to recover the vertical crack of Figure \ref{Fig:DemoMult}. The method is carried with artificial disks of radius $\rho = 0.1$. The regularization parameter of the indicator is set to $\eta = \rho$.}                                                                             
\label{Fig:DemoMultHR}
\end{figure}

\newpage

\noindent Finally, we implement the method to two less academic situations in Figure \ref{fig:MultiFrq}. It gives satisfactory results in the sense that it allows us to quantify the damage level of the material. However, this method is quite expensive in computations to obtain a better resolution of the image. This is why we propose another approach in the next section, which requires only measurements at one fixed wavenumber and whose implementation is less costly.

\begin{figure}[!ht]
\centering
\includegraphics[height=5cm]{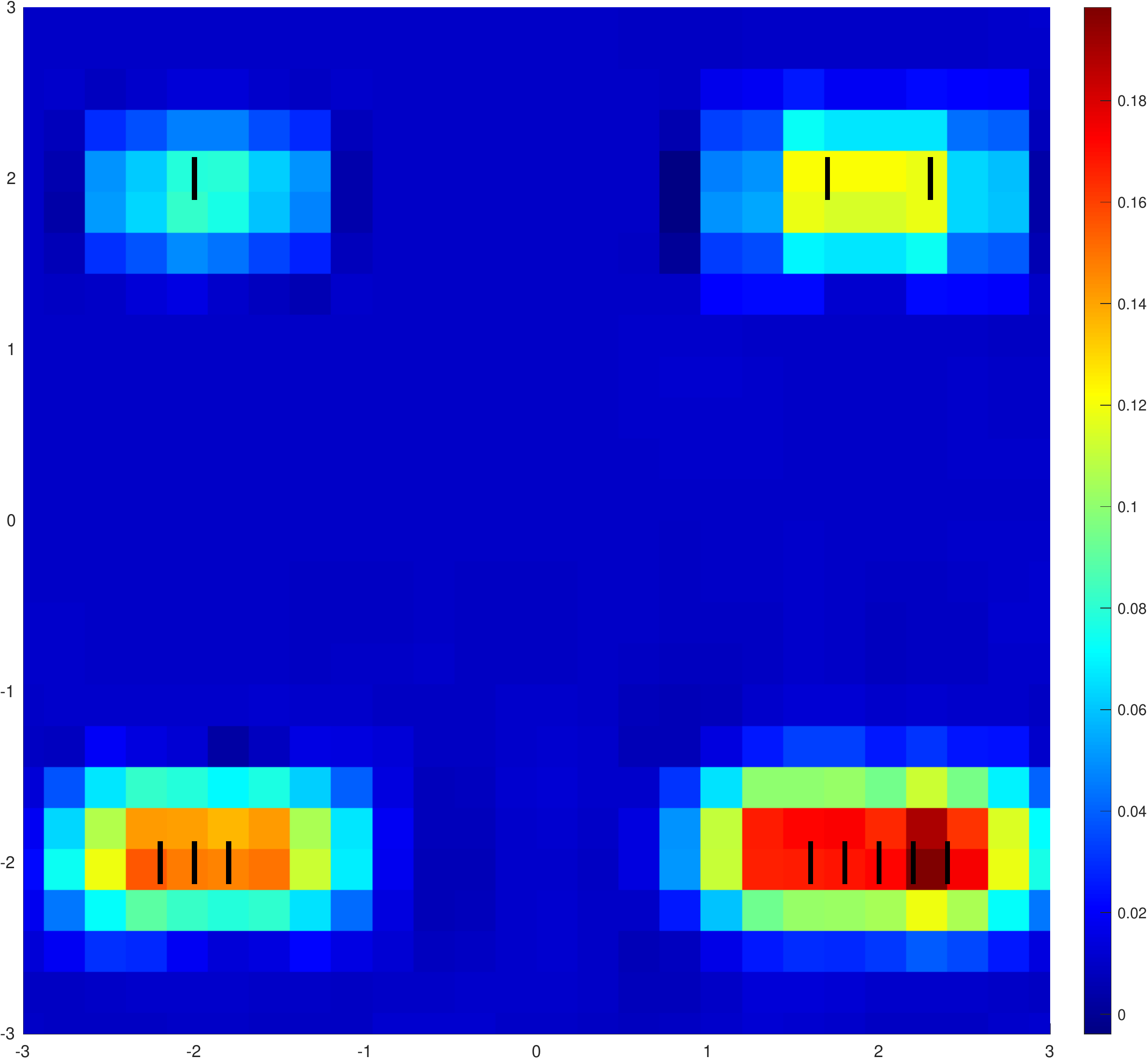}\qquad
\includegraphics[height=5cm]{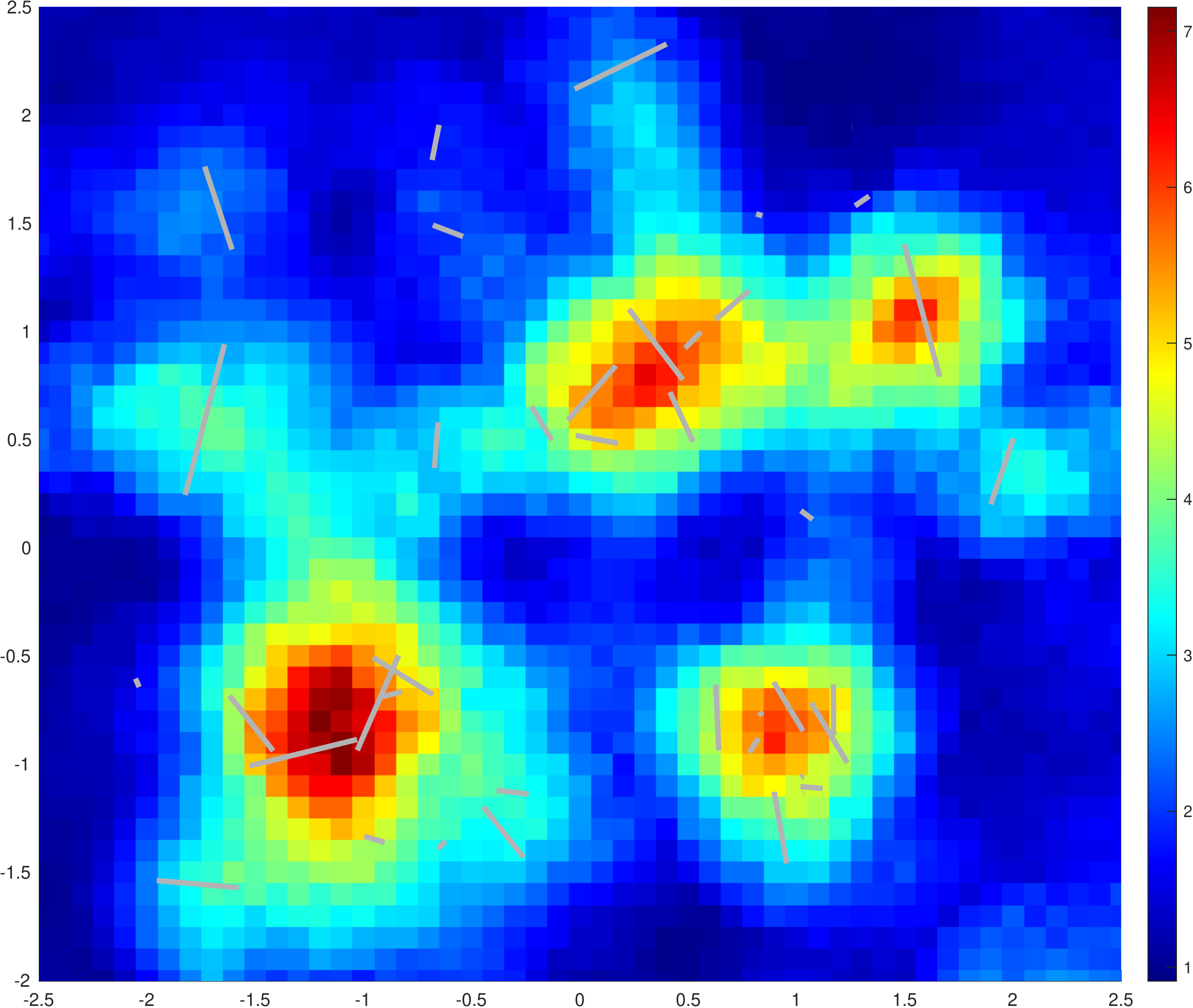}
\caption{Indicator $\mathcal{I}^n$ for two different simulated damaged materials. Left: 11 vertical cracks of length 0.25 arranged in  4 areas with different damage levels. The radius of the artificial obstacles  is $\rho = 0.25$. Right: 40 cracks of different lengths arranged randomly. The radius of the artificial obstacles is  $\rho=0.1$. The data is corrupted with $1\%$ of noise.} 
\label{fig:MultiFrq}
\end{figure}

\section{An alternative method using measurements at a fixed frequency}\label{SectionSecondApproach}

In this section, we develop an alternative method to detect the presence of cracks, the difference with the previous section being that the present approach requires far field data at only one fixed wavenumber. It is inspired by the Differential Linear Sampling Method introduced in \cite{DLSM} which allows to an observer having data before and after the emergence of a defect to get an idea where the damage has appeared. Here, instead of comparing far field data before and after the emergence of a defect, we will compare measured data and artificially computed data. The new proposed indicator function is expected to be sensitive to the local density of cracks but there is no theoretical justification for the quantification of this density (no monotonicity property is proved with respect to the size of the network). However, in addition to be much less costly, numerical simulations suggest that it may provide better results than the previous algorithm when the network is relatively sparse (see the numerical section \ref{ParagNum}).

\subsection{Approximation of the solution of the far field equation}

The proof of Proposition \ref{Prop:LSM+AB} guarantees that when $k^2$ is not a RTE, for $z\in\Om$, the equation $G^\rel\psi = \Phi_z^\infty$ admits a unique solution 
\begin{equation}\label{DefPsi}
\psi=(\psi_1,\psi_2)\qquad\mbox{ with }\qquad \psi_1 = \partial_\nu(w-\Phi_z)|_{\partial\Om}\qquad\mbox{ and }\qquad\psi_2 = -\partial_{\nu}\Phi_z|_{\Gamma\cap\Omega^c}.
\end{equation}
Here $w\in\mH^1(\Om\setminus\Gammabar)$ is the function such that
\begin{equation}\label{eq:wintEquation+AB}
\begin{array}{|rccl}
\Delta w + k^2 w &=&  0 & \text{in } \Omega\setminus\Gammabar\\
\partial_{\nu}^{\pm}w &=& 0 & \text{on } \Gamma\cap\Omega\\
 w & = & \Phi_z & \text{on } \d\Omega. 
\end{array}
\end{equation}
The following theorem shows how one can approximate $\psi$, and so $\partial_{\nu}w$ on $\partial\Om$, by means of the sequence of minimizers of the functional $J^n_z$ defined in (\ref{def:cost_function+AB}). This is interesting because $\partial_{\nu}w|_{\partial\Om}$ contains information on the cracks located inside $\Om$ that we will exploit in a second step.
\begin{theorem}\label{th:Converge2FFSolution+AB}
Assume that $k^2$ is not a RTE nor a Dirichlet eigenvalue of problem \eqref{eq:RTEPbNoCrack+AB}. Assume also that $\mathcal{H}_\Gamma : \mL^2(\dsphere) \to \mH^{-1/2}(\Gamma) 
$ is injective. Then for all $z\in\Omega$, there is a unique $\psi\in\mY$ such that $G^\rel\psi = \Phi_z^\infty$. Furthermore the sequence $(g_z^n)$ defined by \eqref{def:GLSMconvergenceCondition+AB} is such that $(H^{\rel} g_z^n)$ converges strongly to $\psi$ (defined in \eqref{DefPsi}) in $\mY$.
 \end{theorem}

\begin{proof}
The proof of Theorem \ref{th:GLSM+AB} guarantees that the sequence $(P(g_z^n))$ remains bounded as $n\to+\infty$. And the beginning of the proof of Theorem \ref{th:RTEsComputation+AB} ensures that we can extract a subsequence of $(g_z^n)$, also denoted $(g_z^n)$, such that $(H^\rel g_z^n)$ converges weakly to some $\psi\in\mY$. Moreover, we also know that $(G^\rel H^\rel g_z^n)$ converges to $\Phi^{\infty}_z$ (see again the proof of Theorem \ref{th:RTEsComputation+AB}). Thus we have $G^\rel\psi=\Phi^{\infty}_z$ and we deduce that $\psi$ is as in (\ref{DefPsi}). We now show the strong convergence of $(H^{\rel}g_z^n)$ to $\psi$ in $\mY$.\\
\newline 
Since the operator $H^\rel: \mL^2(\dsphere) \rightarrow \mY$ has dense range (Proposition \ref{PropDenseRange} in Appendix), there is a sequence $(g_p)_{p}$ of elements of $\mL^2(\dsphere)$ such that $(H^{\rel}g_p)$ converges to $\psi$ in $\mY$. From estimates (\ref{equivH}) of Lemma \ref{Lem:Penalization+AB}, we deduce that $(H_{\partial\Omega} g_p)_{p}$ and  $(\mathcal{H}_{\Gamma}g_p)_{p}$ converge respectively to $\psi_1\in\mH^{-1/2}(\partial\Om)$ and some $\psi_3\in\mH^{-1/2}(\Gamma)$. By definition of $j_z^n$, for all $p\ge1$, we have $j_z^n \leq J^n_z(g_p)$. Taking the limit $p\to+\infty$ and using that $G^{\rel}\psi=\Phi_z^\infty$, we obtain 
\[
j_z^n \leq \alpha_n\langle \psi_1 ,T_{\partial\Om} \psi_1\rangle_{\partial\Om} +\alpha_n \langle \psi_3,\mathcal{T}_\Gamma \psi_3\rangle_{\Gamma}.
\]
From the definition of $g_z^n$ (see \eqref{def:GLSMconvergenceCondition+AB}), we deduce that
\begin{equation} \label{limsup+AB}
\limsup_{n\to+\infty} \ \langle H_{\partial\Omega} g^n_z,T_{\partial\Om} H_{\partial\Omega}g^n_z\rangle_{\partial\Om} + \langle \mathcal{H}_{\Gamma}g^n_z,\mathcal{T}_\Gamma \mathcal{H}_{\Gamma}g^n_z\rangle_{\Gamma}
  \leq \langle \psi_1 ,T_{\partial\Om} \psi_1\rangle_{\partial\Om} +\langle \psi_3,\mathcal{T}_\Gamma \psi_3\rangle_{\Gamma}.
\end{equation}
Thanks to the coercivity properties of $T_{\partial\Om}$, $\mathcal{T}_{\Gamma}$, we can write
\begin{equation}\label{ineq:Estimation1CV+AB}
\begin{array}{ll}
& c\,\|H_{\d\Omega}g_z^n - \psi_1\|^2_{\mH^{-1/2}(\partial\Om)} \\[3pt]
 &\leq \langle H_{\partial\Omega} g^n_z-\psi_1,T_{\partial\Om} (H_{\partial\Omega}g^n_z-\psi_1)\rangle_{\partial\Om} \\[3pt]
& \leq \langle H_{\partial\Omega} g^n_z,T_{\partial\Om}H_{\partial\Omega}g^n_z\rangle_{\partial\Om}-\langle H_{\partial\Omega} g^n_z,T_{\partial\Om}\psi_1\rangle_{\partial\Om} - \langle \psi_1,T_{\partial\Om} (H_{\partial\Omega}g^n_z-\psi_1)\rangle_{\partial\Om} ,
 \end{array}
 \end{equation}
and similarly, 
 \begin{equation}\label{ineq:Estimation2CV+AB}
 \begin{array}{ll}
 & c\, \|\mathcal{H}_\Gamma g_z^n - \psi_3\|^2_{\mH^{-1/2}(\Gamma)} \\[3pt]
\le &  \langle \mathcal{H}_\Gamma g^n_z,\mathcal{T}_{\Gamma}\mathcal{H}_\Gamma g^n_z\rangle_{\Gamma}-\langle \mathcal{H}_\Gamma g^n_z,\mathcal{T}_{\Gamma}\psi_3\rangle_{\Gamma} - \langle \psi_3,\mathcal{T}_{\Gamma} (\mathcal{H}_\Gamma g^n_z-\psi_3)\rangle_{\Gamma}.
\end{array}
 \end{equation}
Adding  \eqref{ineq:Estimation1CV+AB} to \eqref{ineq:Estimation2CV+AB} and using \eqref{limsup+AB}, we deduce that $(H_{\partial\Omega}g^n_z)$ (resp. $(\mathcal{H}_{\Gamma}g^n_z)$)  converges strongly to $\psi_1$ (resp. $\psi_3$) in $\mH^{-1/2}(\partial\Om)$ (resp. $\mH^{-1/2}(\Gamma)$). From estimate (\ref{equivH}),  this implies that $(H^{\rel}g^n_z)$ converges strongly to $\psi$ in $\mY$.  The uniqueness of the limit and the fact that the above convergence arguments hold for any subsequence ensure that the convergence holds for all the sequence.
 \end{proof}

 \subsection{Comparison of two problems revealing the presence of cracks}
 
For $z\in\Om$, we consider the two following boundary value problems
\begin{equation}\label{eq:PGamma+AB}
(\mathscr{P}_z)\ \begin{array}{|rcll}
 \multicolumn{4}{|l}{\mbox{Find }w_z\in \mH^1(\Omega)\mbox{ such that}}\\[4pt]
   \Delta w_z + k^2 w_z &= &0 & \text{in} ~  \Omega \\[2pt]
w_z &= &\Phi_z & \mbox{on } \partial \Omega  
\end{array}\qquad\quad 
(\mathscr{P}^\Gamma_z)\ \begin{array}{|rcll}
 \multicolumn{4}{|l}{\mbox{Find }w_z^\Gamma\in \mH^1(\Omega\setminus\Gammabar)\mbox{ such that }}\\[4pt]
   \Delta w_z^\Gamma + k^2 w_z^\Gamma &= &0 & \text{in} ~  \Omega\setminus\Gammabar \\[2pt]
   \partial_{\nu}^{\pm}w_z^\Gamma &=& 0& \mbox{on } \Gamma\cap\Omega  \\[2pt]
 w_z^\Gamma &= &\Phi_z & \mbox{on } \partial \Omega.  
  \end{array}
\end{equation}
The theorem \ref{th:Converge2FFSolution+AB} guarantees that we have access to $\partial_{\nu}w_z^\Gamma$ on $\partial\Om$ from the knowledge of the far field data though the location of the cracks is unknown. On the other hand, the term $\partial_{\nu}w_z$ on $\partial\Om$ can be computed numerically. The idea of the method below is to compare $\partial_{\nu}w_z^\Gamma$ and $\partial_{\nu}w_z$ to know if there are some cracks inside $\Om$. Of course if $\Gamma\cap\Omega = \emptyset$, the two problems above are the same. If $k^2$ is not an eigenvalue of the Dirichlet problem \eqref{eq:RTEPbNoCrack+AB}, then for all $z\in\Om$, the function $w_z$ is uniquely defined. And we have $\Im m\,w_z=\Im m\,\Phi_z$ in $\Omega$. Indeed, $\Phi_z$ is given by (\ref{FormulaPhi}) and due to the smoothness of $J_0$ for $d=2$ or the smoothness of $x\mapsto \sin(x)/x$ for $d=3$, we observe that we have $\Im m\,\Phi_z\in \mH^1(\Omega)$. This is enough to guarantee that $\Im m\,\Phi_z$ coincides with $\Im m\,w_z$ in $\Om$. Moreover we have the following result.

\begin{lemma}
\label{lem:DLSM+AB}
Assume that $k^2$ is not a RTE nor an eigenvalue of the Dirichlet problem \eqref{eq:RTEPbNoCrack+AB}. Let $w_z$, $w^\Gamma_z$ be the solutions of $(\mathscr{P}_z)$, $(\mathscr{P}^\Gamma_z)$ respectively. If $\Gamma\cap\Omega\neq\emptyset$, then the set 
\[\{z\in\Om\,|\,\partial_{\nu}(\Im m\,w_z^\Gamma)=\partial_{\nu}(\Im m\,\Phi_z)\mbox{ on }\partial\Om\}\]
has empty interior. 
\end{lemma}

\begin{proof}
We prove this result by working by contradiction. Assume that there is a non empty open ball $B\subset\Om$ such that for all $z\in B$, we have $\partial_{\nu}(\Im m\,w_z^\Gamma)=\partial_{\nu}(\Im m\,\Phi_z)$ on $\partial\Om$. In this case, for all $z\in B$, $\Im m\,w_z^\Gamma$ and $\Im m\,\Phi_z$ have the same Cauchy data on $\d\Omega$. From classical results of unique continuation, this implies that we must have $\Im m\,w_z^\Gamma=\Im m\,\Phi_z$ in $\Omega\setminus\Gammabar$. Hence for all $z\in B$, we must have 
\begin{equation}\label{hyp:contradiction+AB}
\partial_{\nu}^{\pm}(\Im m\,\Phi_z) = 0\mbox{ on }\Gamma\cap\Om.
\end{equation} 
We now show that the above result leads to a contradiction. Let $\beta$ be a real valued function of $\mL^2(\Gamma)$. We define $f$ such that 
\begin{equation}\label{def:technicalFunc}
f(z) = \int_\Gamma\beta(x)\d^+_{\nu(x)}\Phi_z(x)\,dx.
\end{equation}
We observe that $f$ solves the Helmholtz equation in $\R^d\setminus\Gammabar$ and satisfies the Sommerfeld radiation condition. From (\ref{hyp:contradiction+AB}), we note also that there holds $\Im m\,f = 0$ in $B$. Then the unique continuation principle implies that $\Im m\,f$ vanishes in $\R^d$. But a real valued solution of the Helmholtz equation which satisfies the Sommerfeld radiation condition is necessarily zero. Hence we have $f=0$ in $\R^d$. This contradicts the classical jump property of the double layer potential which states that $[\d_\nu f] = \beta$ on $\Gamma$. This ends the proof.   
\end{proof}

 \begin{theorem}
 \label{th:DLSM+AB}
Assume that $k^2$ is not a RTE nor an eigenvalue of the Dirichlet problem \eqref{eq:RTEPbNoCrack+AB}.
For $z\in\Omega$, let $(g_z^n)_{n\in\N}$ be a sequence defined via \eqref{def:GLSMconvergenceCondition+AB}. Then we have
\[
\begin{array}{lcl}
\Gamma\cap\Omega = \emptyset\qquad&\Leftrightarrow &\qquad\dsp\lim_{n\to+\infty}\|H_{\d\Omega}g_z^n-\partial^+_{\nu}(w_z - \Phi_z)\|_{\mH^{-1/2}(\partial\Om)} = 0 \mbox{ for  a.e $z$ in $\Omega$}\\[10pt]
\qquad&\Leftrightarrow &\qquad\dsp\lim_{n\to+\infty}\|\Im m\,(H_{\d\Omega}g_z^n)\|_{\mH^{-1/2}(\partial\Om)}=0\mbox{ for a.e $z$ in $\Omega$.}
\end{array}
\]
where $w_z$ is the solution of $(\mathscr{P}_z)$ (see \eqref{eq:PGamma+AB}).
 \end{theorem}

\begin{proof}
Let $w_z^\Gamma$ (resp. $w_z$) be the solution of $(\mathscr{P}^\Gamma_z)$ (resp. ($\mathscr{P}_z$)). Note that $\Gamma\cap\Omega$ can be empty or not. Theorem \ref{th:Converge2FFSolution+AB} implies that $(H_{\d\Omega} g_z^n)$ converges to $\partial^+_{\nu}(w_z^\Gamma - \Phi_z)$ in $\mH^{-1/2}(\partial\Om)$. The result is then a direct consequence of Lemma \ref{lem:DLSM+AB} and of the fact that $w_z^\Gamma=w_z$ when $\Gamma\cap\Om=\emptyset$.
\end{proof}
\noindent As in the previous section, we propose to detect the position of the cracks by sweeping the probed region with a collection of artificial obstacles $\Omega_t = B(t,\rho)$ (the ball of radius $\rho$ centered at $t$). In light of Theorem \ref{th:DLSM+AB}, we define the indicator function 
\begin{equation}\label{eq:CrackIndicator}
\mathcal{J}^n(t) := \int_{\Omega_t}\|H_{\d\Omega_t}g_z^n-\partial^+_{\nu}(w_z - \Phi_z)\|_{\mH^{-1/2}(\partial\Om_t)}\,dz.
\end{equation}
Then for all $t\in\R^d$, we have 
\[
\Gamma\cap\Omega_t = \emptyset \qquad\Leftrightarrow\qquad\lim_{n\to+\infty} \mathcal{J}^n(t) = 0.
\]

\begin{remark}
\label{Rem:contrast}
According to Theorem  \ref{th:DLSM+AB}, we can also define the following indicator function which does not require to compute $w_z$:
\begin{equation}
\label{def:indicJ2}
\tilde{\mathcal{J}}^n(t):= \int_{\Omega_t}\|\Im m\,(H_{\d\Omega_t}g_z^n)\|_{\mH^{-1/2}(\partial\Om)}\,dz.
\end{equation}
Since for all $z\in \Omega$, there holds $\Im m (H_{\d\Omega}g_z^n-\partial^+_{\nu}(w_z - \Phi_z)) = 
\Im m (H_{\d\Omega}g_z^n) \text{ on } \d\Omega$, from the definitions of $\mathcal{J}^n$, $\tilde{\mathcal{J}}^n$, we observe that we always have $\tilde{\mathcal{J}}^n(t)\leq \mathcal{J}^n(t)$ for $t\in \R^d$. Therefore  \textit{a priori} $\tilde{\mathcal{J}}^n$ may reveal less clearly the presence of the crack than $\mathcal{J}^n$.
\end{remark}

\subsection{Numerical results and comparison with the multiple frequencies approach}\label{ParagNum}

Similarly to what has been done in the previous section, we generate far field data from a damaged material (in other words, we work with synthetic data). But this time, we emphasize that we need data at only one single wavenumber $k$. The matrices $F^\delta$ and $F_t^\rel$ are then defined as in \S\ref{ParagrapheNum}. For a given $t\in\R^d$, we compute $g_z^n$ by solving the regularized version of the far field equation \eqref{def:RegCostFunc+AB}. Once again, the parameter $n$ is fit to the noise level according to the Morozov discrepancy principle. Then we compute $H_{\d\Omega_t}g_z^n$ solving the problem \eqref{PbChampTotalFreeSpaceComp} using analytical formulas as in (\ref{DefScaAnalytic}).  Similarly to what has been done in the previous section, we begin with the neat example of the single crack.  In Figure \ref{Fig:DemoFixed}, the crack is recovered with different resolutions using the indicator $\mathcal{J}^n(t)$. Then we take advantage of this simple example to highlight the behavior of the indicator  $\mathcal{J}^n(t)$ in Figures \ref{Fig:FixedSweepNoCrack}-\ref{Fig:FixedSweepCrack}.

\begin{figure}[H]
\centering
\includegraphics[height=4cm]{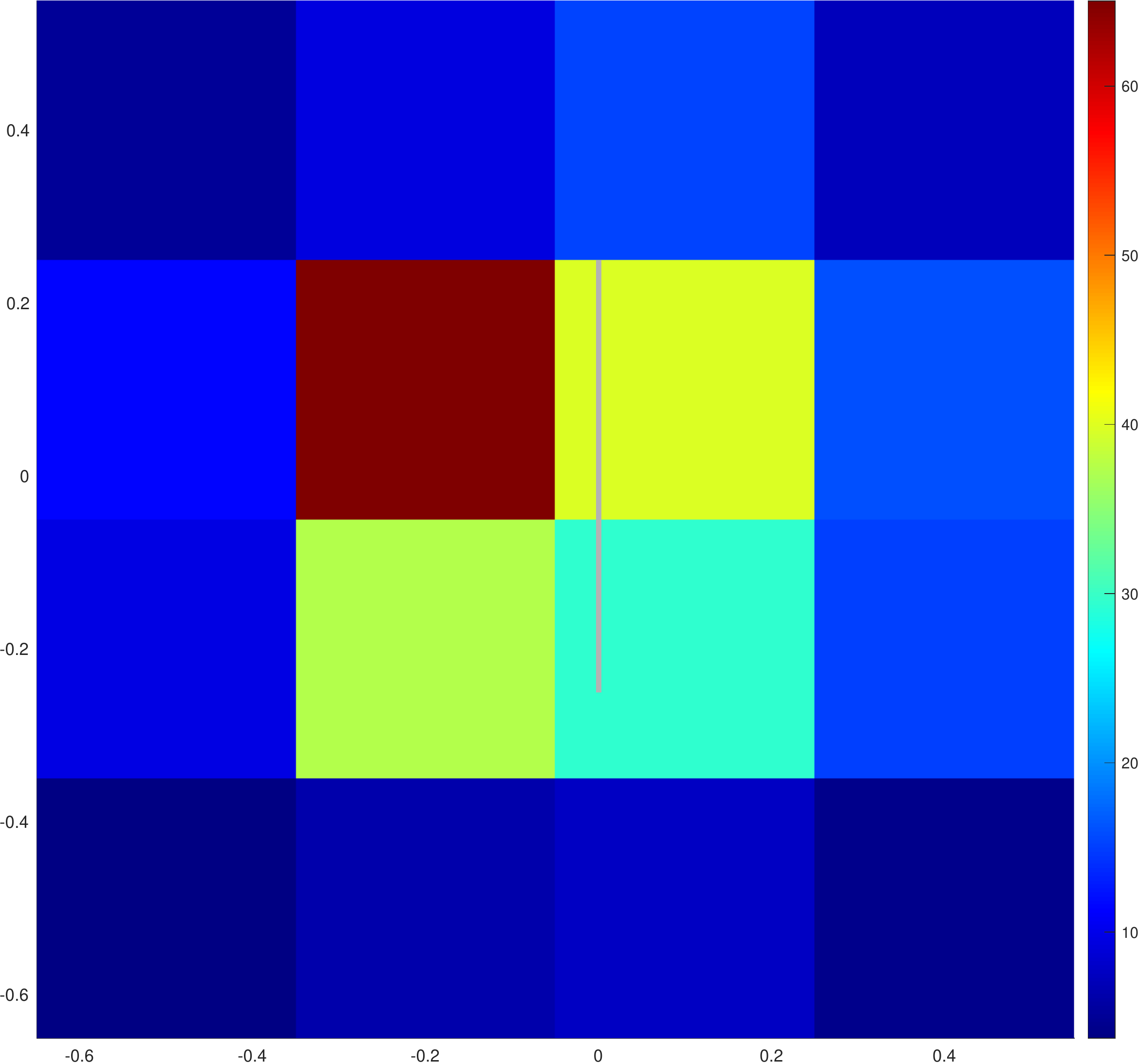}\qquad
\includegraphics[height=4cm]{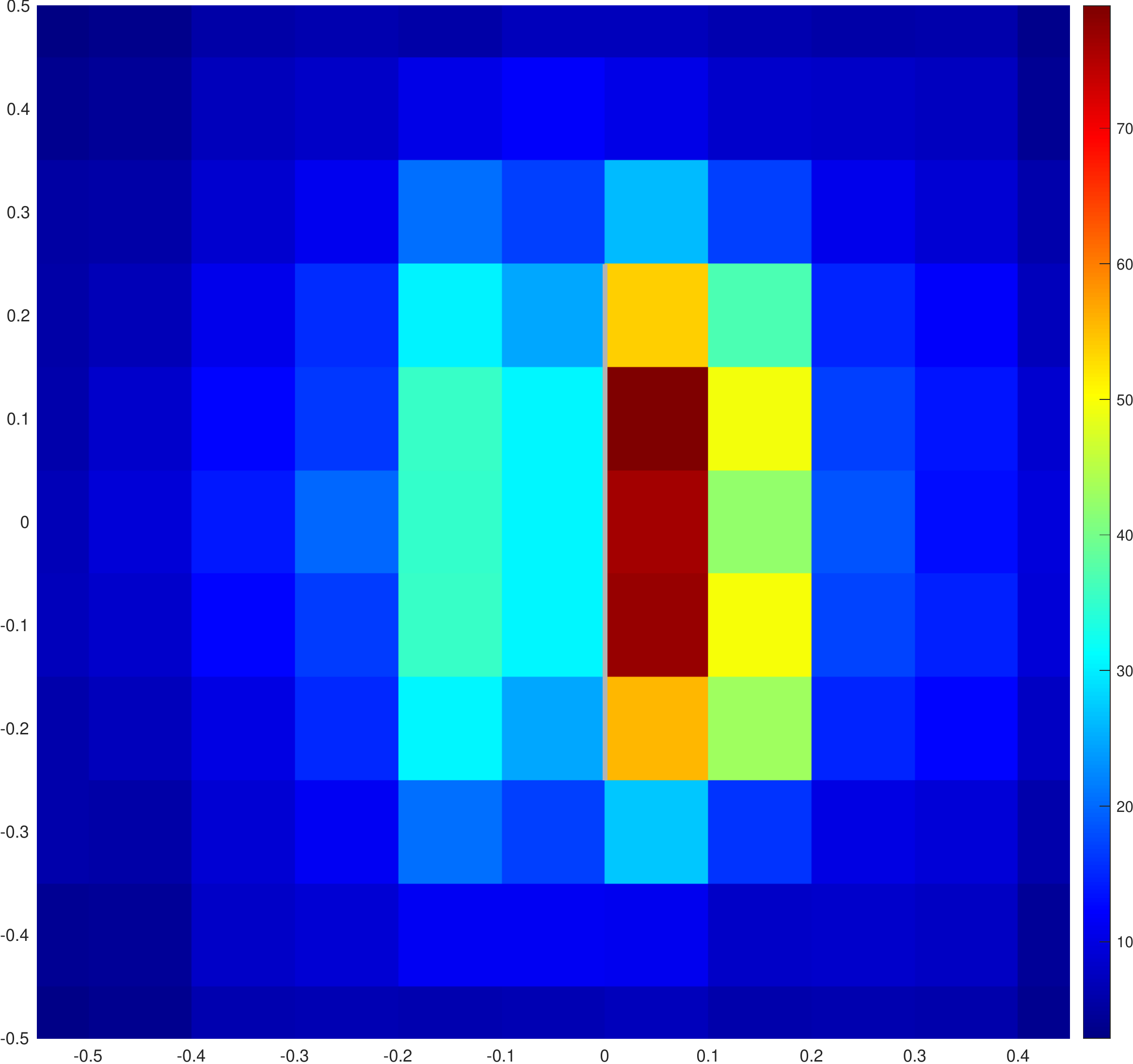}\qquad
\includegraphics[height=4cm]{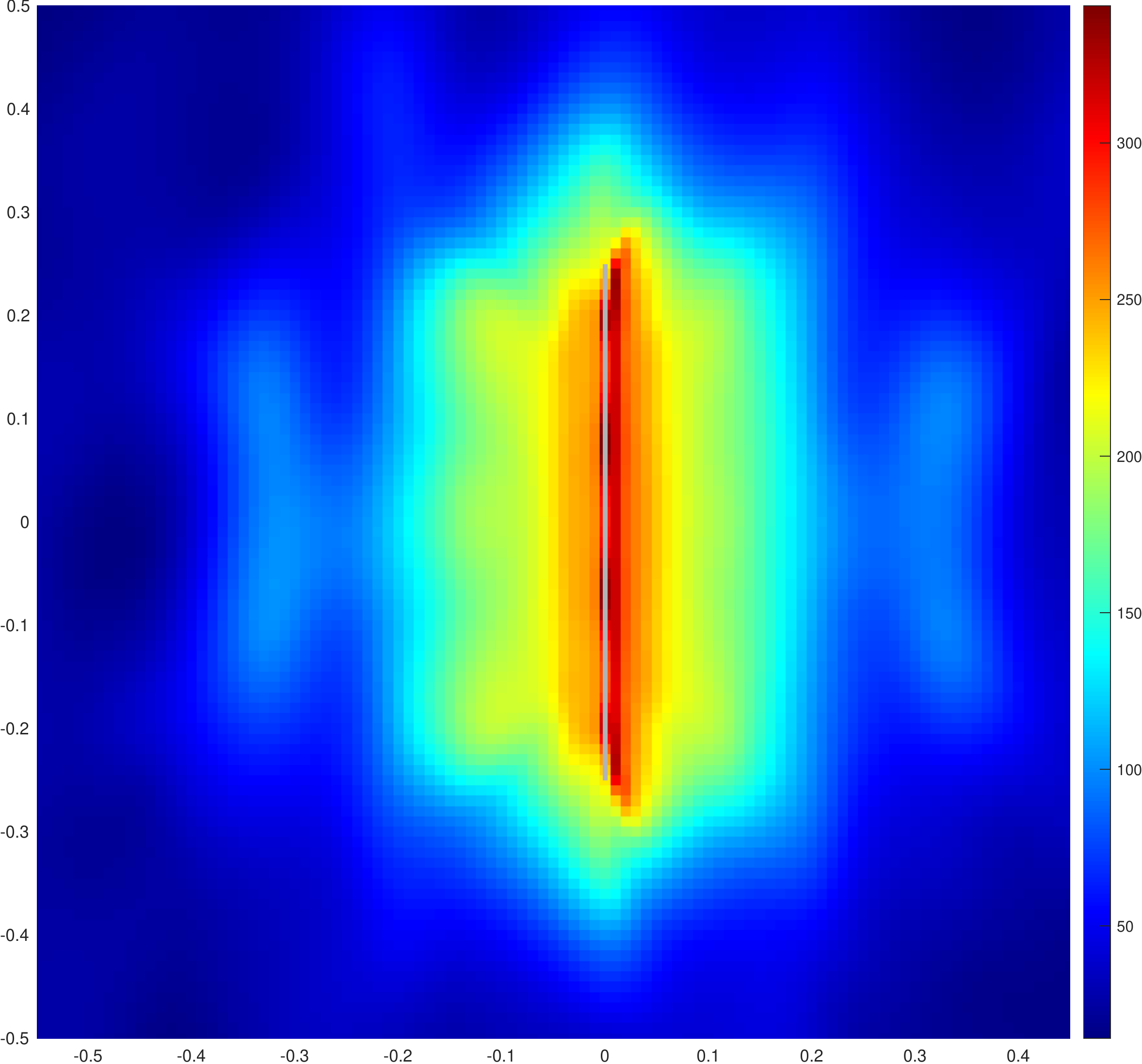}
\caption{Reconstruction of a single crack with the indicator $\mathcal{J}^n$ from given far field data generated at wavenumber $k=15$. The resolution of the image is set by the radius of the used artificial disks: $\rho=0.3$ (left), $\rho=0.1$ (middle), $\rho=0.01$ (right). The data is corrupted with $1\%$ of noise.}
\label{Fig:DemoFixed}                                                                            
\end{figure}

\noindent In  Figure \ref{Fig:FixedSweepNoCrack}, we display the plots of the quantities $H_{\d\Omega_t}g_z^n$ and $\d_\nu(w_z-\Phi_z)$ for a particular $z\in\Omega_t$ where $\Omega_t$ does not intersect the crack. We present the plots of the same quantities when $\Omega_t$ intersects the crack in Figure \ref{Fig:FixedSweepCrack}.

\begin{figure}[H]
\centering
\includegraphics[height=4cm,trim = 14cm 2.5cm 13cm 2cm, clip]{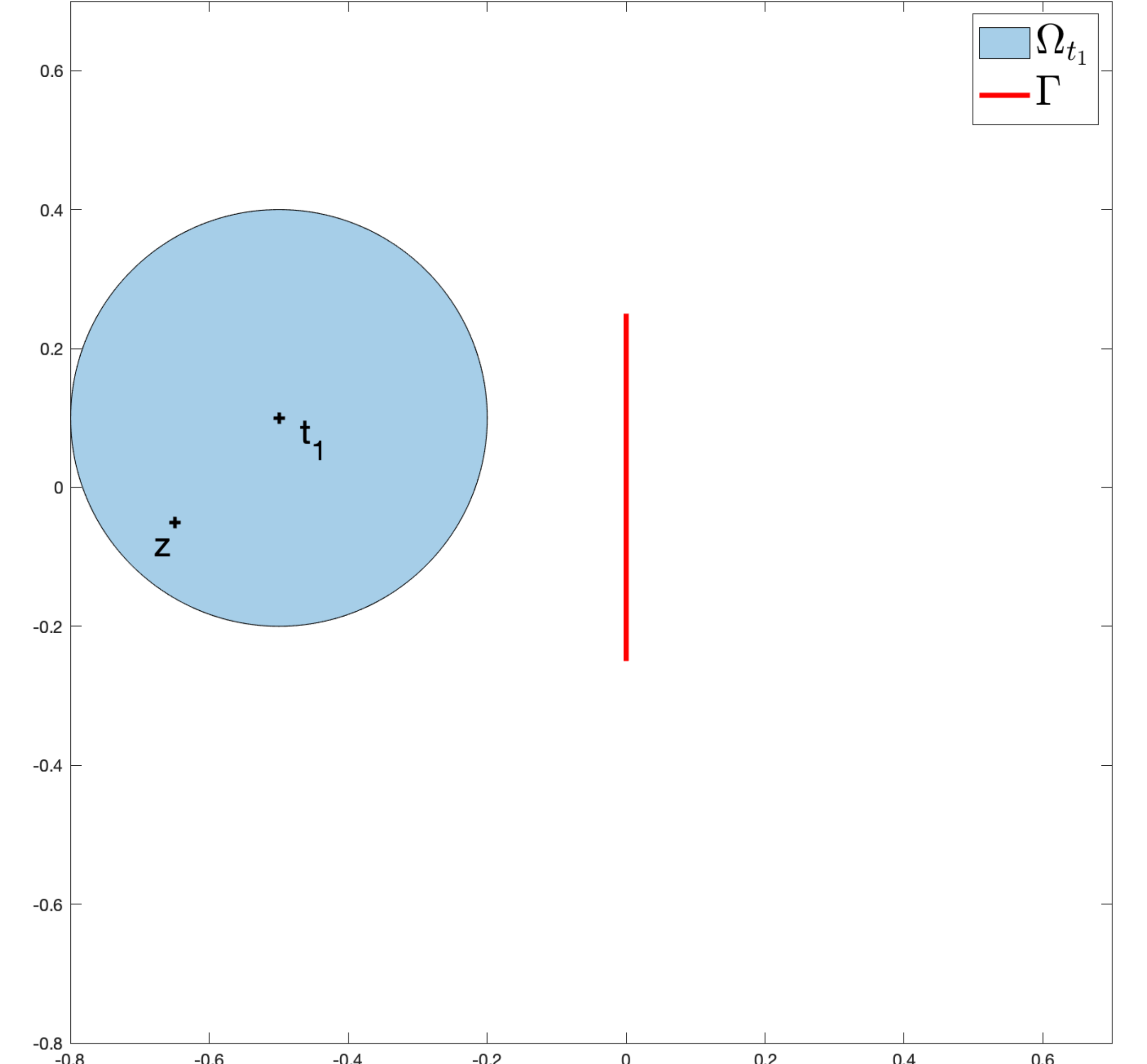}\qquad\qquad
\includegraphics[height=4cm]{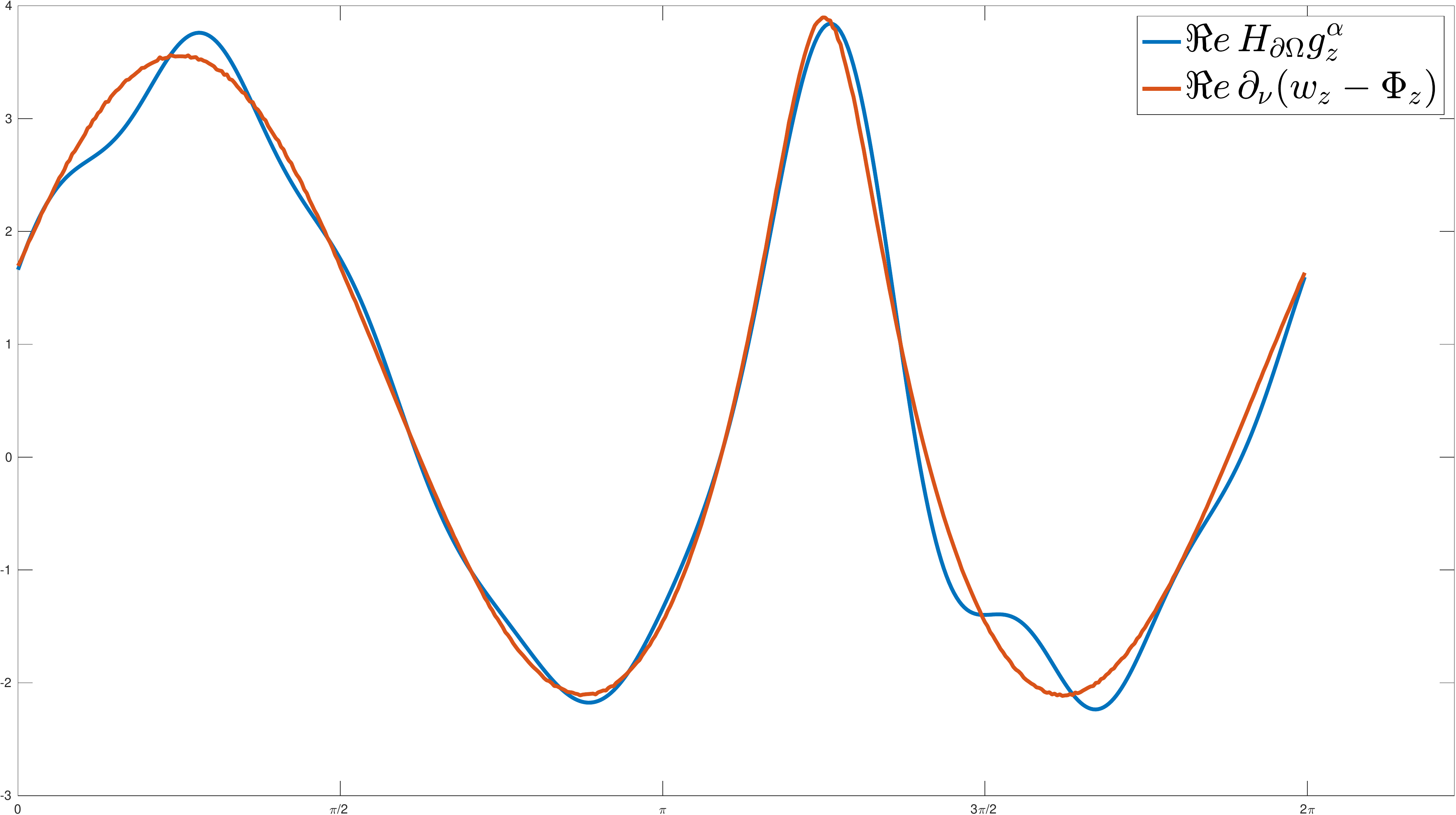}\\[10pt]
\includegraphics[height=4cm]{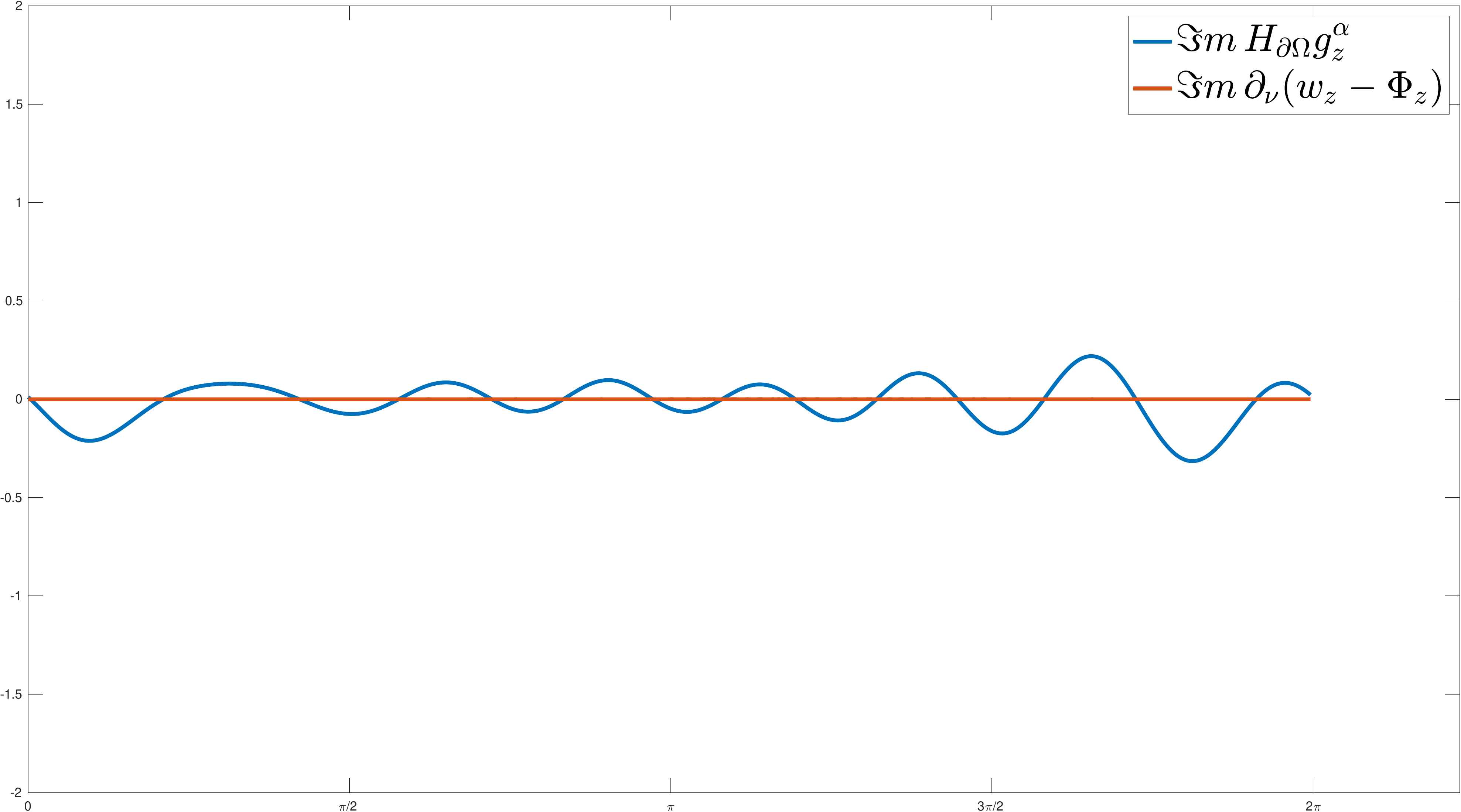}
\caption{Artificial background $\Omega_{t_1}$ and crack $\Gamma$ (top left). Since $\Omega_{t_1}$ does not intersect $\Gamma$, for any $z\in\Omega_{t_1}$, the quantities $H_{\d{\Omega_{t_1}}}g_z^n$ and $\d_\nu(w_z-\Phi_z)|_{\partial\Om_{t_1}}$ are expected to be close in $H^{-1/2}(\d\Omega_{t_1})$. The real and imaginary parts of the two latter quantities for a particular $z\in\Omega_{t_1}$ are displayed respectively on the top right graph and the bottom graph.}  
\label{Fig:FixedSweepNoCrack}  
\end{figure}

\begin{figure}[H]
\centering
\includegraphics[height=4cm,trim = 14cm 2.5cm 13cm 2cm, clip]{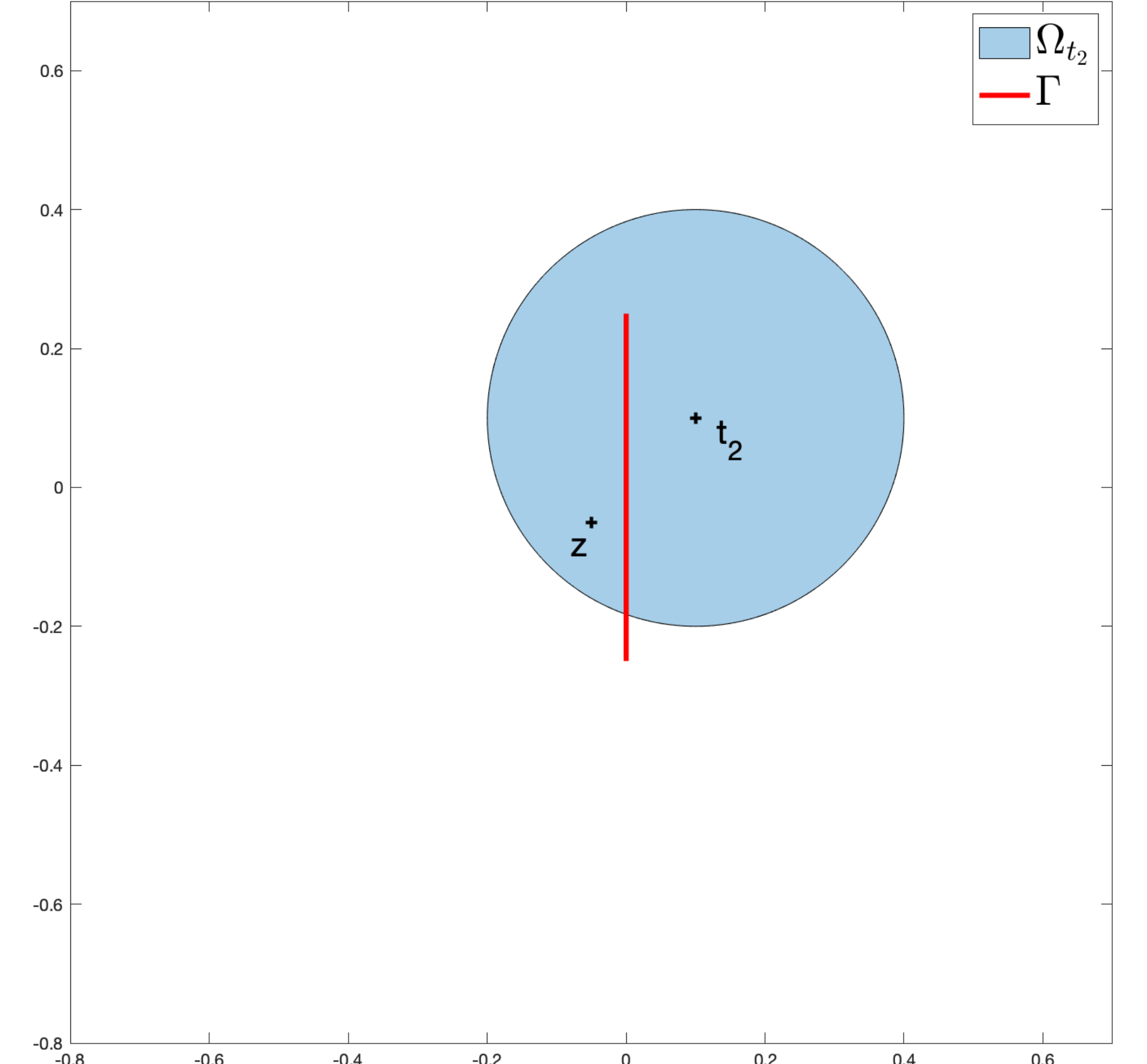}\qquad\qquad
\includegraphics[height=4cm]{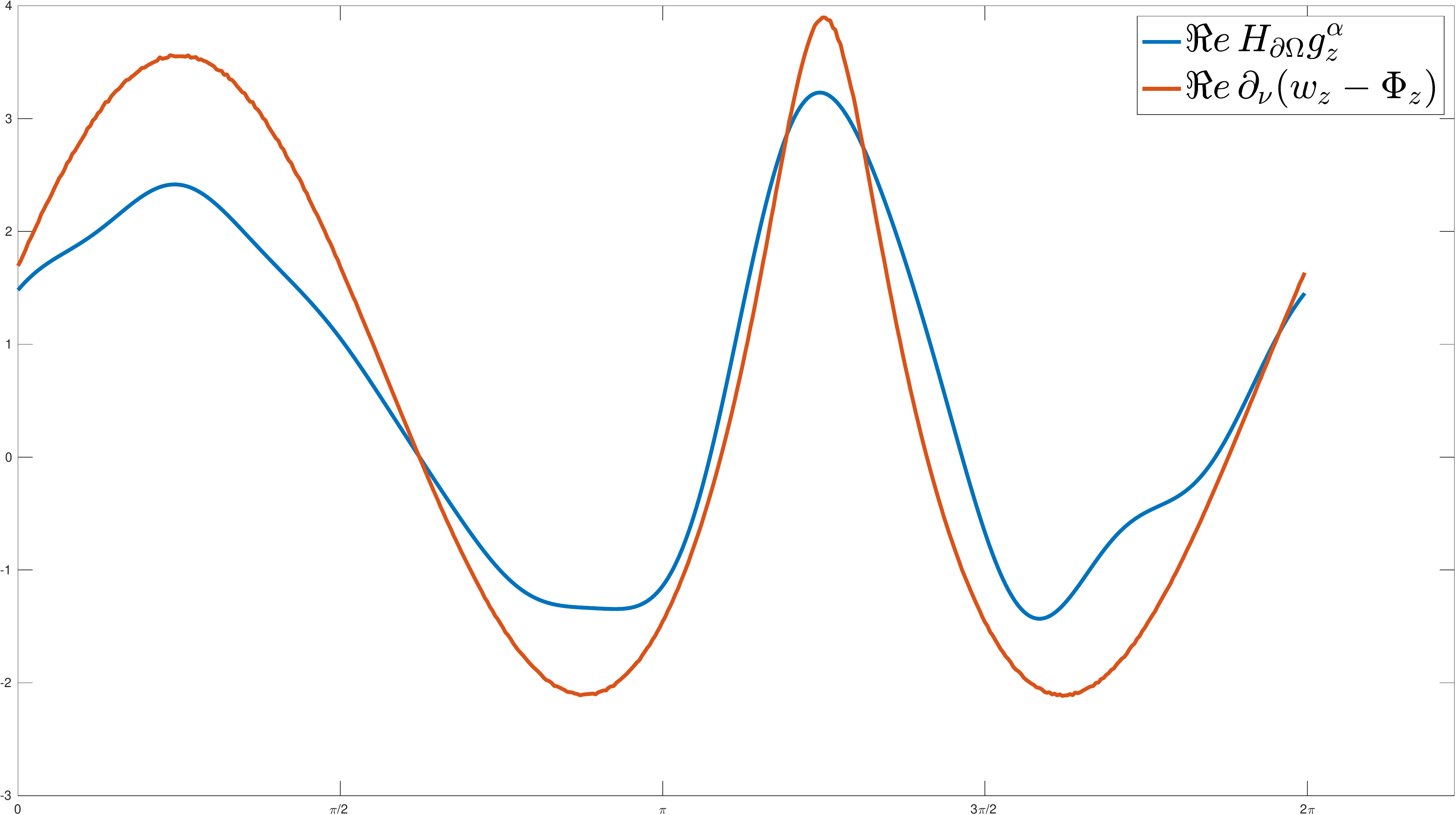}\\[10pt]
\includegraphics[height=4cm]{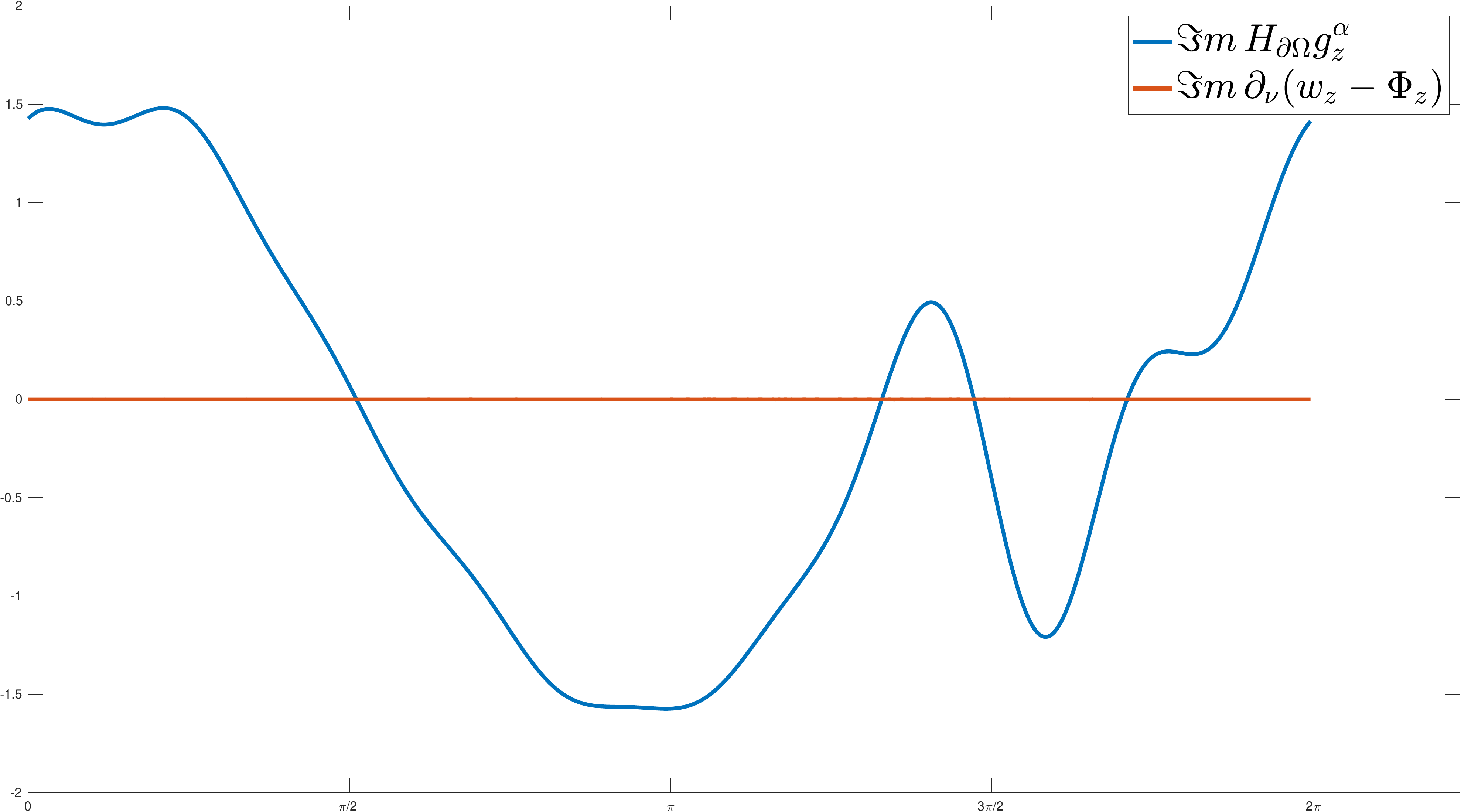}
\caption{Artificial background $\Omega_{t_2}$ and the crack $\Gamma$ (top left). Since $\Omega_{t_2}$ intersects $\Gamma$, in general for $z\in\Omega_{t_2}$, the quantities $H_{\d{\Omega_{t_2}}}g_z^n$ and $\d_\nu(w_z-\Phi_z)|_{\partial\Om_{t_2}}$ are expected to be different in $H^{-1/2}(\d\Omega_{t_2})$. The real and imaginary parts of the two latter quantities for a particular $z\in\Omega_{t_2}$ are displayed respectively on the top right graph and the bottom graph.}  
\label{Fig:FixedSweepCrack}      
\end{figure}

\noindent Since the implementation of the method is quite fast, we have been able to image the damaged materials considered in Figure \ref{fig:MultiFrq} with a higher resolution. The results are presented in Figures \ref{Fig:FixedStraightHR}-\ref{Fig:FixedRandomHR} where we also provide the images obtained with the indicator $\tilde{\mathcal{J}}^n$ defined in \eqref{def:indicJ2}. In Figure \ref{Fig:FixedRandomHR}, we observe that the indicator $\mathcal{J}^n$ offers a better contrast than $\tilde{\mathcal{J}}^n$. This is in agreement with Remark \ref{Rem:contrast}.
 
\begin{figure}[H]
\centering
\includegraphics[height=5cm]{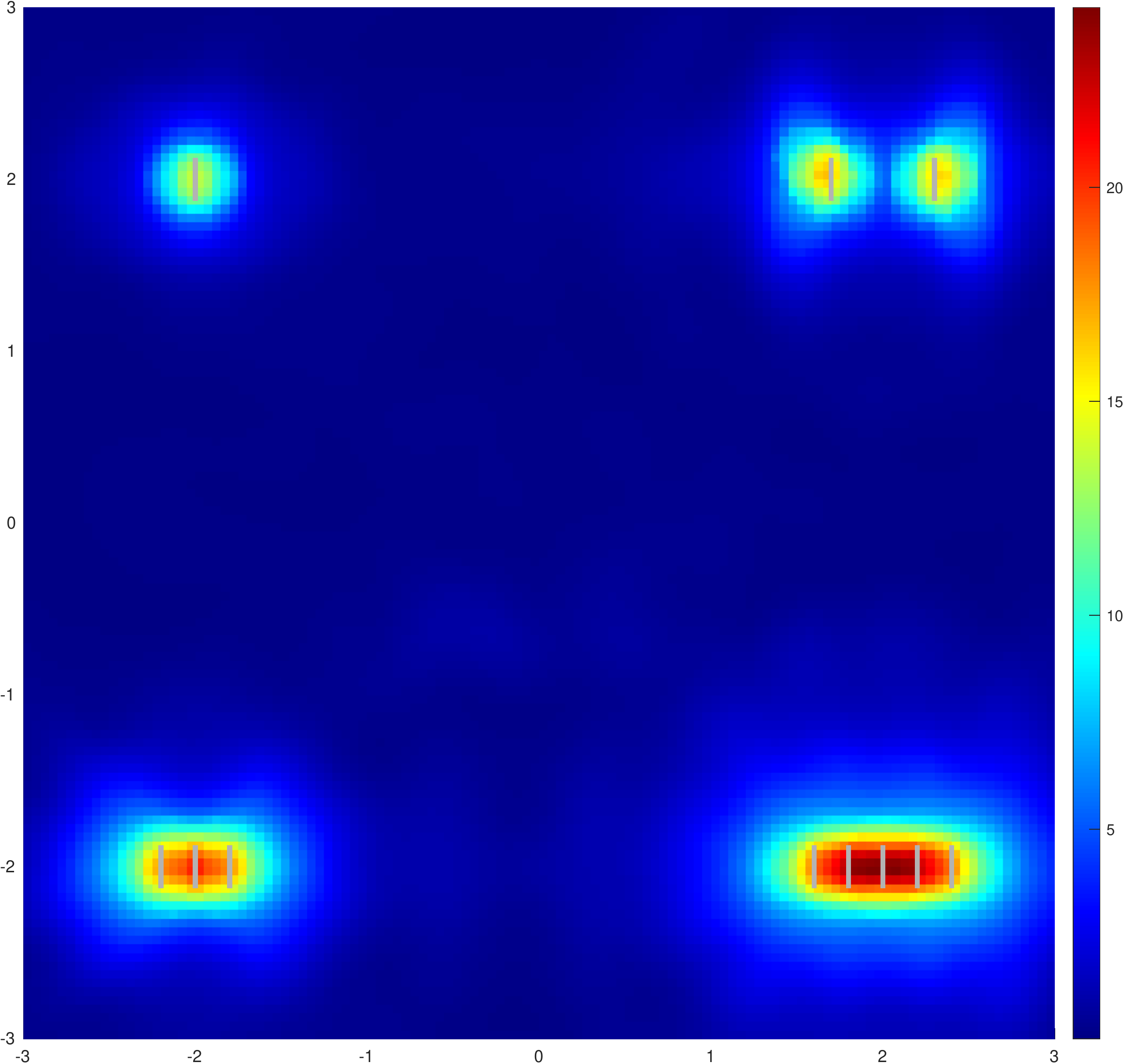}\qquad
\includegraphics[height=5cm]{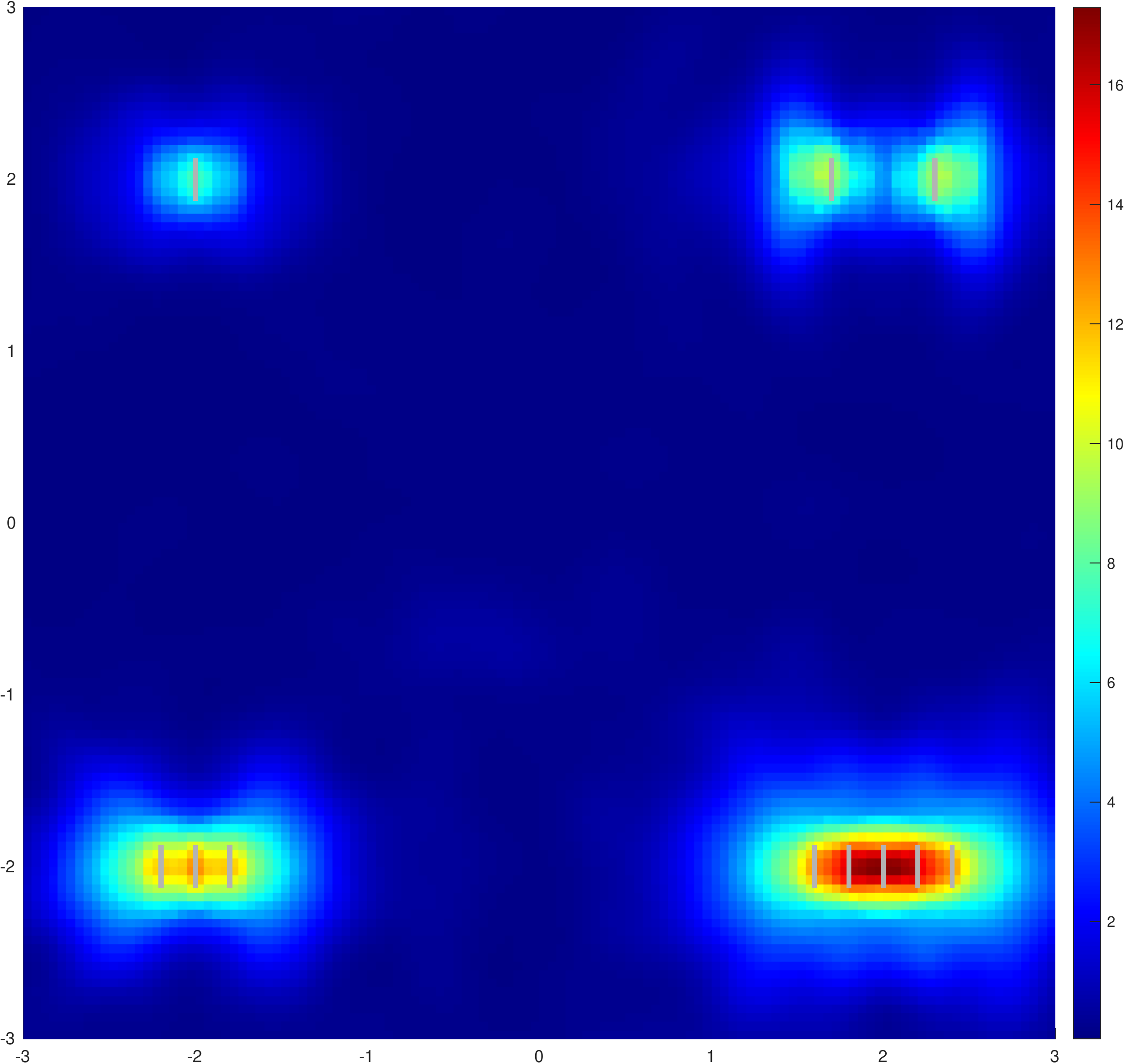}
\caption{Indicators $\mathcal{J}^n$ (left) and  $\tilde{\mathcal{J}}^n$ (right) to image a simulated damaged background with 11 vertical cracks of length 0.25 arranged in 4 areas with different damage levels. The radius of the artificial obstacles is  $\rho=0.1$. The data is corrupted with $1\%$ of noise. } 
\label{Fig:FixedStraightHR}     
\end{figure}

\begin{figure}[htpb]
\centering
\includegraphics[height=5cm]{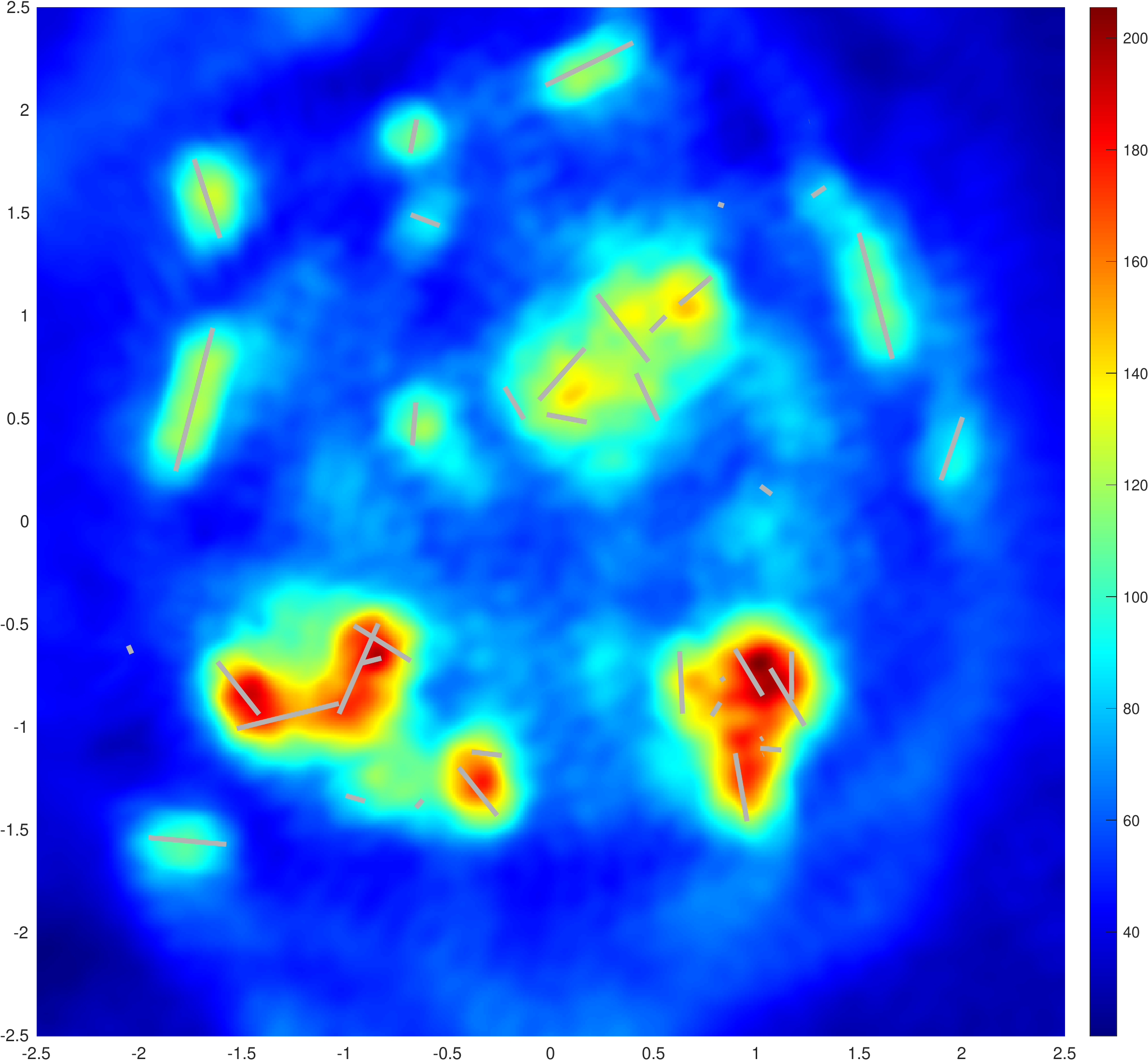}\qquad
\includegraphics[height=5cm]{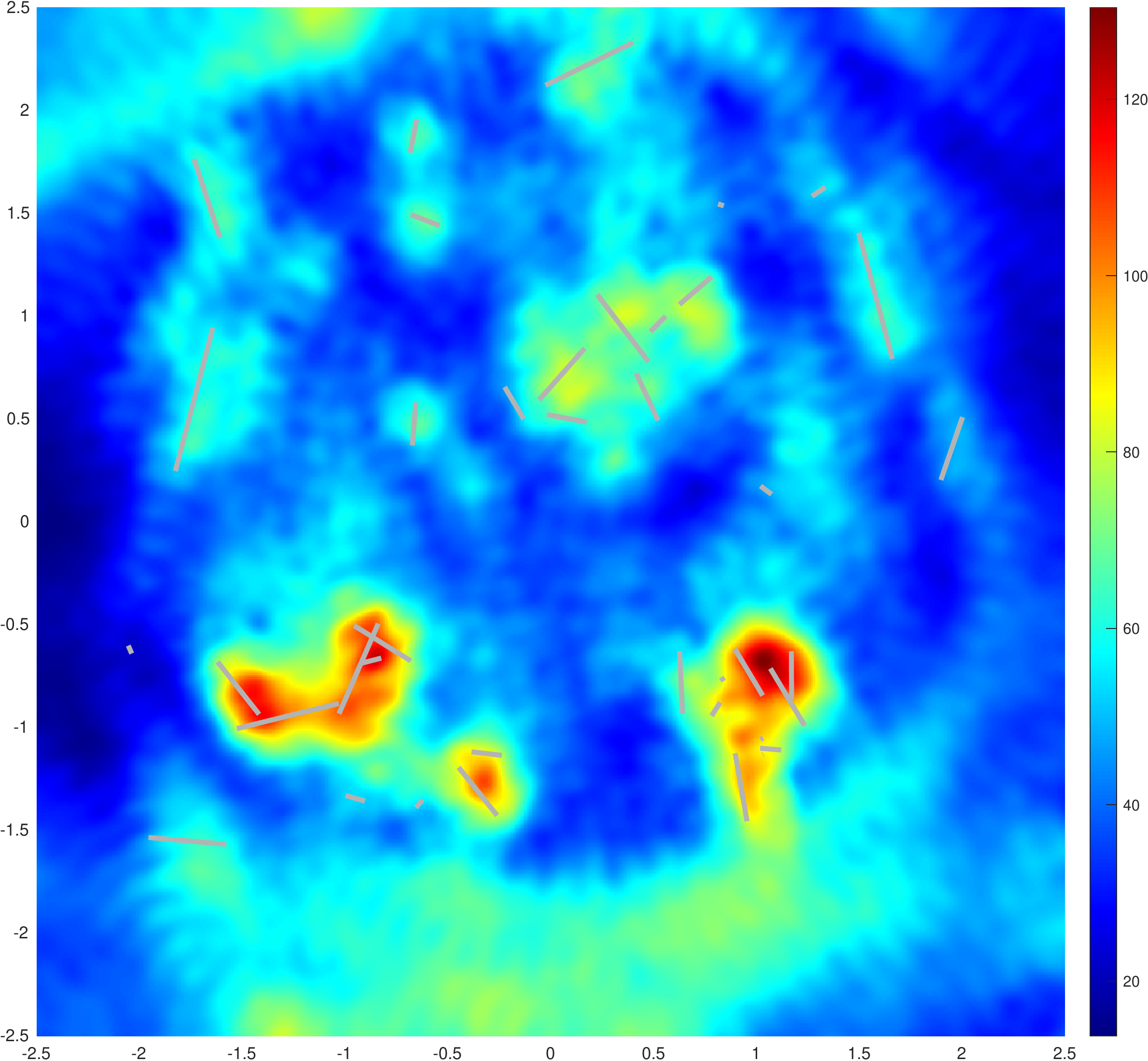}
\caption{Indicators $\mathcal{J}^n$ (left) and  $\tilde{\mathcal{J}}^n$ (right) to image a simulated damaged background with 40 cracks of different lengths arranged randomly.  The radius of the artificial obstacles is  $\rho=0.1$. The data is corrupted with $1\%$ of noise.} 
\label{Fig:FixedRandomHR}         
\end{figure}

\noindent Finally in Figure \ref{Fig:MonoDistrib7}, we compare three different indicators, namely $\mathcal{I}^n$, $\mathcal{J}^n$ and the one of the classical FM on a series of increasingly damaged materials (new born cracks are progressively added). For these examples, the data used to compute $\mathcal{J}^n$ and the FM indicator were generated at the same wavelength $\lambda = 0.15$. The distance $dz$ between sampling points when computing the FM indicator is equal to the radius of the artificial obstacles used to construct $\mathcal{J}^n$: $\rho=dz=0.01$. For $\mathcal{I}^n$, we used data generated for a sample of wavelengths between $\lambda_{min} = 0.15$ and $\lambda_{max} = 0.42$. For this indicator, the radius of the artificial obstacles is set to $\rho'=0.1$. With this setting, we have two eigenvalues in $\sigma_{\emptyset}(\Om_t)$ (see after (\ref{EigCurve})). Note that in the representation of $\mathcal{I}^n$, the colormap changes from one line to another. Figure \ref{Fig:MonoDistrib7} shows the superiority of the indicator function $\mathcal{J}^n$  when the  network is relatively sparse. When the crack network becomes  dense, only  $\mathcal{I}^n$ provides an indicator function that shows variations with respect to local densities of cracks.

\begin{figure}[htpb]
\centering
\includegraphics[height=3.8cm]{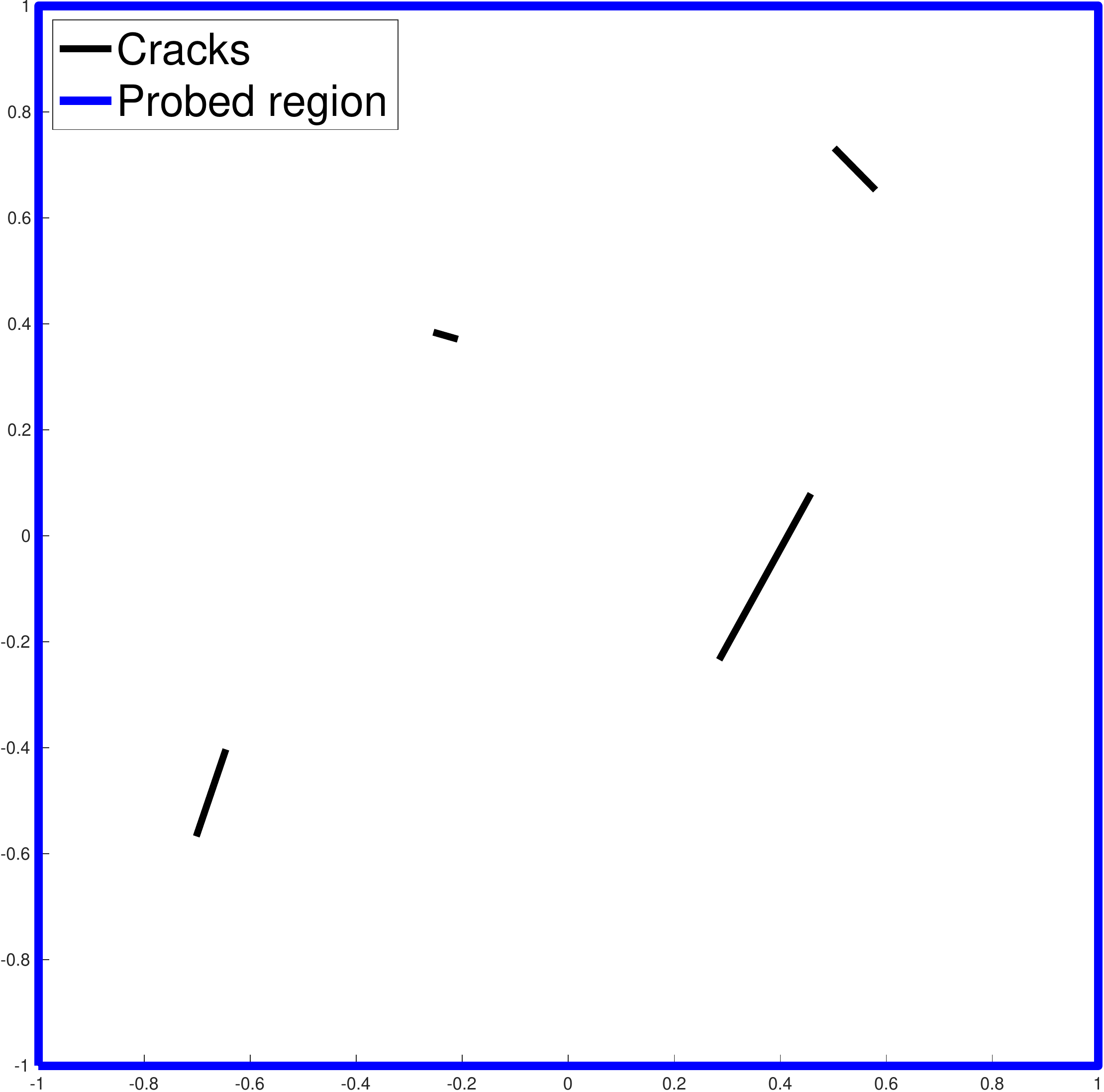}\ 
\includegraphics[height=3.8cm]{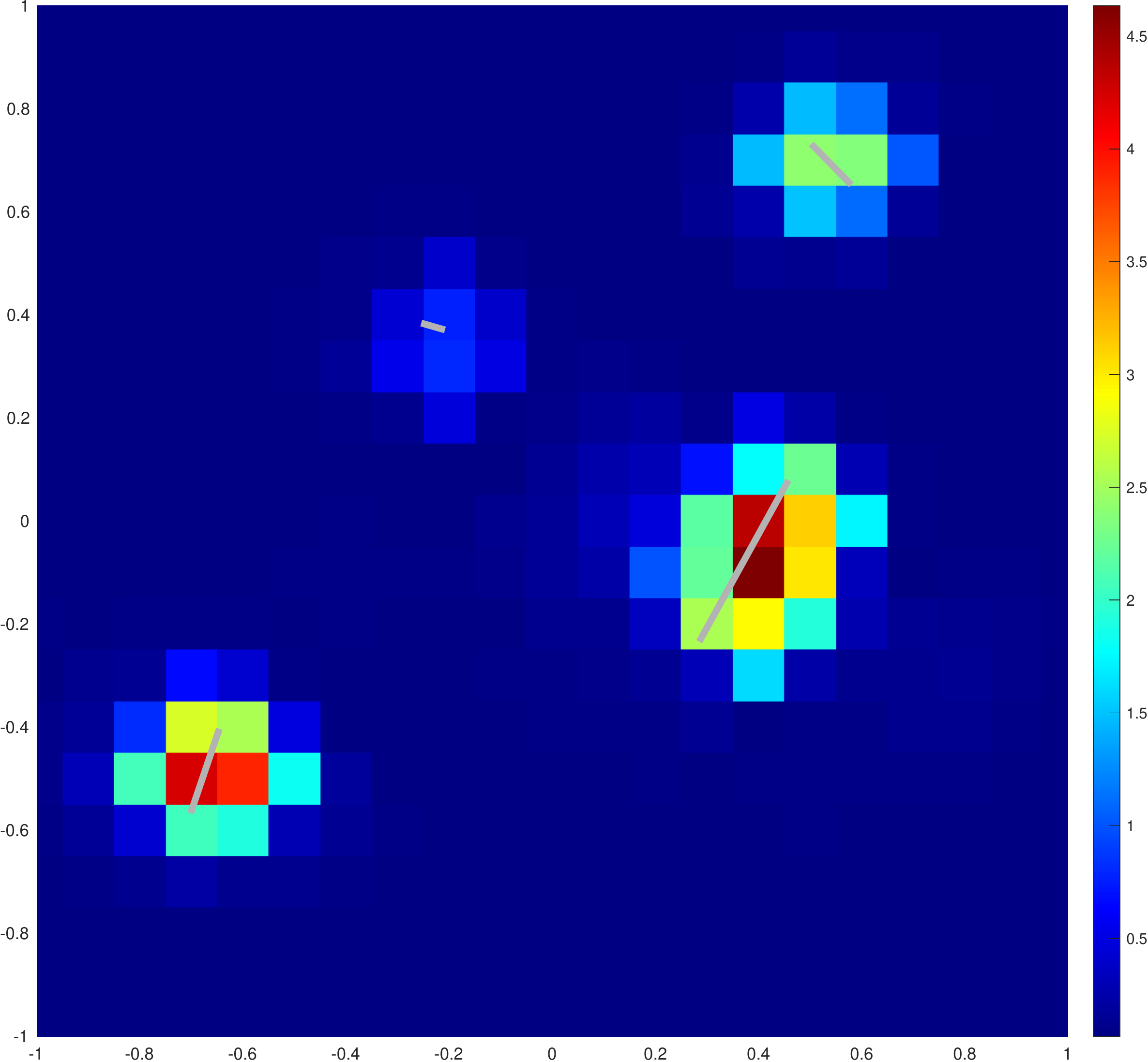}
\includegraphics[height=3.8cm]{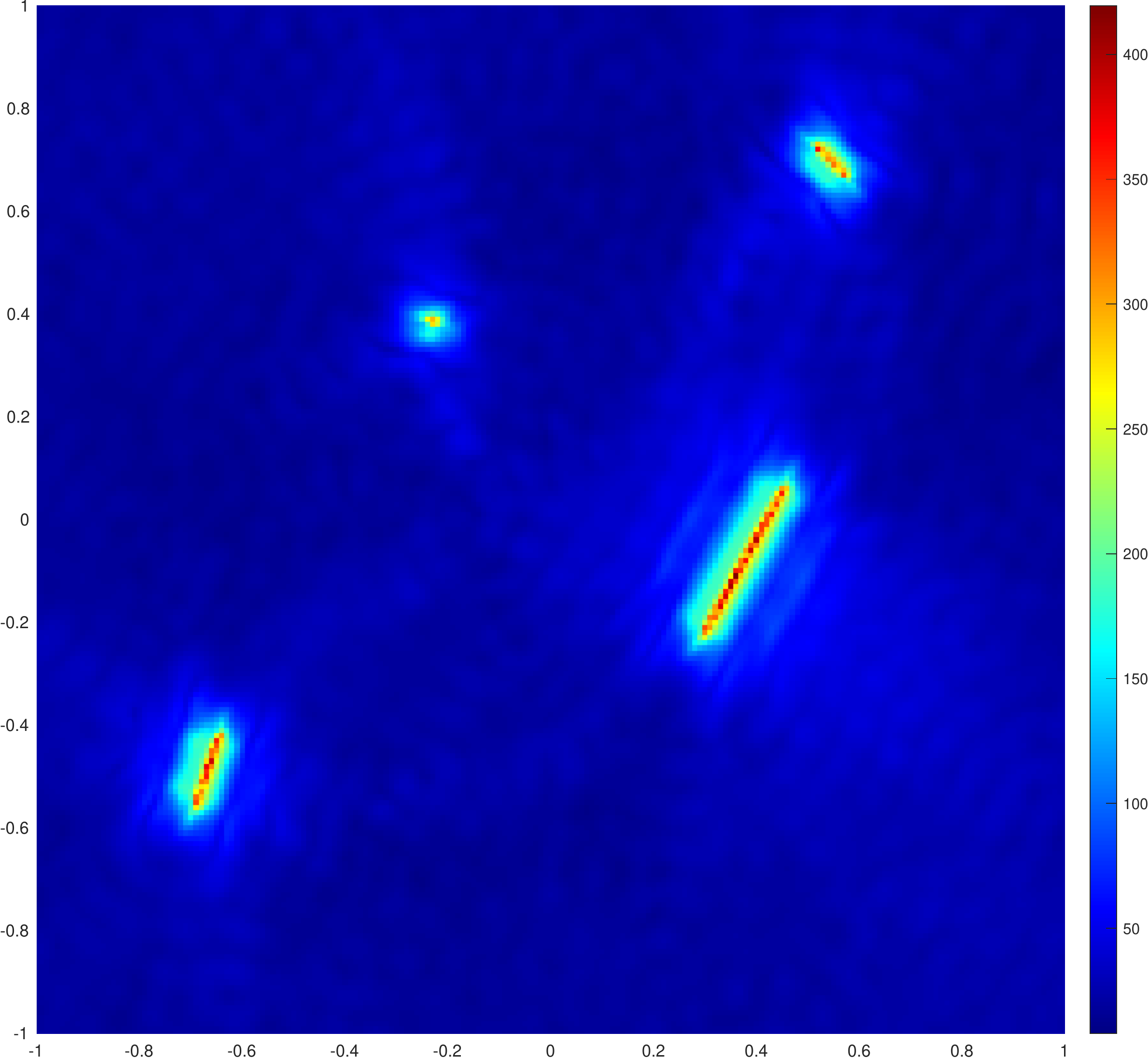}
\includegraphics[height=3.8cm]{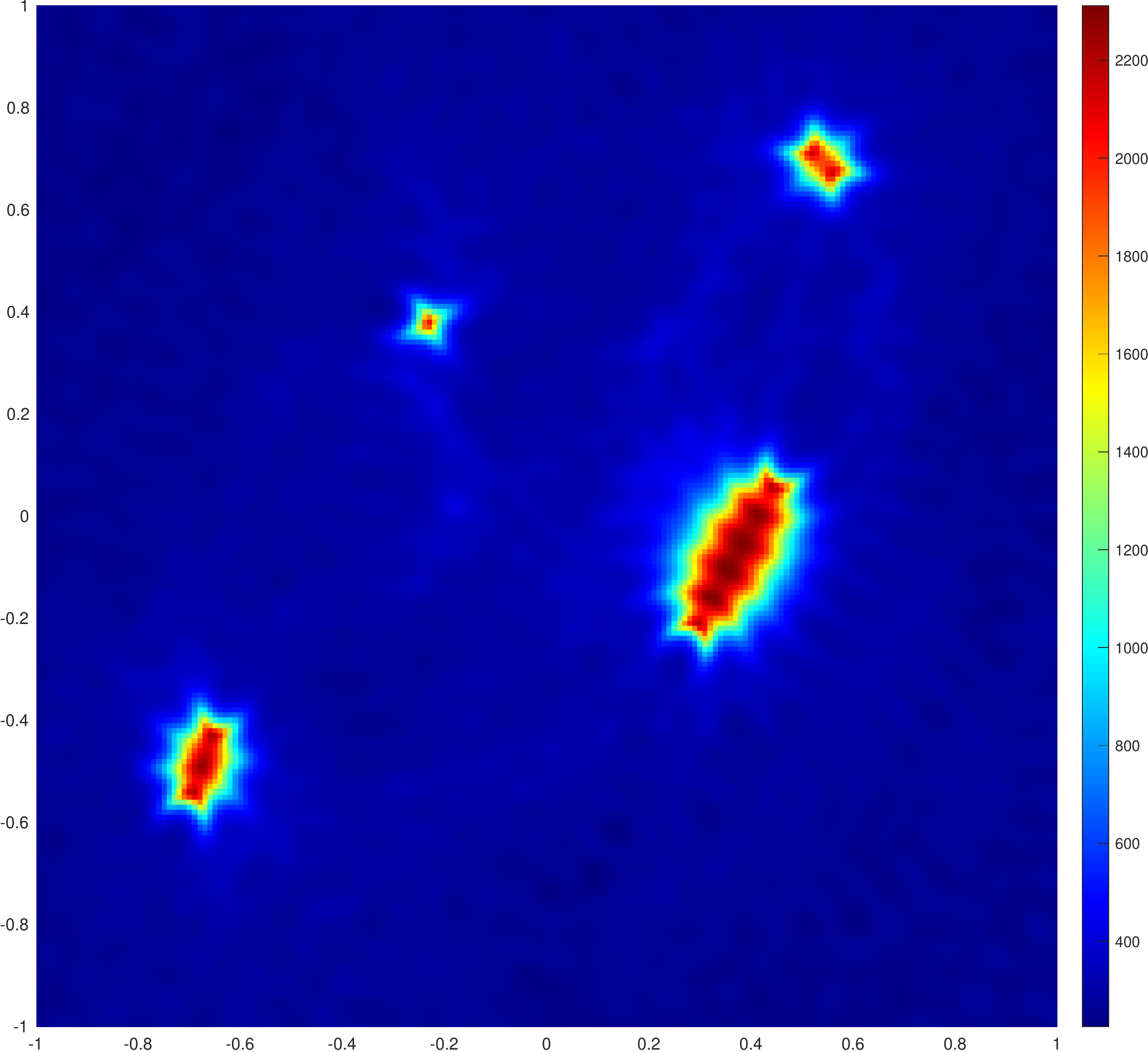}
\centering
\includegraphics[height=3.8cm]{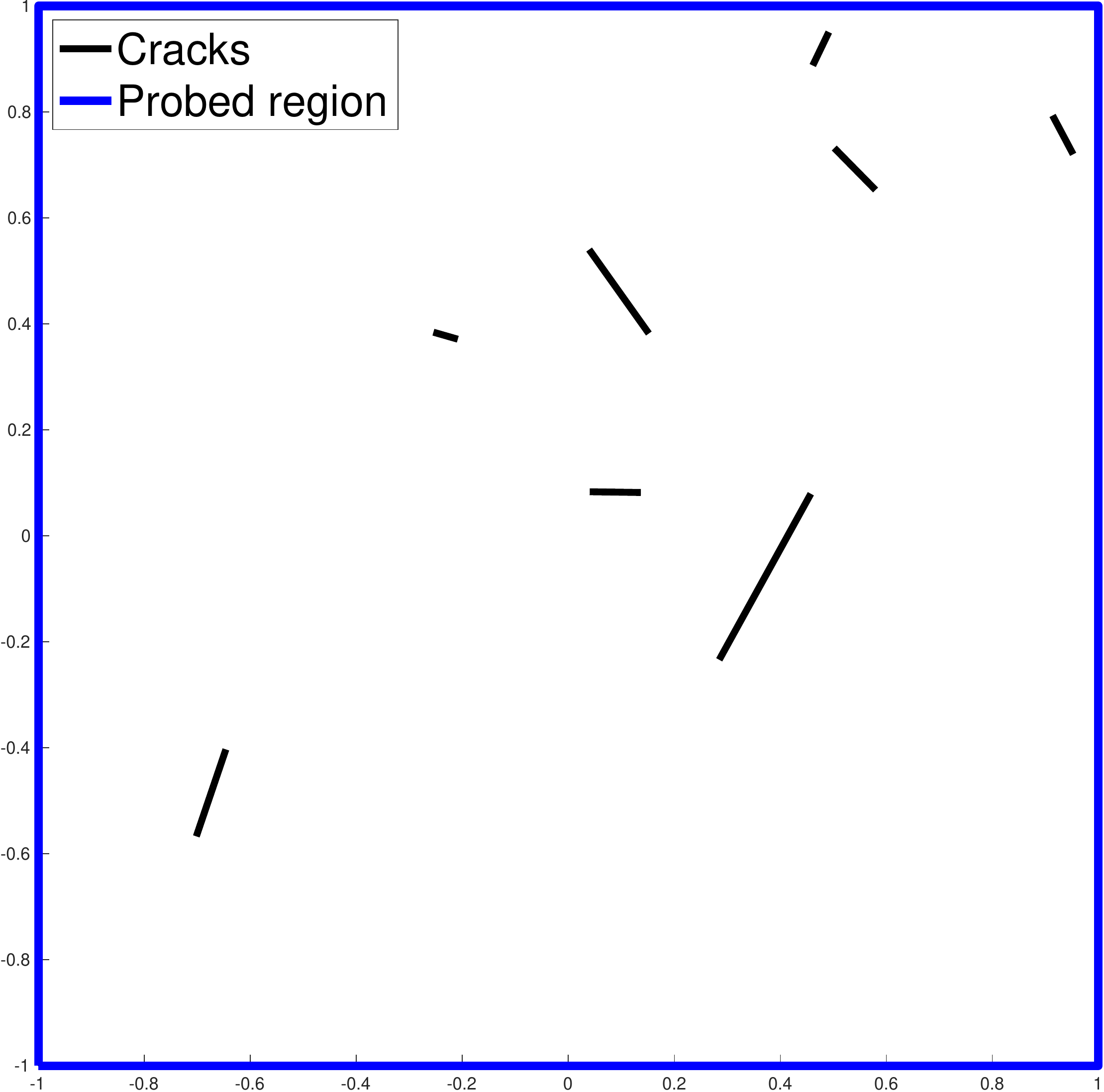}\ 
\includegraphics[height=3.8cm]{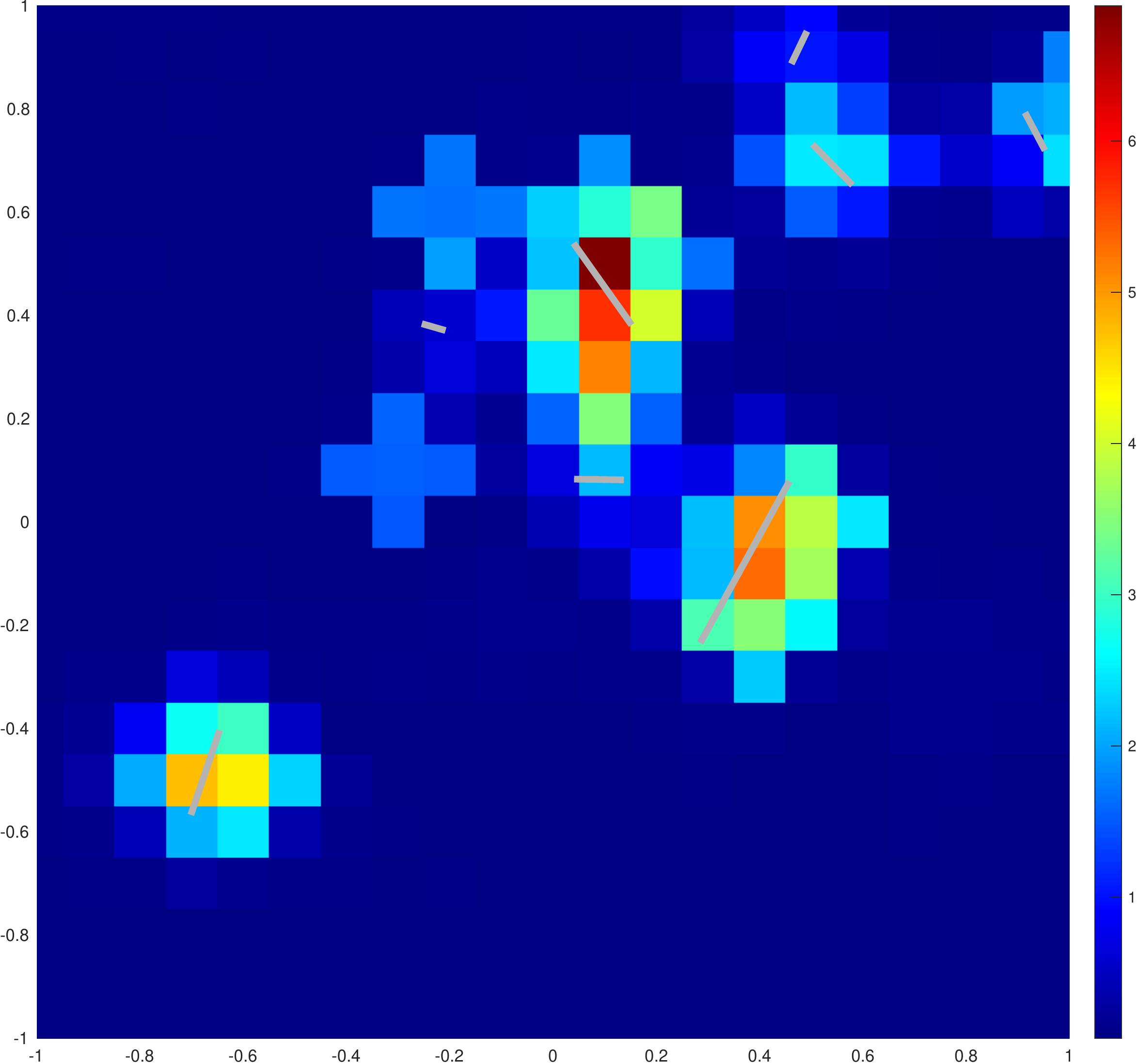}
\includegraphics[height=3.8cm]{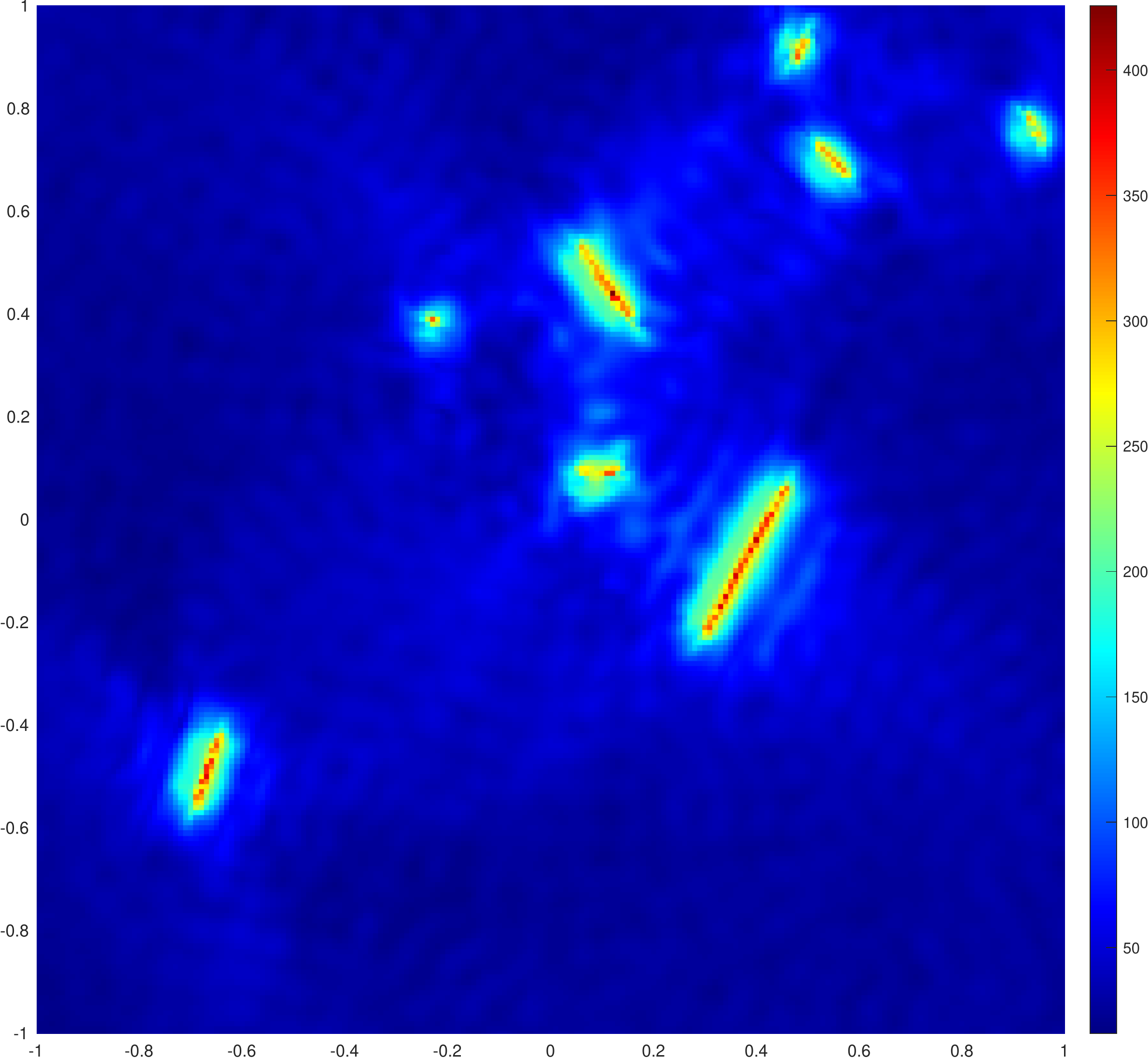}
\includegraphics[height=3.8cm]{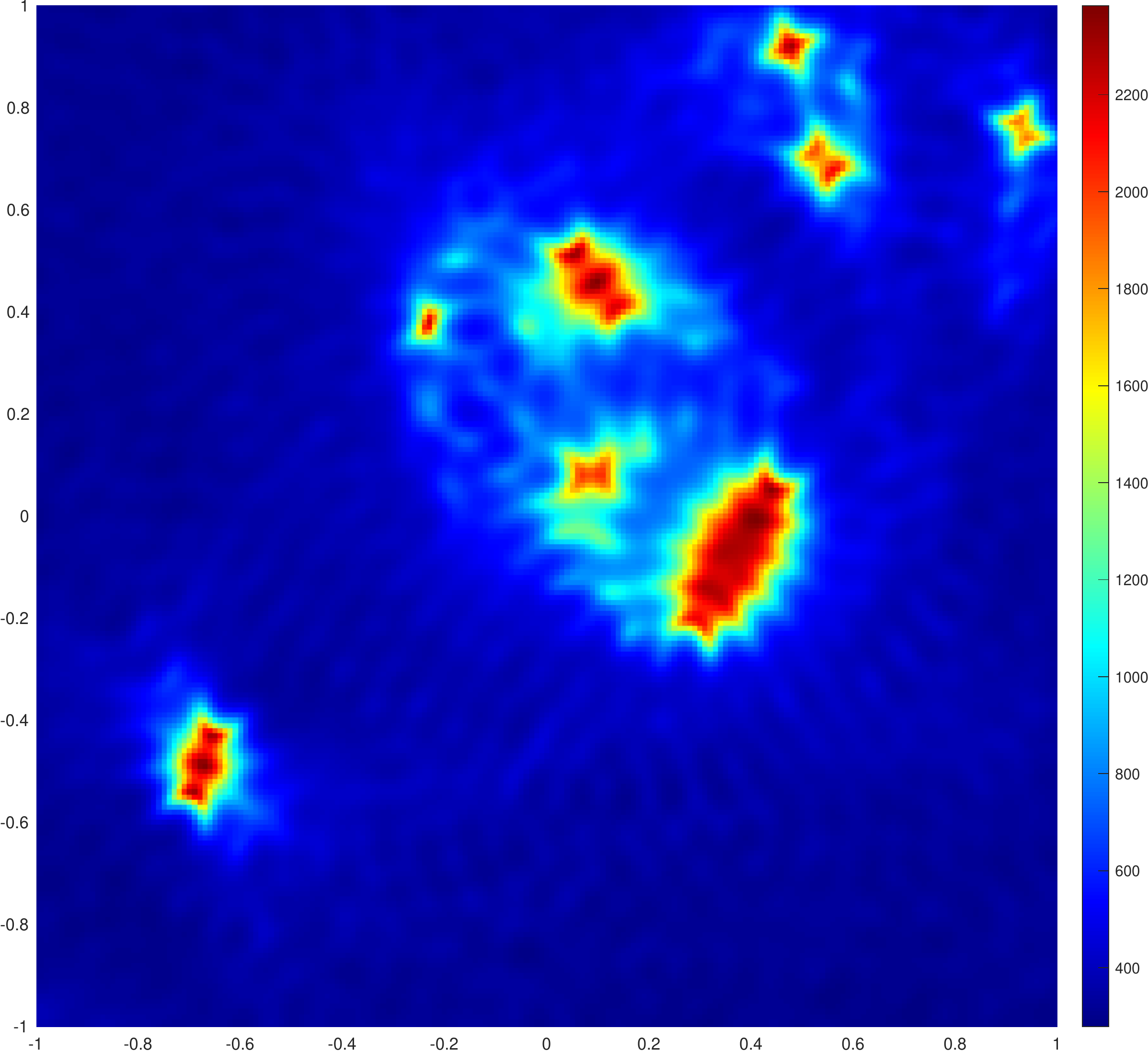}
\centering
\includegraphics[height=3.8cm]{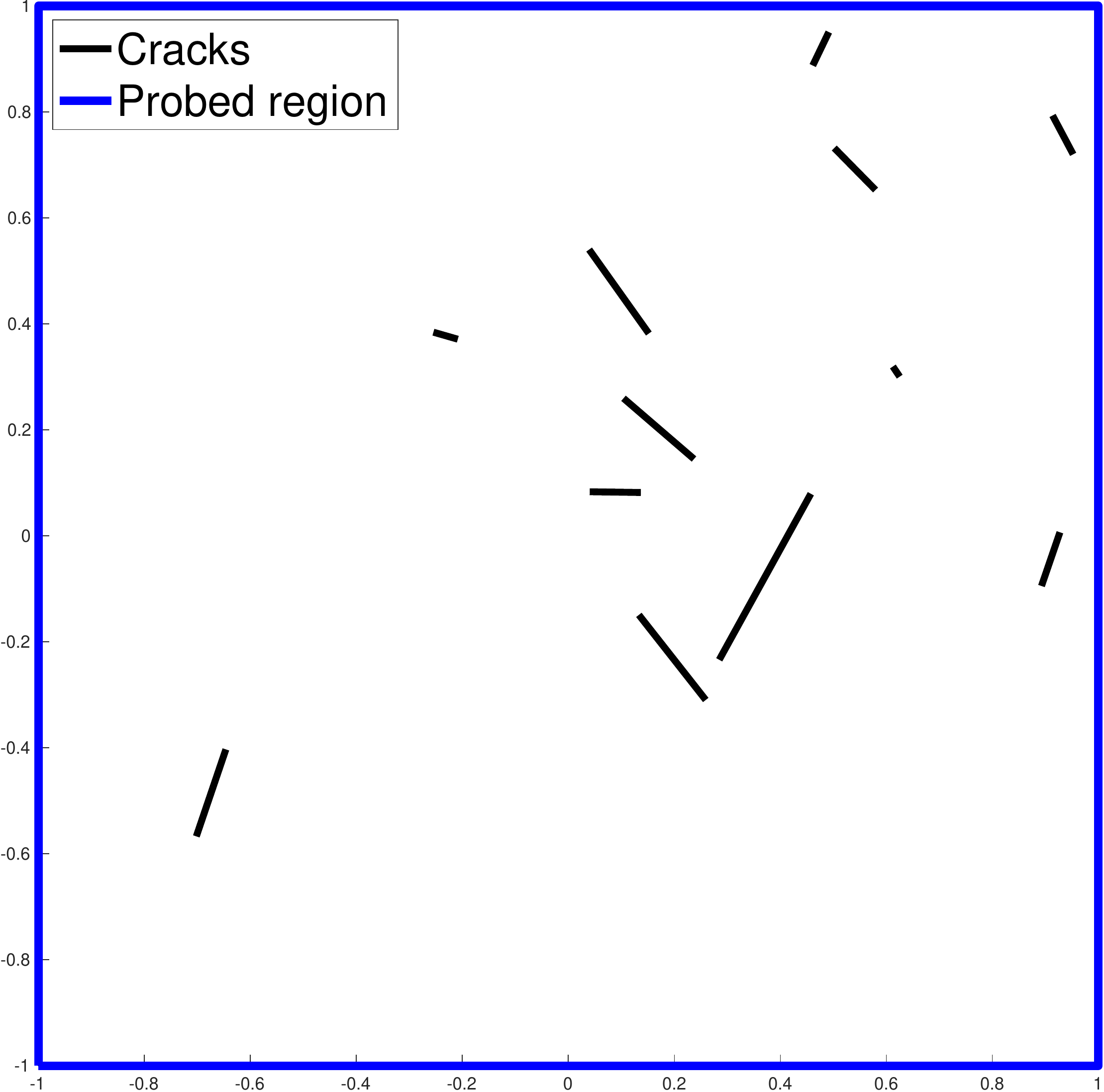}\ 
\includegraphics[height=3.8cm]{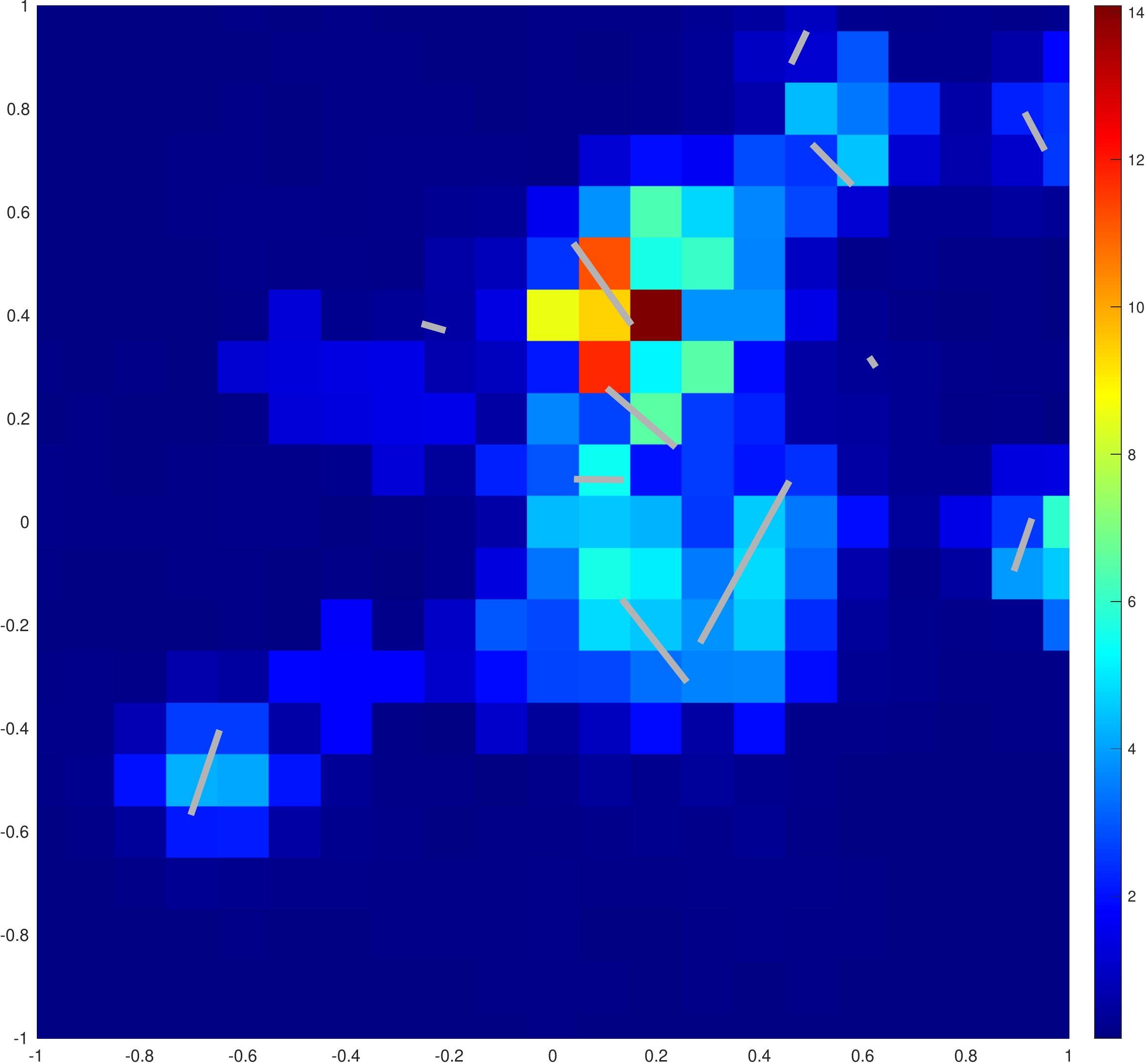}
\includegraphics[height=3.8cm]{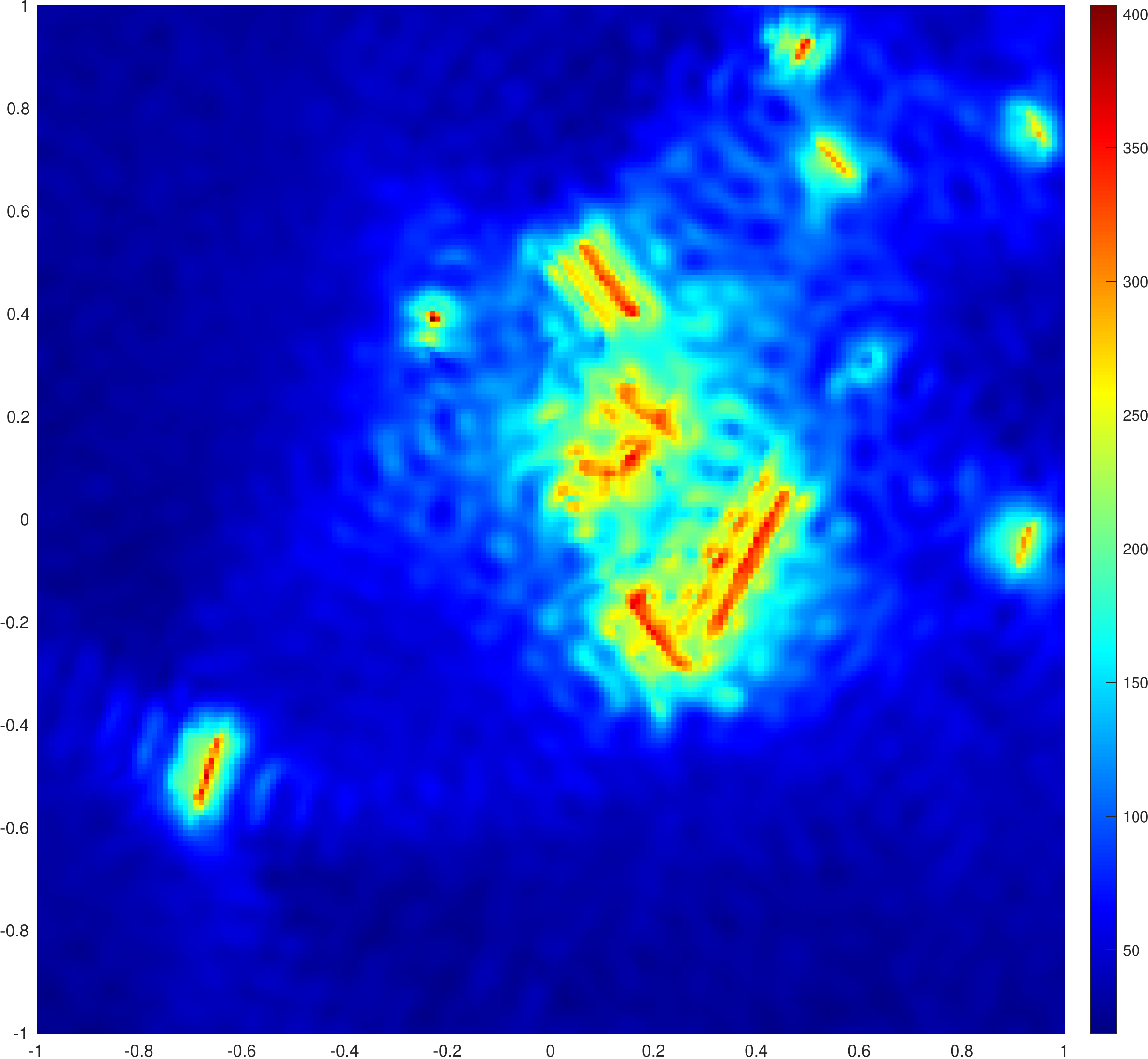}
\includegraphics[height=3.8cm]{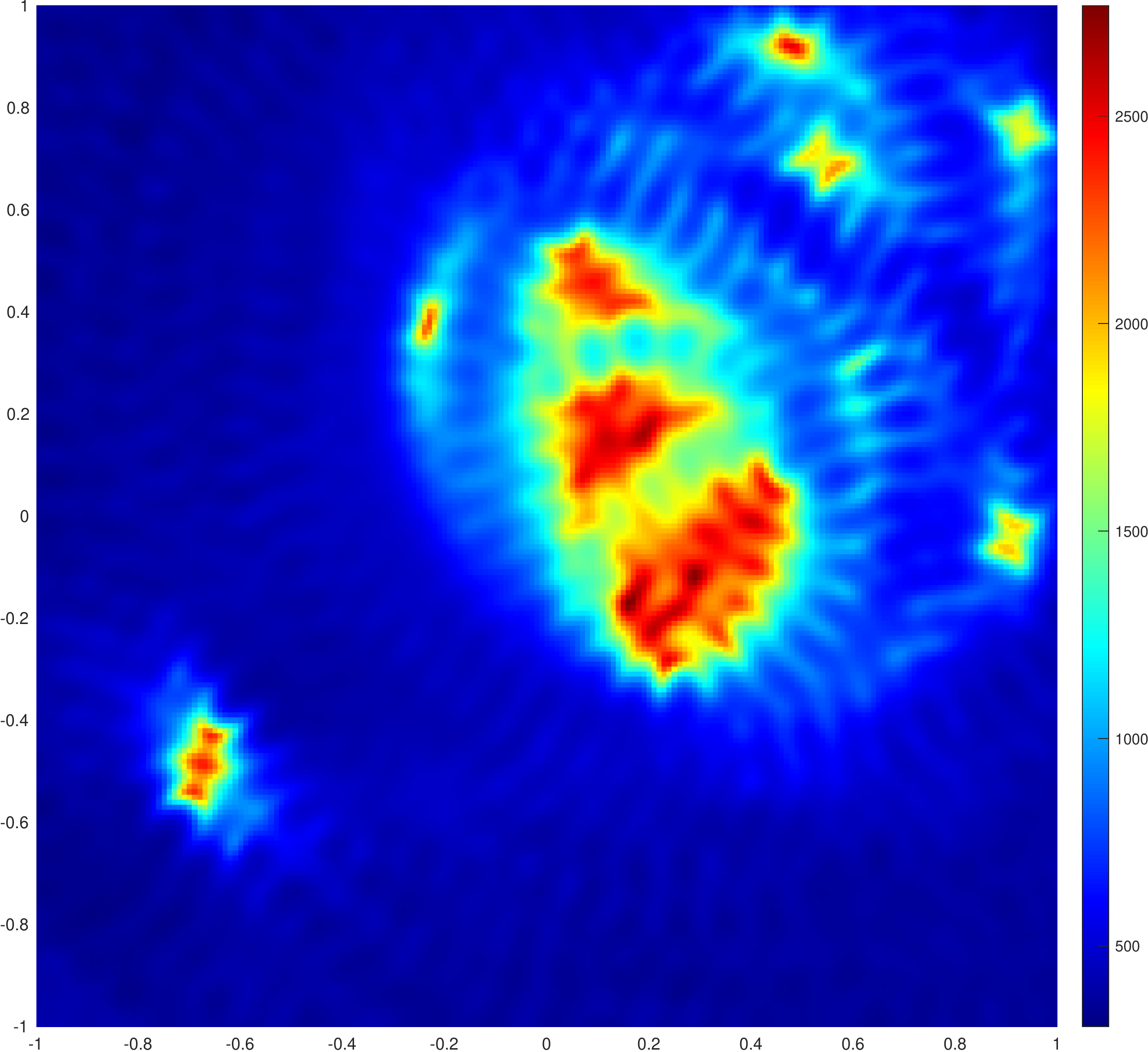}
\centering
\includegraphics[height=3.8cm]{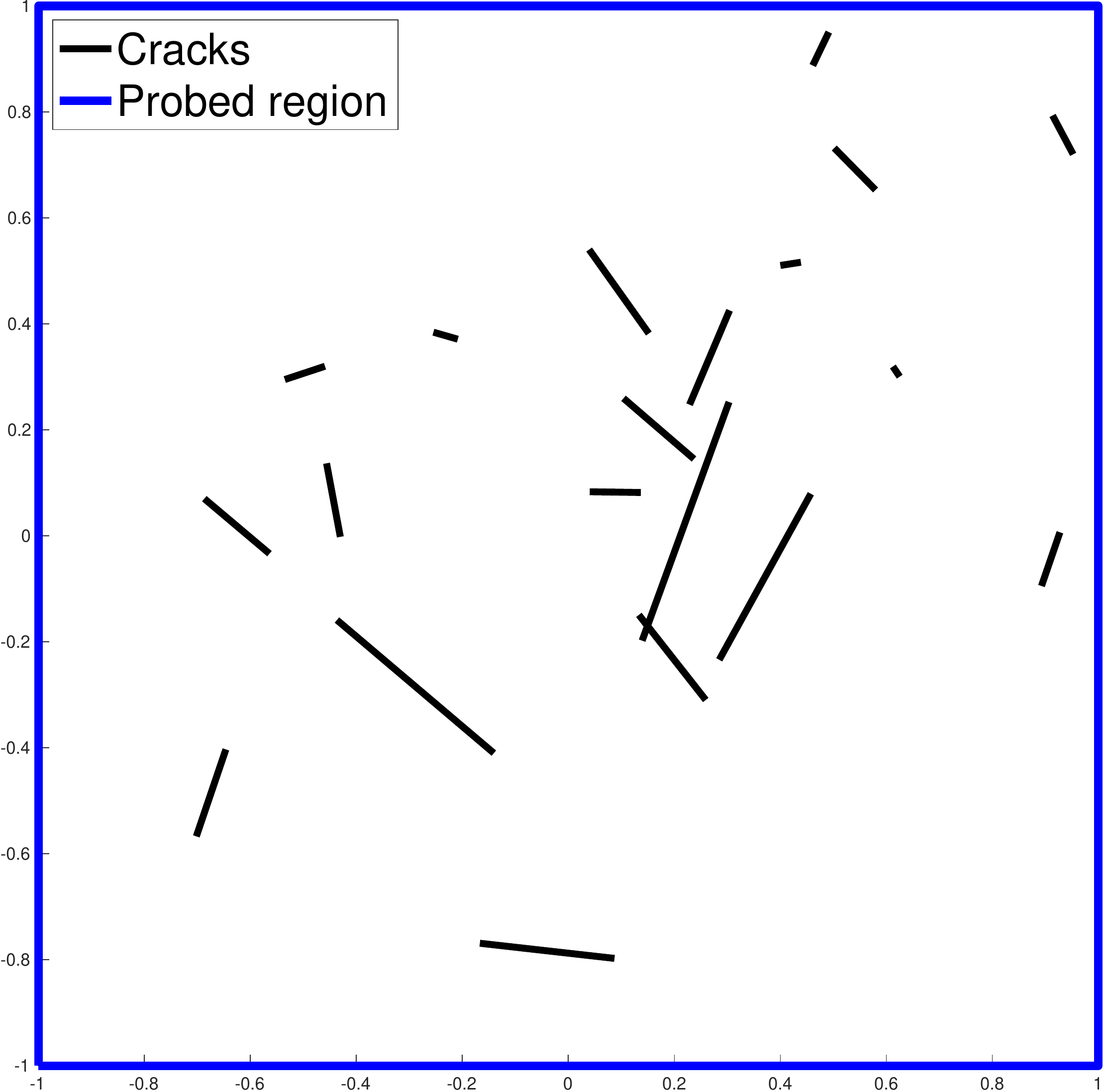}\  \includegraphics[height=3.8cm]{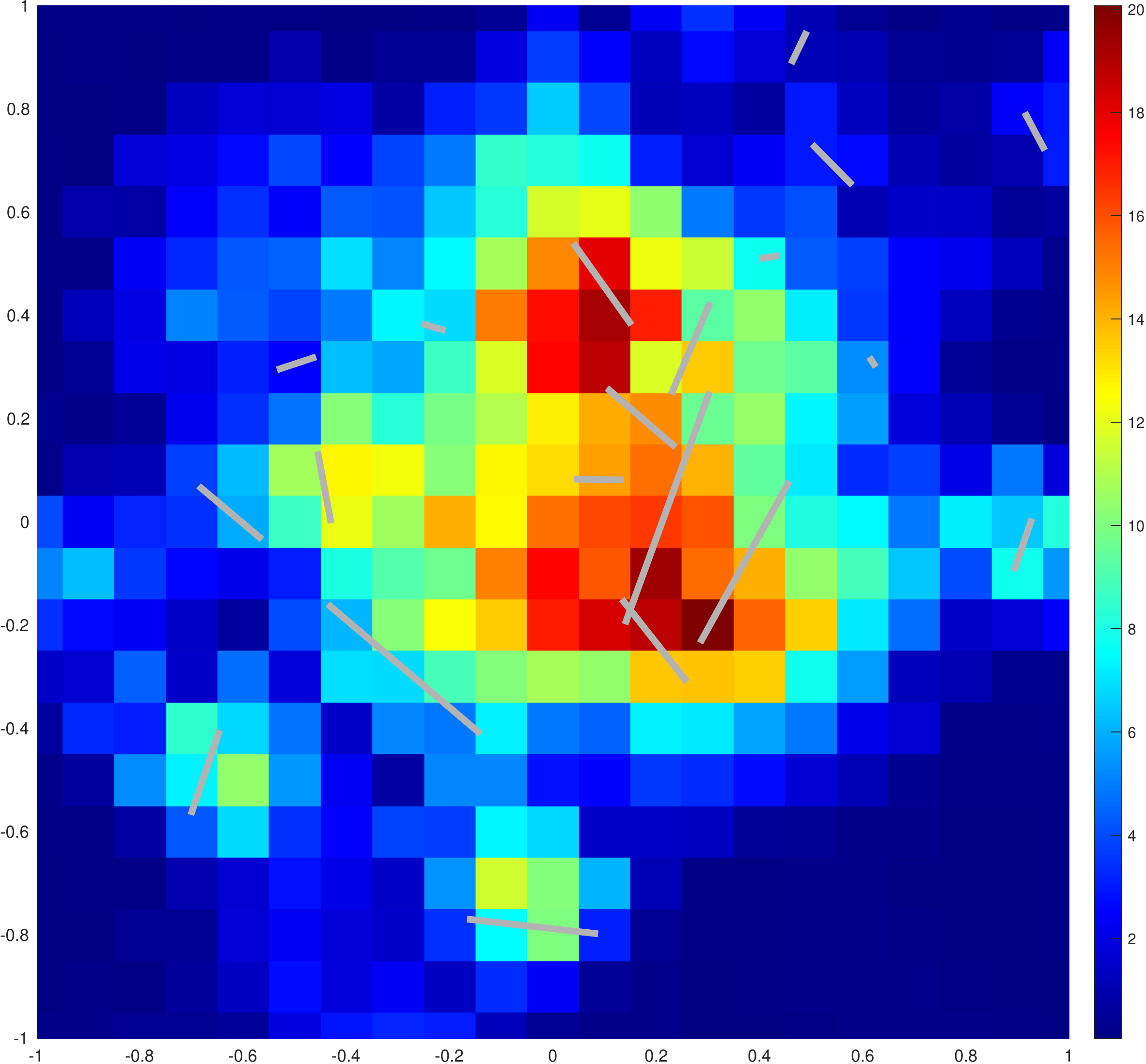}
\includegraphics[height=3.8cm]{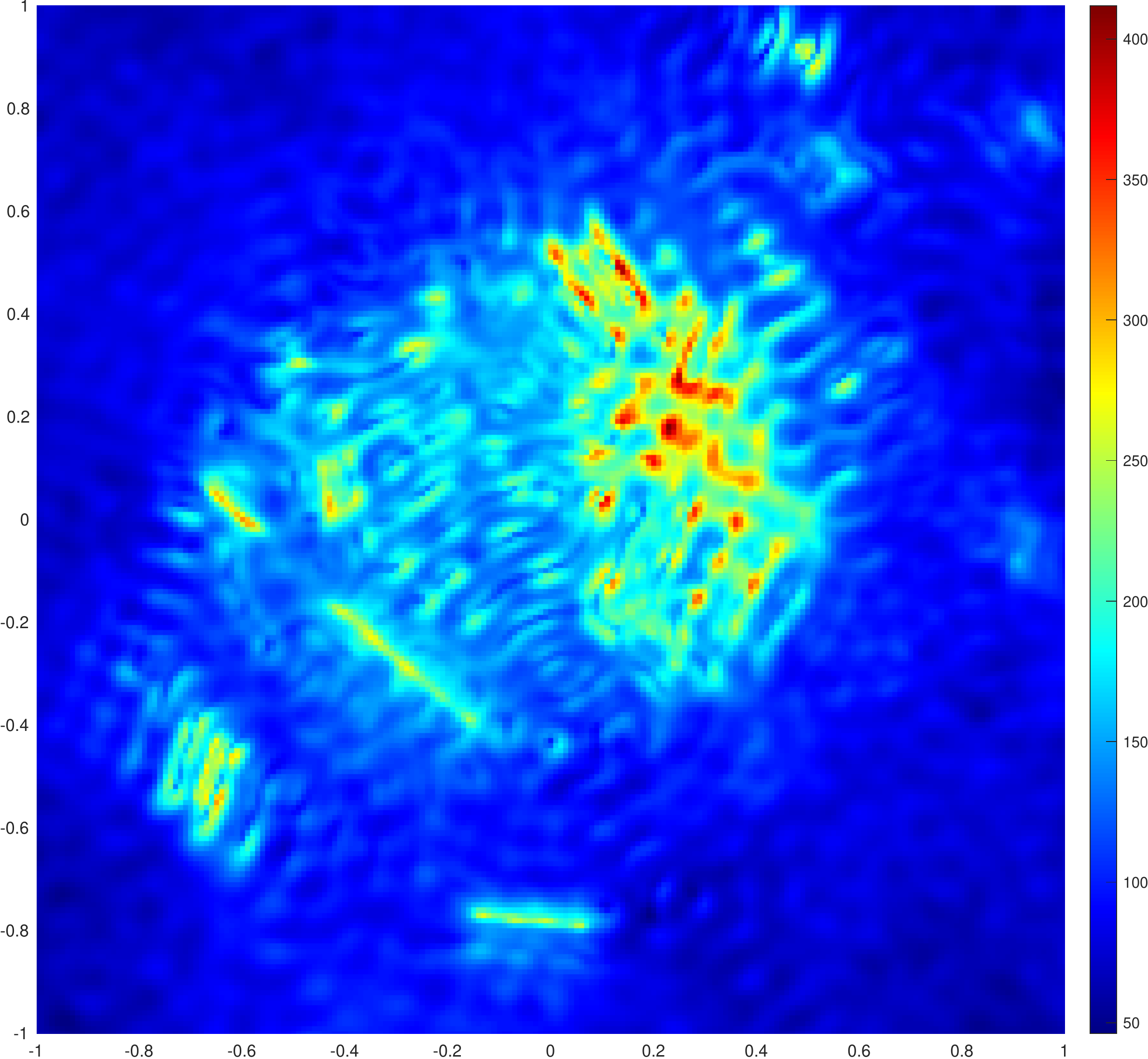}
\includegraphics[height=3.8cm]{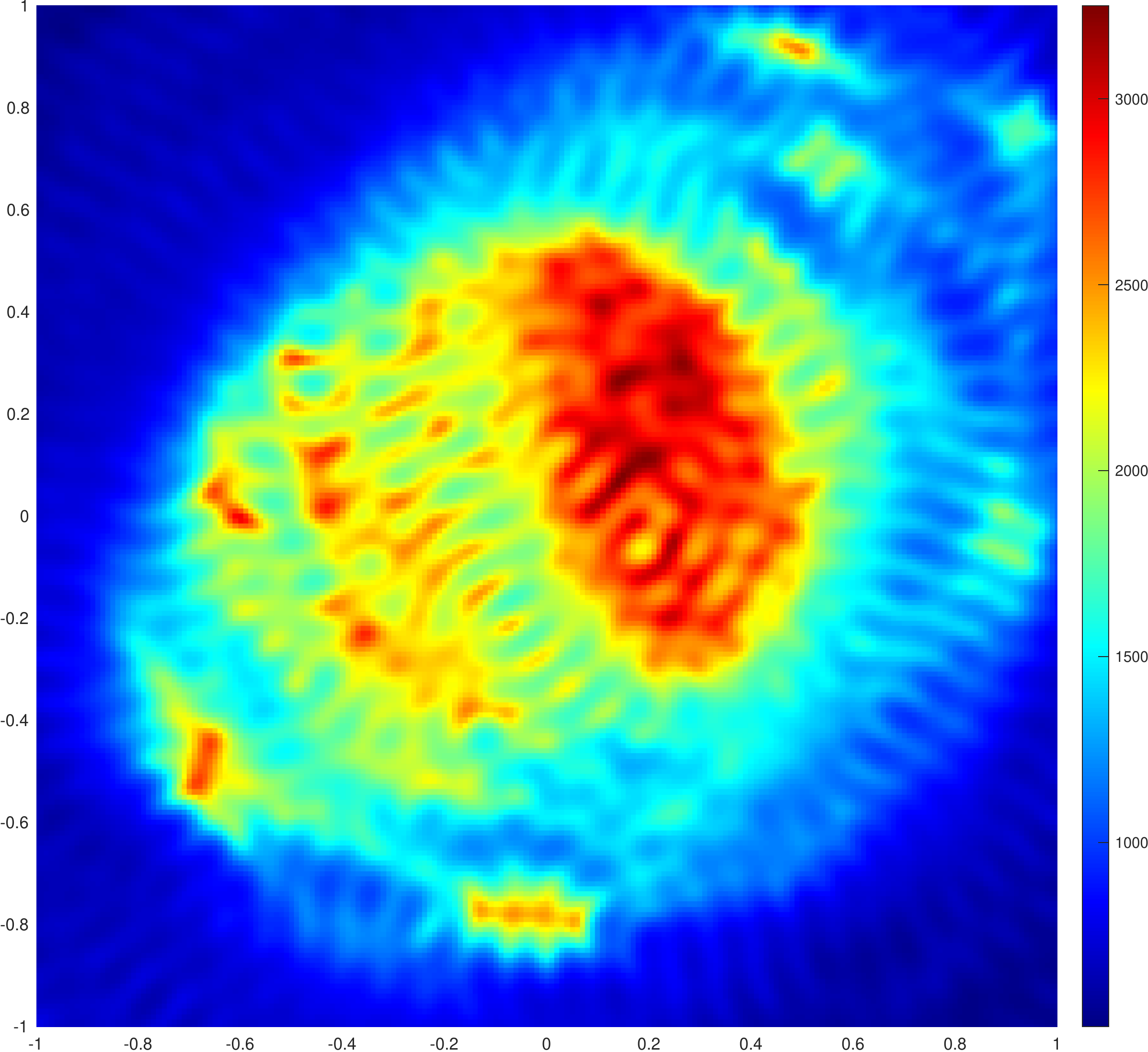}
\centering
\includegraphics[height=3.8cm]{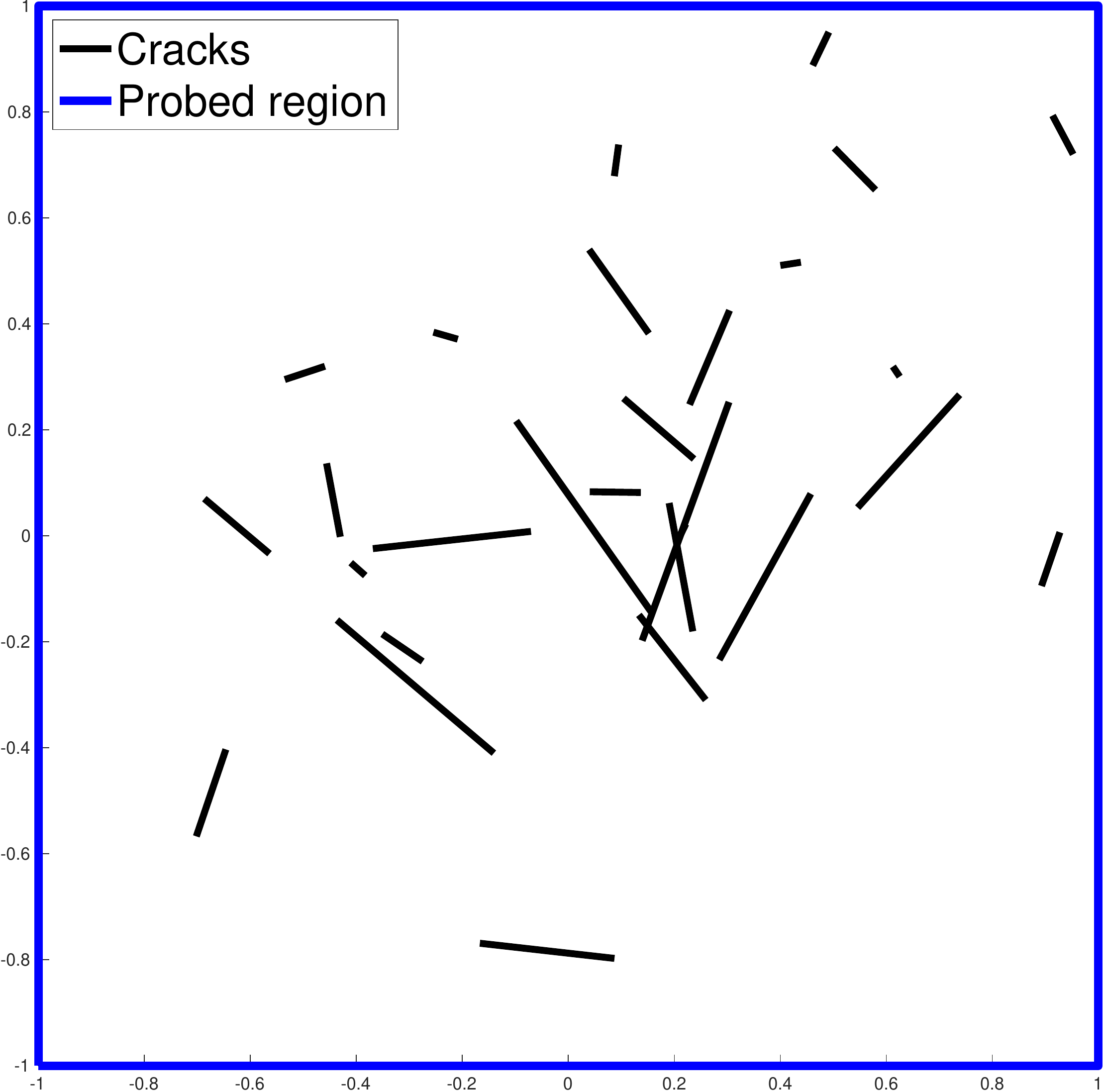}\ 
\includegraphics[height=3.8cm]{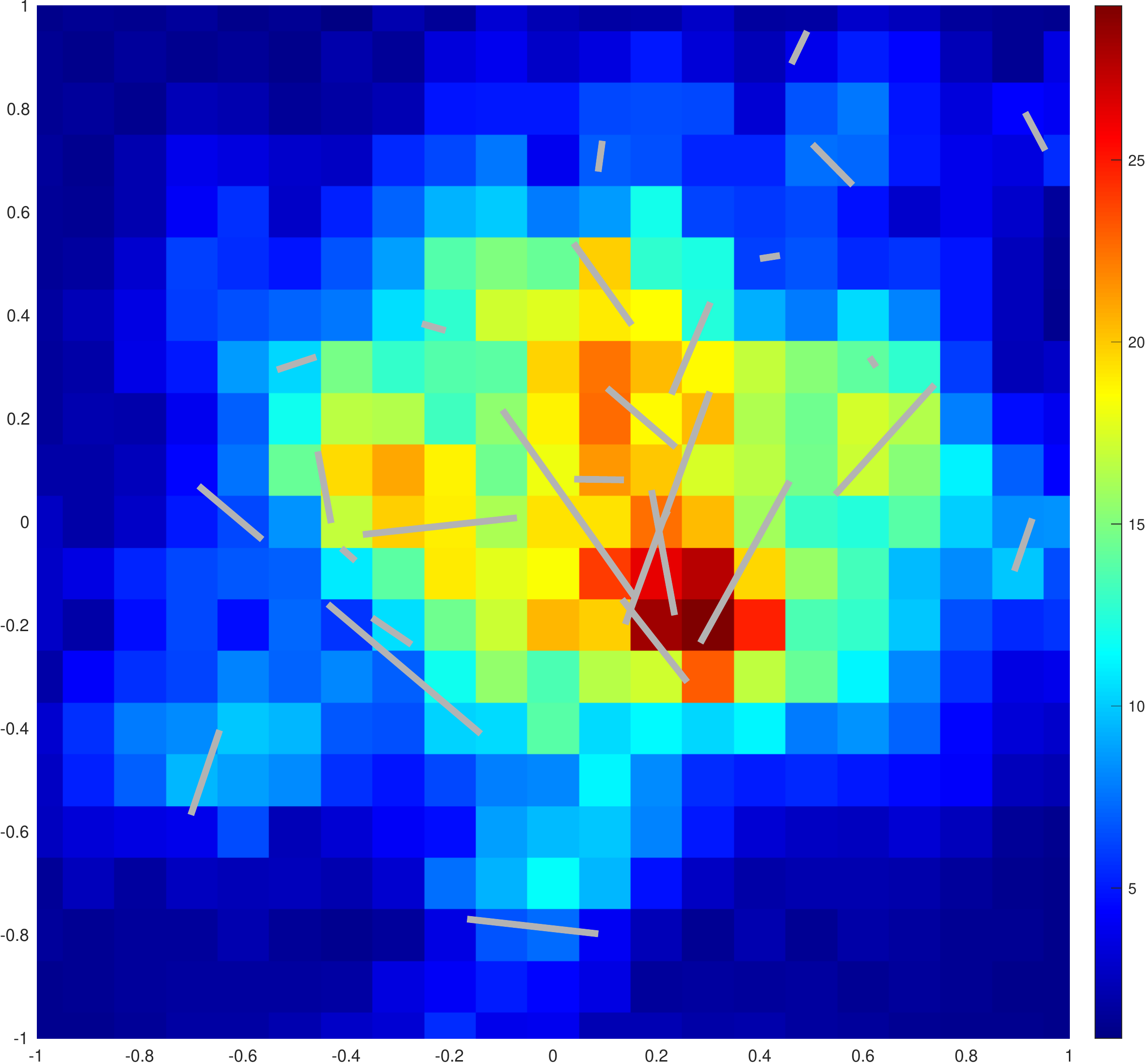}
\includegraphics[height=3.8cm]{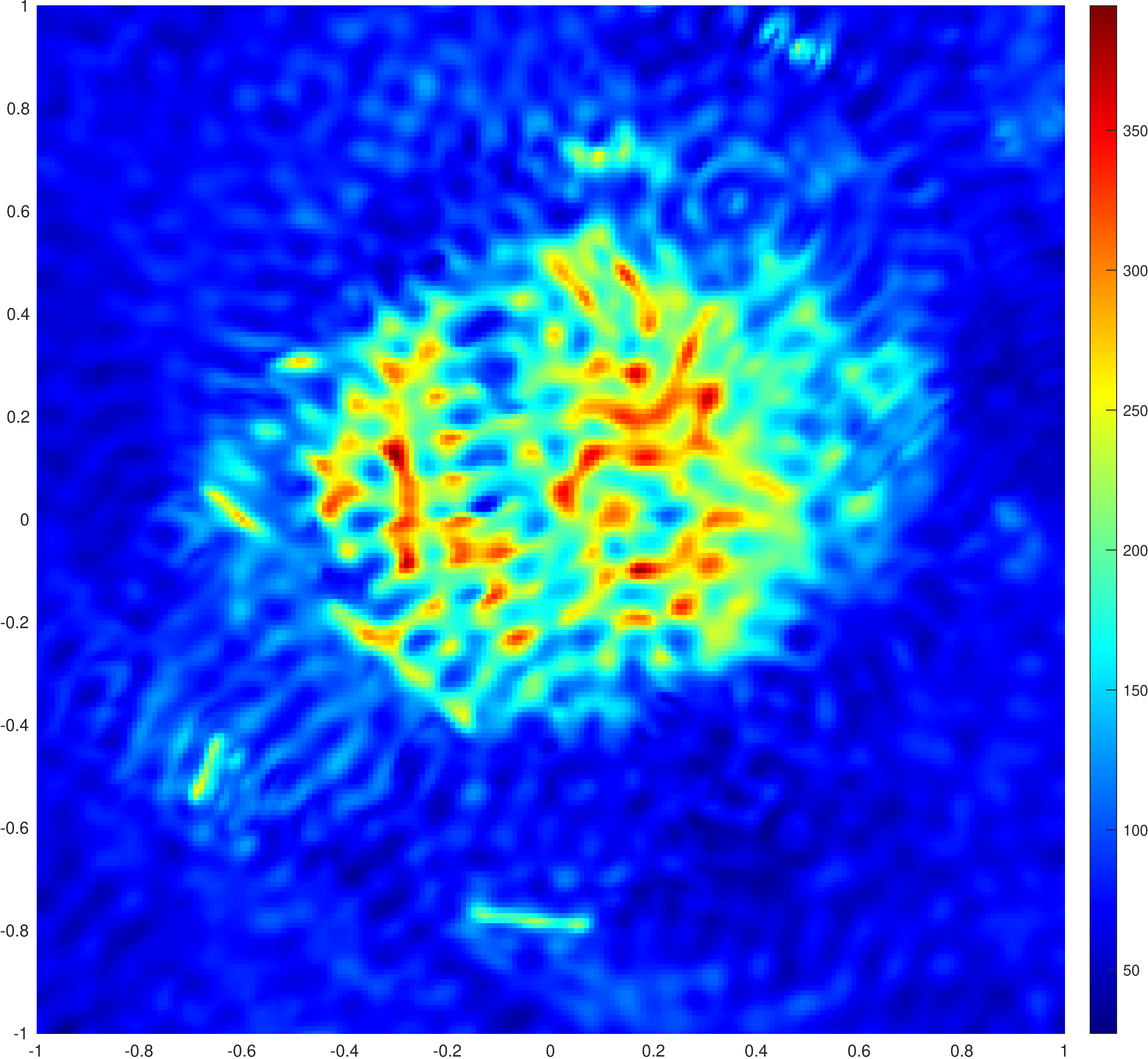}
\includegraphics[height=3.8cm]{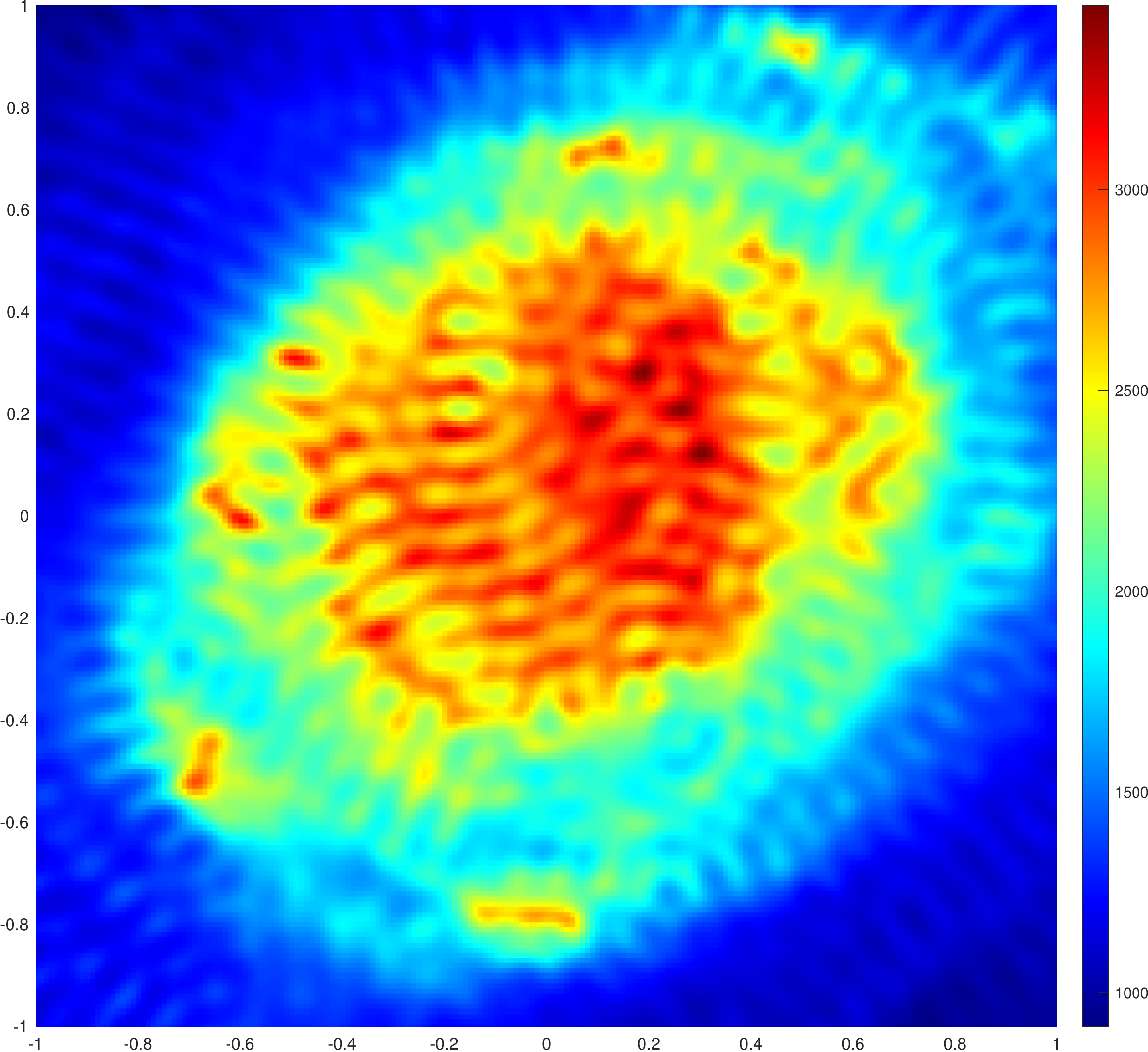}
\caption{From left to right column: exact configuration of the cracks, $\mathcal{I}^n$, $\mathcal{J}^n$ and FM indicator functions. From top to bottom rows, additional cracks are added to the crack network. The synthetic data is corrupted with 1\% random noise. }
\label{Fig:MonoDistrib7}
\end{figure}

\newpage
\section*{Appendix}
In this appendix, we establish a series of technical results which are used in the proofs above. 
\begin{proposition}\label{PropDenseRange}
The operator $H^\rel: \mL^2(\dsphere) \rightarrow \mY$ defined in (\ref{DefHrel}) has dense range.
\end{proposition}
\begin{proof}
Denote by $\tilde \mH^{1/2}(\Gamma\cap\Omega^c)$ the dual space of $\mH^{-1/2}(\Gamma\cap\Omega^c)$ that can be defined as the set of the restrictions to  $\Gamma\cap\Omega^c$ of functions in $\mH^{1/2}(\partial D)$ that are supported in $\Gamma\cap\Omega^c$. Consider $(h_1, h_2) \in  \mH^{1/2}(\partial\Om) \times \tilde \mH^{1/2}(\Gamma\cap\Omega^c)$ such that 
\begin{equation} \label{dif1}
\int_{\partial \Omega} h_1 \, \partial_\nu u_b\,ds + \int_{\Gamma\cap\Omega^c} h_2 \, \partial_\nu u_b\,ds =0
\end{equation}
for all $g\in\mL^2(\dsphere)$. Here $u_b$ is the solution of (\ref{PbChampTotalFreeSpaceComp}) with $u_i=v_g$ and the integrals have to be understood as duality products. Now introduce $w_h$ the unique solution to
\begin{equation}\label{PbAdjoint}
\begin{array}{|rcllrcll}
\Delta w_h+k^2\,w_h & = & 0 & \mbox{ in }\R^d\setminus(\Omegabar\cup{\Gammabar}) & \\[3pt]
w_h &=& h_1 & \mbox{ on } \d\Omega& \\[3pt]
\left[ w_h\right] &=& h_2  & \mbox{ on } \Gamma \cap\Omega^c
\\[3pt]
\left[ \partial_{\nu}w_h\right] &=& 0  & \mbox{ on } \Gamma \cap\Omega^c \\[3pt]
\multicolumn{4}{|l}{+\ \mbox{Radiation condition}}
\end{array}
\end{equation}
which is in $\mH^1(\mathcal{O})$ for all bounded domain $\mathcal{O}\subset\R^d\setminus(\Omegabar\cup{\Gammabar})$. 
Setting $u_{b,s}= u_b - v_g$, then we have, thanks to the radiation condition,
\begin{equation} \label{dif2}
\int_{\partial \Omega}    w_h \, \partial_\nu u_{b,s}  - \partial_\nu w_h  u_{b,s} \,ds + \int_{\Gamma\cap\Omega^c} [w_h] \, \partial_\nu u_{b,s}  ds =0.
\end{equation}
Computing the difference (\ref{dif1})-(\ref{dif2}) and using that $u_b=0$ on $\partial \Omega$ so that $u_{b,s}=-v_g$ on $\partial \Omega$, we get
\[
\int_{\partial \Omega}    w_h \, \partial_\nu v_g  - \partial_\nu w_h  v_g\,ds + \int_{\Gamma\cap\Omega^c} [w_h] \, \partial_\nu v_g  ds =0
\]
for all $g\in \mL^2(\dsphere)$. Using the definition \eqref{HerglotzWave} of $v_g$ and inverting the order of integration over $\dsphere \times \partial \Omega \cup \Gamma$ show that the far field of $w_h$ vanishes. This implies $w_h=0$ in $\R^d\setminus(\Omegabar\cup{\Gammabar})$ and from (\ref{PbAdjoint}), we infer that $h_1 = 0$ and $h_2 =0$. This is enough to conclude to the desired density result.
\end{proof}

\begin{proposition}\label{PropCompactness}
The operator $G^\rel: \mY\rightarrow\mL^2(\dsphere)$ defined in (\ref{DefHrel}) is compact.
\end{proposition}
\begin{proof}
For a given $\psi=(\psi_1,\psi_2)\in\mY$, denote $w$ the solution of (\ref{eq:Def_G+AB}) which is in $\mH^1(\mathcal{O}\setminus\overline{\Gamma})$ for all bounded domain $\mathcal{O}\subset\R^d$. Its far field pattern $w^{\infty}$ admits the representation (see e.g. \cite{Colton-Kress})
\begin{equation}\label{IdentityFarField}
w^{\infty}(\theta_s)=\int_{|x|=R}w\partial_r(e^{-ik\theta_s\cdot x})-\partial_rw\,e^{-ik\theta_s\cdot x}\,ds(x),
\end{equation}
where $R$ is such that $\Gammabar\subset B(O,R)$. On the other hand, classical results of interior regularity and the well-posedness of (\ref{eq:Def_G+AB}) guarantee that for all bounded domain $\mathcal{O}$ which does not meet $\partial\Om\cup\Gammabar$, we have
\begin{equation}\label{EstimReg}
\|w\|_{\mH^3(\mathcal{O})} \le C\,\|\psi\|_{\mY}.
\end{equation}
Using (\ref{EstimReg}) in (\ref{IdentityFarField}) allows one to conclude that the operator $G^\rel:\psi\mapsto w^{\infty}$ is compact from $\mY$ to $\mL^2(\dsphere)$. 
\end{proof}

\begin{proposition}\label{densityFrel}
Assume that $k^2$ is not a Relative Non Scattering Eigenvalue. Then the operator $F^\rel: \mL^2(\dsphere) \rightarrow \mL^2(\dsphere)$ defined in (\ref{def:RelFF}) has dense range.
\end{proposition}
\begin{proof}
We show that the assumption on $k^2$ implies that $(F^{\rel})^*$ is injective. First of all, the farfield patterns of problems \eqref{PbChampTotalFreeSpaceIntrod} and \eqref{PbChampTotalFreeSpaceComp} satisfy the reciprocity relation (see e.g. \cite{CCH})
\begin{equation}\label{prop:Reciprocity}
u_s^\infty(\theta_s,\theta_i)=u_s^\infty(-\theta_s,-\theta_i)\qquad\mbox{ and }\qquad u_{b,s}^\infty(\theta_s,\theta_i)  = u_{b,s}^\infty(-\theta_i,-\theta_s).
\end{equation}
As a consequence, this is also true for $w^\infty := u_s^\infty - u_{b,s}^\infty$. By definition of $F^{\rel}$, we have that 
\[
(F^{\rel}g)(\theta_s) = \int_{\dsphere} w^\infty(\theta_s,\theta_i)g(\theta_i)\,ds(\theta_i).
\]
Now let $h\in \mL^2(\dsphere)$ be such that $(F^{\rel})^*h=0$. Then we have 
\[
\int_{\dsphere}\overline{w^\infty(\hat{x},d)}h(\hat{x})\,ds(\hat{x})=0, \qquad \forall d\in \dsphere.
\]
The reciprocity relation implies 
\[
\int_{\dsphere}w^\infty(-d,-\hat{x})\overline{h(\hat{x})}\,ds(\hat{x})=0, \qquad \forall d\in \dsphere
\]
and the change of variables $\theta_s=-d$, $\theta_i=-\hat{x}$ leads to 
\[
\int_{\dsphere}w^\infty(\theta_s,\theta_i)\overline{h(-\theta_i)}\,ds(\theta_i)=0, \qquad \forall \theta_s\in \dsphere.
\]
Since $F^{\rel}$ is injective because of the assumption on $k^2$, we conclude that $h=0$ as desired.
\end{proof}

 \bibliographystyle{plain}
\bibliography{biblio}
 
 \end{document}